\documentclass{article}
\usepackage[dvipsnames,svgnames,table]{xcolor}
\usepackage[colorinlistoftodos,textsize=footnotesize,backgroundcolor=yellow!40,bordercolor=red,linecolor=red]{todonotes}
\usepackage{tocloft}
\usepackage{titlesec}
\usepackage{mathtools}
\usepackage[skins,theorems]{tcolorbox}
\usepackage{amsmath}
\usepackage{amsthm}
\usepackage{amssymb}
\usepackage{lipsum}
\usepackage[a4paper, margin=3cm]{geometry}
\usepackage{tikz}
\usepackage{titling}

\theoremstyle{plain}
\newtheorem{theorem} {Theorem} [section]

\newtheorem{lemma} [theorem] {Lemma}
\newtheorem{proposition} [theorem] {Proposition}
\newtheorem{corollary} [theorem] {Corollary}
\theoremstyle{definition}
\newtheorem{definition} [theorem] {Definition}

\newcommand{\C}{\mathbb{C}}
\newcommand{\N}{\mathbb{N}}

\newcommand{\R}{\mathbb{R}}

\newcommand{\Rn}{{\mathbb{R}^n}}

\newcommand{\forwhich}[0]{{ \ ; \ }}
\newcommand{\al}{\alpha}
\newcommand{\be}{\beta}

\newcommand{\ga}{\gamma}
\newcommand{\De}{\Delta}
\newcommand{\de}{\delta}

\newcommand{\epv}{\varepsilon}

\newcommand{\et}{\eta}

\newcommand{\Si}{\Sigma}

\newcommand{\ph}{\phi}

\newcommand{\na}{\nabla}

\DeclareMathOperator*{\curl}{curl}

\DeclareMathOperator*{\divergence}{div}
\DeclareMathOperator*{\supp}{supp}

\newcommand\norm[1]{\left\Arrowvert#1\right\Arrowvert}
\newcommand\ceil[1]{\left\lceil#1\right\rceil}
\newcommand\floor[1]{\left\lfloor#1\right\rfloor}
\newcommand\abs[1]{\left\vert#1\right\vert}

\newcommand\ra{\rightarrow}

\titleformat{\section}
  {\fontfamily{qcr}\Huge\bfseries\filright}{\thesection.}{4mm}{}
\titleformat{\subsection}
  {\fontfamily{qcr}\Large\bfseries\filright}{\thesubsection.}{4mm}{}
\titleformat{\subsubsection}
  {\fontfamily{qcr}\normalsize\bfseries\filright}{\thesubsubsection.}{4mm}{}
\pretitle{\center\huge\bfseries\ttfamily}
\posttitle{\center}
\preauthor{\center\large\texttt}
\postauthor{\center}
\predate{\center\large\itshape\ttfamily}
\postdate{\center}

\title{Morse Index Stability  for   Sequences of Sacks-Uhlenbeck Maps into a Sphere}
\author{Francesca Da Lio, Tristan Rivi\`ere and Dominik Schlagenhauf \thanks{Department of Mathematics, ETH Zentrum,
CH-8093 Z\"urich, Switzerland.}}
\date{ }

\begin{document}
\maketitle
\begin{abstract}{\fontfamily{cmtt}\selectfont\small
 
In this paper we consider sequences of  $p$-harmonic maps, $p>2$, from a closed Riemann surface $\Sigma$  into the $n$-dimensional  sphere $\mathbb{S}^n$ with uniform bounded energy. These are critical points of the energy
$$E_p(u) \coloneqq \int_\Sigma \left( 1+\abs{\nabla u}^2\right)^{p/2} \ dvol_\Sigma.$$
Our  two main results are   an improved pointwise estimate of the gradient in the neck regions around blow up points and the
proof  that the necks are asymptotically not contributing to the negativity of the second variation of the energy $E_p.$
This allows us, in the spirit of the   paper   of the first and second authors in collaboration with M. Gianocca \cite{DGR22},  to show the upper semicontinuity of the Morse index plus nullity for  sequences of $p$-harmonic maps into a sphere.}
\end{abstract}

\vspace{15mm}
{\fontfamily{cmtt}\selectfont\tableofcontents}

\medskip
\section{Introduction}
Let   $(\Sigma,h)$ be   a smooth closed Riemann surface and ${\mathcal{N}}^n\subset \R^m$ be an  at least $C^2$ $n$ dimensional closed manifold in $\R^m$. We consider
  the Dirichlet Energy in $2$-dimensions 
 \begin{equation}\label{DE}
E: W^{1,2}(\Sigma;{\mathcal{N}}^n) \ra \R; \qquad
E(u) \coloneqq \int_\Sigma \abs{\nabla u}^2\ dvol_\Sigma.
\end{equation}
 One of the main applications related to the study of \eqref{DE} is   the existence of critical points. These are called  {\em weak harmonic maps}. 
One of the main strategies in calculus of variations to search critical points of variational problems, which are not necessarily minimizers,  consists in the so-called min-max operations. There is a vast literature concerning different types of min-max operations according to the geometric context.   An important issue in min-max operations is the control of the Morse Index.  This permits   in some particular situations to deduce the existence of infinite number of critical points   (see, e.g. \cite{CM}).   
  It is well known that the energy functional $E$ is  conformal invariant, possesses a non-compact invariance group and represents the  limiting case where the Palais-Smale condition fails. In order to overcome this difficulty one considers suitable relaxations of the energy $E$ satisfying the Palais-Smale condition.\par
  J. Sacks and K. Uhlenbeck in the pioneering  paper \cite{SaU} considered the following regularization of the Dirichlet Energy
\begin{equation}\label{Ep}
E_p: W^{1,p}(\Sigma;{\mathcal{N}}^n) \ra \R; \qquad
E_p(u) \coloneqq \int_\Sigma \left( 1+\abs{\nabla u}^2\right)^{p/2} \ dvol_\Sigma,
\end{equation}
 where $p>2.$ We observe that the energy \eqref{Ep} is sub-critical, in the sense that every sequence of critical points with uniformly bounded energy is pre-compact. Moreover the Sobolev Space $W^{1,p}(\Sigma)$ embeds into $C^0(\Sigma)$.  Hence J. Sacks and K. Uhlenbeck could produce a sequence of smooth minima  of  \eqref{Ep} belonging  for every $p>2$ to a fixed {\em free homotopy class},   which {\bf bubble tree converges} towards $(u_\infty, v_\infty^1\ldots v_\infty^Q)$. As an application of their blow-up analysis, they were able to show the existence of a minimal
 two-sphere (modulo the  action of the first fundamental group $\Pi_1$)  under suitable topological conditions of the target manifold. \par
 It turns out that the energy \eqref{Ep} also satisfies Palais-Smale condition. This permitted later  T. Lamm  \cite{Lamm} to extend the Sacks - Uhlenbeck procedure to sequences of min-max critical points. Lamm  additionally  showed  the {\em energy identity}
in the blow-up process, under an ad-hoc entropy condition, which comes from ``Struwe's monotonicity trick'' in the case of min-max critical points (see e.g \cite{Struwe}).\par
 In a recent paper of the first author with M. Gianocca \cite{DG},  a different perturbation of \eqref{DE} has been considered,
 given by the Ginzburg-Landau energy:
\begin{equation}\label{Eeps}
 {E}_\varepsilon(u)= \int_{\Sigma}e_{\varepsilon}(u)\,dvol_h,
\end{equation}
where
 \begin{equation}\label{energydensintro}
 e_{\varepsilon}(u)=\frac{1}{2}|d u|_h^2+\frac{1}{4\varepsilon^2}(1-|u|^2)^2\,.
 \end{equation}

The energy $ {E}_\varepsilon$ is defined on the Hilbert space $W^{1,2}(\Sigma,\mathbb{R}^{n})$ and can also  be used to approximate harmonic maps $u:\Sigma\to \mathbb{S}^n$, where $\mathbb{S}^n$ is $n$-dimensional sphere (this has been done by   Karpukhin and  Stern, \cite{KaSt}). The analysis of the asymptotic behavior of critical points of $\eqref{Eeps}$ defined on a closed Riemannian manifold of dimensions $m\ge 2$ has been originally performed by Lin and Wang in \cite{LinWa2}.\par
Coming back to  Sacks-Uhlenbeck energy \eqref{Ep}, one can produce critical points whose extended index (Morse Index plus Nullity) can be controlled by  the number of parameters involved in the  min-max procedure, (see, e.g. 
Theorem 10.2 in Ghoussoub \cite{G}).\par One of the main problems  related to the applications is  the analysis of the asymptotic behavior of the extended index of a sequence of critical points  to $\eqref{Ep}$.
\par
The bound from above of the index of a harmonic map   by the index of approximating $p$-harmonic maps is quite standard. Indeed once we have the quantization of the energy,   one can easily  produce, by using a standard $\log$-cutoff method (see e.g.  \cite{KaSt}  and \cite{Riv2} ), from a negative direction of  the second derivative of the Dirichlet Energy at the {\em bubble tree limit $(u_\infty, v_\infty^1\ldots v_\infty^Q)$}, a sequence of negative variations for approximating sequences of $p$-harmonic maps.\par
Precisely for $p\to 2^+$ we have
\begin{equation}\label{lowersc}
\mathrm{Ind}_{ E_{p}}(u_{p})\geq\mathrm{Ind}_{{E}}(u_\infty)+\sum_{i=1}^Q\text{Ind}_{E}(v^i_{\infty}).
\end{equation}
 
For the readers convenience, we give a simple proof of the inequality \eqref{lowersc}, in the setting of the present paper, in the Appendix \ref{SUBSECTION: A Lemma on Capacity and the Lower Semicontinuity of the Morse Index}. \par
Compared to the lower semi-continuity of the Morse index (see e.g. \cite{CM} for minimal surfaces), the upper semi-continuity is 
typically much more delicate due to the need of obtaining precise estimates of the sequence of
solutions in regions of loss of compactness and in fact it does not always hold.

In the joint  paper  \cite{DGR22} of the first  and second authors   with  Gianocca, a new method  have been developed  to prove     the upper-semi-continuity of the Morse index and the nullity associated to conformally invariant variational problems in $2$-D, which include the case of harmonic maps. This new theory have turned to be very efficient in several other geometrical analysis settings (see the recent results in the context of biharmonic maps \cite{Mi, MiRi3}, of Willmore surfaces \cite{MiRi2}, of Yang-Mills connections \cite{GL, GLR}, of Ginzburg-Landau energies  \cite{DG} 
and of constant mean curvature surfaces \cite{Work}).\par
In the present work, we consider the case where the target manifold is the $n$-sphere. One of our  main results   is   the upper semicontinuity of the extended Morse index for sequences of critical points of $ E_{p}$ 
  \begin{theorem}\label{morseindextheoremfordirichletintr}
Let $u_{k}:(\Sigma,h)\to\mathbb{S}^n$ be a sequence of critical points of $ E_{p_k}$ with uniformly bounded energy $ E_{p_k}(u_k)\leq \Lambda$ and assume that it converges in the bubble tree sense to $(u_\infty, v_\infty^1\ldots v_\infty^Q)$. Then for $k$ large enough,
\begin{equation}
\mathrm{Ind}_{ E_{p_k}}(u_k)+\mathrm{Null}_{ E_{p_k}}(u_k)\leq\mathrm{Ind}_{ E}(u_\infty)+\mathrm{Null}_{ E}(u_\infty)+\sum_{i=1}^Q\text{Ind}_{E}(v^i_{\infty})+\sum_{i=1}^Q\text{Null}_{ E}(v^i_\infty)\,.
\end{equation}
\end{theorem}
As in the   paper \cite{DGR22}, the proof of Theorem \ref{morseindextheoremfordirichletintr} strongly  relies on a precise pointwise gradient estimate  of sequences  of $p$-harmonic maps on degenerating annuli domains $A(\eta,\delta_k)=B(0,\eta)\setminus \bar{B}(0,\delta_k/\eta)$ where $\delta_k\to 0$ as $k\to +\infty$:
\begin{theorem}
\label{pointwiseestimate}
Let $u_{k}\in W^{1,p}(\Sigma;{\mathbb{S}}^n)$  be a sequence of $p_k$ harmonic maps . Then, for any given $\be \in \left(0,\log_2 (3/2) \right)$,   for $k \in \N$ large and $\eta>0$ small, it holds:
\begin{equation}
\label{pe1}
\begin{gathered}
\forall x \in A(\et,\de_k): \qquad
\abs{x}^2 \abs{\nabla u_{k}(x)}^2 \le \left[ \left( \frac{\abs{x}}{\et} \right)^\be + \left( \frac{\de_k}{\et \abs{x}} \right)^\be \right]\boldsymbol{\epsilon}_{\eta, \delta_k} +  \boldsymbol{\mathrm c}_{\eta, \delta_k},
\end{gathered}
\end{equation}
where
\begin{equation}\label{pe2}
\lim_{\eta\searrow0}\limsup_{k\rightarrow\infty}\boldsymbol{\epsilon}_{\eta, \delta_k}=0,
\qquad \text{and} \qquad
\lim_{\eta\searrow0}\limsup_{k\rightarrow\infty}\boldsymbol{\mathrm c}_{\eta, \delta_k} \log^2\left(\frac{\eta^2}{\delta_k}\right) =0. ~~~\square
\end{equation}
\end{theorem}

\bigskip
It is a striking fact that the
  estimate \eqref{pe1} together with \eqref{pe2}  do not  hold for any   sequence of $p$-harmonic maps taking values into  a general target ${\mathcal{N}}^n$, as we  clarify later.\par
  Going back to the Dirichlet Energy, 
in  the    paper  \cite{DGR22} the authors could obtain  a   pointwise gradient estimate of   critical points to conformally invariant Lagrangians as a by-product of   an energy quantization in the Lorentz space $L^{2,1}$ (the pre-dual of the Marcinkiewicz  $L^{2,\infty}$ space).\footnote{We recall that for a given domain $\mathbb{D}$ $$
{  L^{2,\infty}(\mathbb{D})=\{f ~\mbox{measurable}:~~\sup_{\lambda\ge 0}\lambda |\{x\in \mathbb{D}~: |f(x)|\ge\lambda\}|^{1/2}<+\infty\}}
$$. 
and
$${L^{2,1}(\mathbb{D}) =\{f ~\mbox{measurable}:~~2\int_{0}^{+\infty}|\{x\in\mathbb{D}~: |f(x)|\ge\lambda\}|^{1/2} d\lambda<+\infty \}.} $$}

   We recall that the energy quantization  asserts that for a sequence of critical points of a given  Lagrangian of bounded energy, that bubble converges to a limiting critical point, the energy is preserved in the limit, i.e., the limiting energy is equal to the sum of the macroscopic maps and the bubbles. The neck regions are annular  regions that link the macroscopic map to its bubbles.   The relevance of the use interpolation Lorentz spaces  has been originally observed in the work by Lin and Rivi\`ere \cite{LiRi02}  in the study of bubbling phenomena of sequences of harmonic maps, and then generalized by Laurain and Rivi\`ere \cite{LR14} to sequences of critical points to conformally invariant variational problems. 
The $L^{2,1}$ quantization of   sequences of  harmonic maps (established first by  Laurain and Rivi\`ere in \cite{LR14}) is  crucial to get the so-called the {\it $C^0$-no neck property} or {\em necklessness property},  which says, roughly speaking,    that there is no distance between the bubbles $v^j_\infty$ and the main part $u_\infty$, namely,  the bubbles $v_\infty^j$ are ``$L^\infty$-glued'' directly to the weak limit $u_\infty$.   As it has been showed in \cite{LiWa-1} and \cite{LiWa-2} the energy identity and the necklessness property do not hold for general $p$-harmonic maps and one needs to add some additional  conditions linking the parameter $p$ and the degenerating conformal class of the neck regions.
In the case of $p $-harmonic maps into spheres  nevertheless the situation is improving and the necklessness property holds, as it has been shown in  \cite{LZ19},   where the combination of  Wente estimates   with   Lorentz spaces plays a crucial role. We mention that the combination of conservation laws and the interpolation Lorentz spaces   has been introduced twenty years before in \cite{LiRi02} and systematized in \cite{LR14}. This strategy has proven to be very versatile, in the sense that, since the work \cite{LiRi02},  it has been applied in several other geometrical problems (Willmore surfaces, bi-harmonic maps, fractional harmonic maps, Yang-Mills connections.. ).  
\par
In the current paper we   obtain  the $L^{2,1}$ quantization, and therefore the the necklessness property, for $p$-harmonic maps into the  $n$ dimensional sphere and   show the estimate
\eqref{pe1}, by taking over the approach in the paper by Bernard and Rivi\`ere  \cite{BR19}, that was in turns inspired by the commutator estimates in \cite{RW83}. The $L^{2,1}$ quantization   is of independent interest, since in the sphere case it  is not actually necessary for the proof of the Theorem \ref{morseindextheoremfordirichletintr}.\par

  The extension of our results to $p$-harmonic maps into a general manifold and in particular the necklessness property, would need
the following {\em entropy type condition}
\begin{equation}
\label{II.r7-ann}
\limsup_{p_k\rightarrow 2} \sqrt{p_k-2}\,\log\delta_k^{-1}=0.
\end{equation}
 
 The condition \eqref{II.r7-ann} has been   found  in \cite{LiWa-1} for at least $C^3$ manifolds, and then in  \cite{DR23}, by different methods, in the case of at least $C^2$   closed manifolds and for  a much wider case of PDE's including $p $-harmonic relaxation of arbitrary conformally invariant Lagrangians of maps.\par 
We recall that $\log\delta_k^{-1}$ is nothing but the conformal class of the neck domain and represents a crucial quantity.\footnote{We recall that the modulus or the conformal class  of an annulus $\{z\in\C:~~a<|z-z_0|< b\}$ with inner radius $a$ and outer radius $b$ is defined (up to a constant) to be $\log\left(\frac{b}{a}\right).$}
The condition \eqref{II.r7-ann} is stronger than 
\begin{equation}
\label{II.r7-annbis}
\limsup_{p_k\rightarrow 2}{(p_k-2)}\,\log\delta_k^{-1}=0
\end{equation}
which comes from {\em Struwe's monotonicity trick} \cite{Struwe} and which has been used by Lamm \cite{Lamm} to show the energy identity of $p$ harmonic maps.\par
We would like to conclude this section by mentioning that recently Gauvrit, Laurain and Rivi\`ere \cite{GLR} have showed the upper semicontinuity of the extended index in the case of Sack-Uhlenbeck relaxation of Yang-Mills  connections in general $4$-D domains without any type of entropy   condition. Therefore the  of $p$-harmonic case into a sphere can be seen as the $2$ dimensional version of Yang-Mills  connections in $4$ dimensions.\par

\section{Preliminary definition and results  }
\label{SECTION: Preliminary definition and results}

Let $(\Sigma,h)$ be a smooth closed Riemann surface and let $S^n\subset \R^n$ be the Euclidean unit sphere of dimension $n\in \N$.
 For $p \geq 2$ we define the $p-$energy as
\begin{equation}
E_p: W^{1,p}(\Sigma;S^n) \ra \R; \qquad
E_p(u) \coloneqq \int_\Sigma \left( 1+\abs{\nabla u}^2\right)^{p/2} \ dvol_\Sigma.
\end{equation}

\begin{definition}
[$p-$Harmonic Map]
We say that a function $u \in W^{1,p}(\Sigma;S^n)$ is a $p$-harmonic map if it is a critical point of $E_p$ with respect to inner variations. In that case $u$ satisfies the Euler-Lagrage equation
\begin{equation}
\label{EQ: Euler-Lagrange equations in div form for p harm}
-\divergence\left( \left( 1+\abs{\nabla u}\right)^{\frac{p}{2} -1} \nabla u \right)= \left( 1+ \abs{\nabla u}^2 \right)^{\frac{p}{2}-1} u \abs{\nabla u}^2 \in \R^{n+1},
\end{equation}
or in non-divergence form
\begin{equation}
\label{EQ: Euler-Lagrange equations in non-div form for p harm}
\De u +\left(\frac{p}{2}-1\right)\frac{\langle \na^2 u, \nabla u \rangle \nabla u}{1+\abs{\nabla u}^2}+ u\abs{\nabla u}^2 =0 \in \R^{n+1}.
\end{equation}
\end{definition}

\begin{lemma}
[$\varepsilon$-regularity]
\label{LEMMA: epsilon reg sacks uhlenbeck for p harm}
There exists an $\varepsilon > 0$, a constant $C>0$ and some $p_0>2$ such that for any $p$-harmonic map $u \in W^{1,p}(\Sigma;S^n)$ with $p\in [2,p_0)$ and any geodesic ball $B_r\subset \Sigma$    if
\begin{equation}
\int_{B_r} \abs{\nabla u}^2 dx \le \varepsilon,
\end{equation}
then
\begin{equation}
\norm{\nabla u_k}_{L^\infty(B_{r/2})}
\le \frac{C}{r} \norm{\nabla u_k}_{L^2(B_{r})}.~~~\square
\end{equation}
\end{lemma}
For a proof see in \cite{SU81} Chapter 3, Main Estimate 3.2 and Lemma 3.4.


Let
\begin{equation}
\pi: \R^{n+1} \backslash \{0\} \ra S^n, \qquad
\pi(p)=p/|p|,
\end{equation}
be the nearest point projection into $S^n$.

\begin{proposition}[Second variation of $p$-energy]\label{secondvariation}
Let $p\geq 2$, let $u \in W^{1,p}(\Sigma;S^n)$ be a $p$-harmonic map and let $\phi \in C^{\infty}(\Sigma;\R^{n+1})$.
We introduce the notations $u_t = \pi(u+t\phi)$, $w_t= \frac{du_t}{dt}$ and $w=w_0$.
Then 
\begin{equation}
\begin{aligned}
\left(\frac{d^2}{dt^2} E_p(u_t) \right)\Bigg|_{t=0}
&=p\ (p-2) \int_\Sigma\left( 1+\abs{\nabla u}^2\right)^{p/2-2} \left(\nabla u \cdot \nabla w \right)^2  \ dvol_\Si. \\
&\hspace{10mm}+p \int_\Sigma\left( 1+\abs{\nabla u}^2\right)^{p/2-1} \left[\abs{\nabla w}^2 - \abs{\nabla u}^2 \abs{w}^2 \right] \ dvol_\Sigma.
\end{aligned}
\end{equation}
\end{proposition}

\begin{proof}
We start by computing with the chain rule
\begin{equation}
\frac{d}{dt} E_p(u_t) = p \int_\Si \left( 1+\abs{\na u_t}^2\right)^{p/2-1} \na u_t \cdot\na w_t \ dvol_\Si.
\end{equation}
and again with the chain rule and the product rule
\begin{equation}
\begin{aligned}
\frac{d^2}{dt^2} E_p(u_t) 
&= p\ (p-2) \int_\Si \left( 1+\abs{\na u_t}^2\right)^{p/2-2} \left(\na u_t \cdot \na w_t \right)^2 \ dvol_\Si \\
&\hspace{8mm} + p\int_\Si \left( 1+\abs{\na u_t}^2\right)^{p/2-1} \left( \abs{\na w_t}^2 + \na u_t \cdot \na \left( \frac{dw_t}{dt} \right) \right) \ dvol_\Si.
\end{aligned}
\end{equation}
Evaluating at $t=0$ one has
\begin{equation}
\begin{aligned}
\left(\frac{d^2}{dt^2} E_p(u_t) \right)\Bigg|_{t=0}
&= p\ (p-2) \int_\Si \left( 1+\abs{\na u}^2\right)^{p/2-2} \left(\na u \cdot \na w \right)^2 \ dvol_\Si \\
&\hspace{8mm} + p\int_\Si \left( 1+\abs{\na u}^2\right)^{p/2-1} \left( \abs{\na w}^2 + \na u \cdot  \na \left( \frac{dw_t}{dt} \right)\bigg|_{t=0} \right) \ dvol_\Si.
\end{aligned}
\end{equation}
Integrating by parts and using \eqref{EQ: Euler-Lagrange equations in div form for p harm} we find 
\begin{equation}\label{EQ: ibniuUBiZBi3813HB1wA}
\begin{aligned}
\left(\frac{d^2}{dt^2} E_p(u_t) \right)\Bigg|_{t=0}
&= p\ (p-2) \int_\Si \left( 1+\abs{\na u}^2\right)^{p/2-2} \left(\na u \cdot \na w \right)^2 \ dvol_\Si \\
&\hspace{8mm} + p\int_\Si \left( 1+\abs{\na u}^2\right)^{p/2-1} \left( \abs{\na w}^2 + \abs{\na u}^2 u \cdot \left( \frac{dw_t}{dt} \right)\bigg|_{t=0} \right) \ dvol_\Si.
\end{aligned}
\end{equation}
Rewrite
\begin{equation}
\label{EQ: Expresion of wt wjrnrf371g1}
w_t = \frac{du_t}{dt} = \frac{d}{dt} \pi(u+t\ph) = D\pi_{u+t\ph} \ \ph.
\end{equation}
Note that $D\pi_p \in \R^{n \times n}$ is a matrix and in formulas is given by
\begin{equation}
\label{EQ:  Computation of orth proj 9ft6fn1nf8}
D\pi_p [v] 
= \frac{v}{\abs{p}} -\langle p, v\rangle \frac{p}{\abs{p}^3} 
= \frac{1}{\abs{p}} \Big( v -\langle \pi(p), v\rangle \pi(p) \Big)
= \frac{1}{\abs{p}} D\pi_{\pi(p)} [v] 
, \qquad \forall v \in \Rn. 
\end{equation}
Combining \eqref{EQ: Expresion of wt wjrnrf371g1} and \eqref{EQ:  Computation of orth proj 9ft6fn1nf8} together with the fact that $D\pi_p=D\pi_{\pi(p)}$ we obtain
\begin{equation}
w_t =  D\pi_{u+t\ph} \ \ph = \frac{1}{\abs{u+t\phi}}  D\pi_{u_t} \ \ph = \frac{1}{\abs{u+t\phi}} \Big( \ph- \langle u_t, \ph \rangle u_t \Big).
\end{equation}
Now by the chain rule
\begin{equation}
\label{EQ: dwt expression one 3u2re8bf1}
\frac{dw_t}{dt} = \frac{1}{\abs{u+t\phi}}\Big( -\langle w_t,\ph\rangle u_t -\langle u_t,\ph \rangle w_t\Big) + \left(\frac{d}{dt} \frac{1}{\abs{u+t\phi}}  \right) \Big( \ph- \langle u_t, \ph \rangle u_t \Big).
\end{equation}
Evaluating
\begin{equation}
\begin{aligned}
u \cdot \left( \frac{d w_t}{dt} \right)\bigg|_{t=0}
&= -u \cdot \left(\langle w,\phi \rangle u +\langle u,\ph \rangle w \right) \\
&= - \Big( \underbrace{\abs{u}^2}_{=1}  \langle w,\ph\rangle +\langle u,\ph \rangle \underbrace{\langle u,w \rangle}_{=0} \Big) \\
&= -\langle w,\ph\rangle
= -\langle w,D\pi_u \ph\rangle =-\langle w,w\rangle \\
&= -\abs{w}^2
\end{aligned}
\end{equation}
Going back to \eqref{EQ: ibniuUBiZBi3813HB1wA} the claim follows.
\end{proof}

\begin{definition} [Morse index of $p$-harmonic maps]
\label{DEFINITION: Morse index and Nullityof p-harmonic maps}
Let $u \in W^{1,p}(\Sigma;S^n)$ be a $p$-harmonic map.
Then we introduce the space of variations as
\begin{equation}
\label{EQ: wigunUBNI8193hf23f}
V_u = \Gamma(u^{-1} TS^n)= \left\{ w \in W^{1,2}(\Sigma; \R^{n+1}) \forwhich w(x) \in T_{u(x)} S^n, \quad \text{for a.e. } x\in\Sigma \right\}.
\end{equation}
The second variation is given by $Q_u:V_u \ra \R$,
\begin{equation}
\begin{aligned}
Q_u(w) 
&\coloneqq p\ (p-2) \int_\Sigma \left( 1+\abs{\nabla u}^2\right)^{p/2-2} \left(\nabla u \cdot \nabla w \right)^2  \ dvol_\Si \\
&\hspace{20mm}+p \int_\Sigma \left( 1+\abs{\nabla u}^2\right)^{p/2-1} \left[\abs{\nabla w}^2 - \abs{\nabla u}^2 \abs{w}^2 \right] \ dvol_\Si.
\end{aligned}
\end{equation}
The \underline{Morse index} of $u$ relative to the energy $E_p(u)$:  
\begin{equation}
\operatorname{Ind}_{E_p}(u):=\max \left\{\operatorname{dim}(W) ; W\right. \text{ is a sub vector space of } V_u \text{ s.t. } \left.\left.Q_u\right|_{W\setminus\{0\}}<0\right\}
\end{equation}
and the \underline{Nullity} of $u$ to be
\begin{equation}
\operatorname{Null}_{E_p}(u):= dim\big(\ker Q_u \big).~~~\square
\end{equation}
\end{definition}

Let $p_k>2$, $k \in \N$, be a sequence of exponents with 
\begin{equation}
p_k \searrow 2, \qquad \text{ as } k \rightarrow \infty.
\end{equation}
and let  
 $u_k \in W^{1,p_k}(\Sigma;S^n)$ be a sequence of $p_k$-harmonic maps.
satisfying 
\begin{equation}
\label{EQ: Uniform bound on the energy 8912H2egBeZ}
\sup_k E_{p_k}(u_k) = \sup_k \int_\Sigma (1+\abs{\nabla u_k}^2)^{\frac{p_k}{2}} dvol_{\Sigma}<\infty.
\end{equation}
Thanks to a classical result in concentration compactness theory, see for instance \cite{SU81}, we know that the sequence will converge up to subsequences strongly to a harmonic map away from a finite set of blow up points, where bubbles start to form while passing to the limit.
For our purposes it suffices to consider the simplified case of a single blow up point with only one bubble.
In this case we have the following

\begin{definition} [Bubble tree convergence with one bubble]
\label{Definition: Bubble tree convergence of p harm map with one bubble}
We say that the sequence $u_k$ bubble tree converges to a harmonic map and one single bubble if the following happens:
There exist harmonic maps $u_\infty \in W^{1,2}(\Sigma;S^n)$ and $v_\infty \in W^{1,2}(\C;S^n)$, a sequence of radii $(\delta_k)_{k \in \N} \subset \R_{>0}$, a sequence of points $(x_k)_{k \in \N}\subset\Sigma$ and a blow up point $q\in \Sigma$ such that
\begin{equation}
\label{EQ: Cond of bubb conv ji0wr931jJGA}
\begin{aligned}
&\bullet \quad u_k \rightarrow u_\infty, \qquad \text{ in } C^\infty_{loc}(\Sigma \setminus \{q\}), \text{ as } k \rightarrow \infty, \\
&\bullet \quad v_k(z) \coloneqq u_k\left(x_k + \delta_k z \right) \rightarrow v_\infty(z), \qquad \text{ in } C^\infty_{loc}(\C), \text{ as } k \rightarrow \infty, \\
&\bullet \quad \lim _{\eta \searrow 0} \limsup_{k \rightarrow\infty} \sup _{\delta_k / \eta<\rho<2 \rho<\eta} \int_{B_{2 \rho}\left(x_k\right) \backslash B_\rho\left(x_k\right)}\left|\nabla u_k\right|^2 d v o l_\Sigma=0,
\end{aligned}
\end{equation}
where in the second line $u_k(\cdot)$ is to be understood on a fixed conformal chart around the point $q$ and also
\begin{equation}
\label{EQ: Bubb conv 2}
 x_k \rightarrow q, \quad \de_k \rightarrow 0, \qquad \text{ as } k \rightarrow \infty.~~~\square
\end{equation}
\end{definition}

\noindent

Henceforward, we will assume that we are in the setting of Definition \ref{Definition: Bubble tree convergence of p harm map with one bubble}.
Furthermore, we are working in a fixed conformal chart around the point $q$ centered at the origin and parametrized by the unit ball $B_1=B_1(0)$.
Also for the sake of simplicity $x_k=0=q$ for any $k\in\N$.

 \ \vspace{15mm}
\section{Energy quantization}
The main goal of this section is to prove first the $L^2$ and then the $L^{2,1}$ quantization of a sequence of $p$-harmonic maps into the round sphere. The second result is of independent interest, since in the  sphere case only the  $L^2$ quantization is enough for  the pointwise estimates on degenerating annuli. 
We mention that $L^2$ and the  $L^{2,1}$ have already been obtained in   \cite{LZ19} by a combination of Wente's Inequality  for the Laplacian in $2$-D and the interpolation Lorentz spaces in the spirits of what was originally discovered in \cite{LiRi02}.\par
In all the current section we will denote by  $u_k$  a sequence $u_k$ of $p_k$-harmonic maps. 
Next we are going to follow the following approach:  \par
{\bf 1. } We consider the vector field 
\begin{equation}\label{Xk}
X_{k}=(1+\abs{\nabla u_k}^2)^{\frac{p_k}{2}-1} \nabla u_k \in L^{p_k^\prime}(B_1),~~p_k^\prime=\frac{p_k}{p_k-1},
\end{equation}
which satisfies by \eqref{EQ: Euler-Lagrange equations in div form for p harm} the equation
\begin{equation}\label{eqXk}
-\divergence(X_k)= (1+\abs{\nabla u_k}^2)^{\frac{p_k}{2}-1} (u_k\wedge\nabla u_k) \cdot \nabla u_k \qquad \text{ in } B_1
\end{equation}

Given $\eta\in (0,1)$, and $\delta_k\to 0$ as $k \to +\infty$, we consider the annulus 
\begin{equation}
A(\eta, \delta_k)\coloneqq B_{\eta}(0)\setminus \overline{B}_{\delta_k/\eta}(0),
\end{equation}
which is called neck-region.\par
 We use the Hodge/Helmholtz-Weyl Decomposition from Lemma \ref{LEMMA: Hodge/Helmholtz-Weyl Decomposition} on the domain $\Omega=B_1$ to find some $a,b \in W^{1,p_k^\prime}(B_1)$ such that
\begin{equation}\label{decXk}
X_k= \nabla a_{\eta,k} + \nabla^{\perp}b_{\eta,k}\quad \text{ in } B_1
\end{equation}
and with $\partial_{\tau} b=0$ on $\partial B_1$. \par
{\bf 2.} We properly estimate $ \nabla a_{\eta,k}$, and  $\nabla^{\perp}b_{\eta,k}$ in $L^{2,1}(A(\eta,\delta_k))$. \par
The main novelty and difficulty in the present section  is the estimate of  $\|\nabla b_{\eta,k}\|_{L^{2,1}(A(\eta,\delta_k))}$   obtained by applying commutator estimates discovered in \cite{RW83} and then used by Bernard and Rivi\`ere \cite{BR19} in the context of Willmore surfaces, and very recently in a joint work  \cite{DR23}
by the first and second authors  in connection with conservation laws for $p$-relaxation of conformally invariant variational problems.\par
\subsection{Analysis of $(1+\abs{\nabla u_k}^2)^{\frac{p_k}{2}-1} (u_k\wedge\nabla u_k)$}
Let us consider $X_k$ defined in \eqref{Xk} and its decomposition \eqref{decXk}.\par

Since with \eqref{EQ: Euler-Lagrange equations in div form for p harm} one has
\begin{equation}
\divergence \Big[(1+\abs{\nabla u_k}^2)^{\frac{p_k}{2}-1} (u_k\wedge\nabla u_k) \Big]=0, \qquad \text{ in } B_1=B_1(0),
\end{equation}
by Lemma \ref{LEMMA: Poincare Lemma stated in our setting}, we can find some potential $B_{\eta,k}\in W^{1,p_k^\prime}(B_1)$ such that
\begin{equation} \label{EQ: Equation for nabla B def}
\nabla^\perp B_{\eta,k} = (1+\abs{\nabla u_k}^2)^{\frac{p_k}{2}-1} (u_k\wedge\nabla u_k) \in L^{p_k^\prime}(B_1).
\end{equation}
A straightforward application of H\"older's inequality and \eqref{EQ: Uniform bound on the energy 8912H2egBeZ} gives
\begin{equation}
\label{EQ: estimate of B 980jt4unJUN}
\begin{aligned}
\norm{\nabla B_{\eta,k}}_{L^{p_k^\prime}(A(\eta,\delta_k))} 
&\le C \norm{(1+\abs{\nabla u_k}^2)^{\frac{p_k}{2}-1}}_{L^{\frac{p_k}{p_k-2}}(A(\eta,\delta_k))} \norm{u_k\wedge\nabla u_k}_{L^{p_k}(A(\eta,\delta_k))} \\
&\le C \norm{\nabla u_k}_{L^{p_k}(A(\eta,\delta_k))}.
\end{aligned}
\end{equation}
We get the equation
\begin{equation}
-\Delta a_{\eta,k}=-\divergence(X_k)=   \nabla^\perp B_{\eta,k} \cdot \nabla u_k \qquad \text{ in } B_1.
\end{equation}
 
Let $\widetilde u_k$ be the Whitney extension to $\C$ of $u_k|_{A(\eta,\delta_k)}$ coming from Lemma \ref{LEMMA: Whitney extension for ann in Lp} with
\begin{equation}
\label{EQ: jnfemorjsfJMNi29d91}
\begin{aligned}
\norm{\nabla \widetilde u_k}_{L^{p_k}(\C)} 
\le C\norm{\nabla  u_k}_{L^{p_k}(A(\eta,\delta_k))}
\end{aligned}
\end{equation}
and also
\begin{equation}
\label{EQ: wuinuifnvuUNIDWU0139jr1}
\supp (\nabla \widetilde u_k) \subset A(2\eta,\delta_k).
\end{equation}
Let $\widetilde B_{\eta,k}$ be the Whitney extensions to $\C$ of $B_{\eta,k}|_{A(\eta,\delta_k)}$ coming from Lemma \ref{LEMMA: Whitney extension for ann in Lp} with
\begin{equation}
\label{EQ: Control of tilde B uhni893hre}
\norm{\nabla \widetilde B_{\eta,k}}_{L^{p_k^\prime}(\C)} 
\le C\norm{\nabla  B_{\eta,k}}_{L^{p_k^\prime}(A(\eta,\delta_k))}
\end{equation}
and also
\begin{equation}
\supp (\nabla \widetilde B_{\eta,k}) \subset A(2\eta,\delta_k).
\end{equation}
Let $\varphi_{\eta,k}\in W^{1,2}(\C)$ be the solution of
\begin{equation}
\label{EQ: Eq for De phi}
-\Delta \varphi_{\eta,k} = \nabla^\perp\widetilde B_{\eta,k} \cdot \nabla\widetilde u_k \qquad \text{ in } \C.
\end{equation}
Now set
\begin{equation}
\mathfrak h_{\eta,k}=a_{\eta,k} - \varphi_{\eta,k} \qquad \text{ in } B_1.
\end{equation}
Clearly, $\mathfrak h_{\eta,k}$ is harmonic in $A(\eta,\delta_k)$.
Now we decompose the harmonic part $\mathfrak h_{\eta,k}$ as follows:
\begin{equation}
\mathfrak h_{\eta,k} =\mathfrak h_{\eta,k}^+ + \mathfrak h_{\eta,k}^- + \mathfrak h_{\eta,k}^0 \qquad \text{ in } A(\eta,\delta_k),
\end{equation}
where
\begin{equation}
\mathfrak h_{\eta,k}^+ = \Re \left(\sum_{l>0} h_l^k z^l \right), \qquad
\mathfrak h_{\eta,k}^- = \Re \left(\sum_{l<0} h_l^k z^l \right), \qquad
\mathfrak h^0_{\eta,k} = h^k_0 + C^k_0 \log\abs{z}.
\end{equation}
From \eqref{decXk} we get the decomposition
\begin{equation}
\label{EQ: Final decomp of na u}
X_k= \nabla^\perp b_{\eta,k}+\nabla \varphi_{\eta,k}+ \nabla \mathfrak h_{\eta,k} \qquad \text{ in } A(\eta,\delta_k).
\end{equation}

\begin{lemma}
There holds $C^k_0=0$ and hence $\nabla \mathfrak h_{\eta,k}^0=0$.
\end{lemma}

\begin{proof}
Let $r \in \left(\frac{\delta_k}{\eta},\eta \right)$.
Then
\begin{equation}
\label{EQ: 2r381brufr1}
\int_{B_r} \De \mathfrak h_{\eta,k} \ dz
= \int_{B_r} \divergence \nabla \mathfrak h_{\eta,k} \ dz
= \int_{\partial B_r} \partial_\nu \mathfrak h_{\eta,k}^+ \ d\sigma
+ \int_{\partial B_r} \partial_\nu \mathfrak h_{\eta,k}^- \ d\sigma
+ \int_{\partial B_r} \partial_\nu \mathfrak h_{\eta,k}^0 \ d\sigma
\end{equation}
Now we compute
\begin{equation}
\label{EQ: 8ru81bf81b}
\int_{\partial B_r} \partial_\nu \mathfrak h_{\eta,k}^+ \ d\sigma = 0 =
\int_{\partial B_r} \partial_\nu \mathfrak h_{\eta,k}^- \ d\sigma.
\end{equation}
Furthermore, 
\begin{equation}
\label{EQ: Comp flux of h0 edjn293}
\int_{\partial B_r} \partial_\nu \mathfrak h_{\eta,k}^0 \ d\sigma =C^k_0 \int_{\partial B_r} \frac{1}{r} \ d\sigma = 2\pi C^k_0.
\end{equation}
Combining \eqref{EQ: 2r381brufr1}, \eqref{EQ: 8ru81bf81b} and \eqref{EQ: Comp flux of h0 edjn293} we find
\begin{equation}\label{EQ: unifw81389BN92IJ4127dudhAWQKGX}
\begin{aligned}
C^k_0 
= \frac{1}{2\pi} \int_{B_r} \De \mathfrak h_{\eta,k} \ dz = \frac{1}{2\pi} \int_{B_r} \De a_{\eta,k} - \De \varphi_k \ dz
\end{aligned}
\end{equation}
Now we compute
\begin{equation}
\begin{aligned}
\int_{B_r} \De a_{\eta,k} \ dz
&= \int_{B_r} \nabla^\perp B_{\eta,k} \cdot \nabla u_k \ dz \\
&= \int_{B_r} \divergence\left(\nabla^\perp B_{\eta,k}\ u_k\right) \ dz \\
&= \int_{\partial B_r} \left(\nabla^\perp B_{\eta,k}\ u_k\right) \cdot \nu \ d\sigma \\
&= \int_{\partial B_r} \left(\nabla^\perp\widetilde B_{\eta,k}\ \widetilde u_k\right) \cdot \nu \ d\sigma \\
&= \int_{B_r} \divergence\left(\nabla^\perp\widetilde B_{\eta,k}\ \widetilde u_k\right) \ dz \\
&= \int_{B_r} \nabla^\perp\widetilde B_{\eta,k} \cdot \nabla\widetilde u_k \ dz
= \int_{B_r} \De \varphi_{\eta,k} \ dz
\end{aligned}
\end{equation}
Going back to \eqref{EQ: unifw81389BN92IJ4127dudhAWQKGX} the claim follows.
\end{proof}


\subsection{$L^{2}$-energy quantization}

We first show the $L^{2,\infty}$ quantization which is a direct consequence of  $\varepsilon$-regularity Lemma \ref{LEMMA: epsilon reg sacks uhlenbeck for p harm}.\par
\begin{theorem}[$L^{2,\infty}$-energy quantization]\label{THM: L2infty en quant of na uk}
There holds
\begin{equation}
\lim_{\eta\searrow0}\limsup_{k\rightarrow\infty}\norm{\nabla u_k}_{L^{2,\infty}(A(\eta,\delta_k))}=0.
\end{equation}
\end{theorem}

\begin{proof}
Let $x\in A(\eta,\delta_k)$. Introduce $r_x=\abs{x}/8$ and $t_x=3\abs{x}/4$.
Then we have
\begin{equation}
B_{2r_x}(x) \subset B_{2t_x} \setminus B_{t_x}
\end{equation}
and hence by \eqref{EQ: Cond of bubb conv ji0wr931jJGA} we can assume that for $\eta$ small and $k$ large one has $\norm{\nabla u_k}_{L^2(B_{2r_x}(x))}<\varepsilon$.
Then using $\varepsilon$-regularity Lemma \ref{LEMMA: epsilon reg sacks uhlenbeck for p harm} and the above inclusion of sets one has
\begin{equation}
\label{EQ: UIN893eNJ12ddqw81Ha}
\abs{\nabla u_k(x)} \le \norm{\nabla u_k}_{L^\infty(B_{r_x}(x))}
\le \frac{C}{\abs{x}} \norm{\nabla u_k}_{L^2(B_{2r_x}(x))}
\le \frac{C}{\abs{x}} \norm{\nabla u_k}_{L^2(B_{2t_x} \setminus B_{2t_x})}.
\end{equation}
Taking $L^{2,\infty}$-norms there follows
\begin{equation}
\begin{aligned}
\norm{\nabla u_k}_{L^{2,\infty}(A(\eta,\delta_k))} 
&\le C \sup _{\delta_k / 2\eta<\rho<2 \rho<2\eta} \left( \int_{B_{2 \rho} \backslash B_\rho}\left|\nabla u_k\right|^2 d v o l_\Sigma \right)^{\frac{1}{2}}.
\end{aligned}
\end{equation}
Using \eqref{EQ: Cond of bubb conv ji0wr931jJGA} the claim follows.
\end{proof}

 Let us consider again
 \begin{equation}
\label{Xk2}
X_k= \nabla^\perp b_{\eta,k}+\nabla \varphi_{\eta,k}+ \nabla \mathfrak h_{\eta,k} \qquad \text{ in } A(\eta,\delta_k).
\end{equation}

\begin{lemma}\label{LEMMA: L21 of na h}
There holds
\begin{equation}
\lim_{\eta\searrow0}\limsup_{k\rightarrow\infty} \norm{\nabla \mathfrak h_{\eta,k}^\pm}_{L^{2,1}(A(\eta,\delta_k))}=0.
\end{equation}
\end{lemma}

\begin{proof}
The proof is the same as   in Lemma III.3 of \cite{DGR22} and we omit it.
\end{proof}

\begin{lemma}
\label{LEMMA: First est of nabla b in necks}
For $k\in \N$ large and $\eta>0$  small one has
\begin{equation}
\label{EQ: HIninfwjei921ejir2}
\norm{\nabla b_{\eta,k}}_{L^{p_k^\prime}(B_1)} \le C (p_k-2)
\end{equation}
\end{lemma}

\begin{proof}
For $p>2$ consider the operator
\begin{equation}
S_p(f)\coloneqq \left[\frac{1+|f|^2}{1+\|f\|_{L^p(B_1)}^2}\right]^{\frac{p}{2}-1} f
\end{equation}
defined for functions $f \in L^p(B_1)$.
Consider the operator $T$ on the domain $B_1$ introduced in Lemma \ref{LEMMA: T assigns div free comp is bd unif}.
We compute the commutator 
\begin{equation}
\label{EQ: Comp of the commutator}
\begin{aligned}
\left[T ,S_{p_k}\right](\nabla u_k|_{B_1})
&\coloneqq T \big[S_{p_k}(\nabla u_k|_{B_1})\big]- S_{p_k}\big[\underbrace{T (\nabla u_k|_{B_1})}_{=0}\big] \\
&= T  \left[\frac{1}{\left(1+\|\nabla u_k \|_{L^{p_k}(B_1)}^2\right)^{\frac{p_k}{2}-1}} X_k\right] \\
&= \frac{1}{\left(1+\|\nabla u_k \|_{L^{p_k}(B_1)}^2\right)^{\frac{p_k}{2}-1}} T  \left[ X_k\right] \\
&= \frac{1}{\left(1+\|\nabla u_k \|_{L^{p_k}(B_1)}^2\right)^{\frac{p_k}{2}-1}} \nabla^\perp b_{\eta,k}.
\end{aligned}
\end{equation}
Using Lemma A.5 of \cite{BR19}
\footnote{We recall Lemma A.5 in  \cite{BR19}: Let $T: L^s(B_1) \rightarrow L^s(B_1)$ be a bounded linear operator for any $4/3 < s < 5$ such that 
$$
\norm T
:= \sup_{\substack{4/3 < s < 5, \\ 0 \neq f \in L^s(B_1)}}
\frac{\norm{T(f)}_{L^s(B_1)}}{\norm{f}_{L^s(B_1)}}
< \infty.
$$
Then there exists a global constant $C>0$ such that for any $\alpha \in (p-2, 1/2)$ one has
\begin{equation*}
\norm{TS_p(f)- S_pT(f)}_{L^{\frac{p}{p-1}}(B_1)}
\leq C \ \frac{p-2}{\alpha} \ \norm{T}
\left( \norm{f}_{L^p(B_1)}^\alpha + \left( 1+\norm f_{L^p(B_1)}^2\right)^{\alpha/2} \right) \norm{f}_{L^p(B_1)}^{1-\alpha}
\end{equation*}
for any $f \in L^{p}(B_1)$.}
we find
\begin{equation}
\begin{aligned}
\norm{\left[T,S_{p_k}\right](\nabla u_k)}_{{L^{p_k^\prime}(B_1)}}
\le C\ \frac{p_k-2}{\alpha} \norm{T} \left( \norm{\nabla u_k}_{L^{p_k}(B_1)} + (1+\norm{\nabla u_k}_{L^{p_k}(B_1)}^2)^{\frac{\alpha}{2}} \norm{\nabla u_k}_{L^{p_k}(B_1)}^{1-\alpha} \right),
\end{aligned}
\end{equation}
 
where $\alpha\in (p_k-2, 1/2)$ and
\begin{equation}
\norm{T} \coloneqq \sup_{\substack{s\in(4/3,5) \\ 0\ne f \in L^s(B_1)}} \frac{\norm{T(f)}_{L^s(B_1)}}{\norm{f}_{L^s(B_1)}}<\infty.
\end{equation}
Using \eqref{EQ: Comp of the commutator} there follows
\begin{equation}
\begin{aligned}
\hspace{-3mm}\norm{\nabla b_{\eta,k}}_{{L^{p_k^\prime}(B_1)}} 
&\le C\ \frac{p_k-2}{\alpha} \norm{T} \left( \norm{\nabla u_k}_{L^{p_k}(B_1)} + (1+\norm{\nabla u_k}_{L^{p_k}(B_1)}^2)^{\frac{\alpha}{2}} \norm{\nabla u_k}_{L^{p_k}(B_1)}^{1-\alpha} \right)\\
&\hspace{30mm} \times \left(1+\|\nabla u_k \|_{L^{p_k}(B_{1})}^2\right)^{\frac{p_k}{2}-1}.
\end{aligned}
\end{equation}
The claim follows from using \eqref{EQ: Uniform bound on the energy 8912H2egBeZ}.
\end{proof}

\begin{lemma} \label{LEMMA: L21 of na b}
For $\eta>0$ we have
\begin{equation}
\label{EQ: L21 est of nabla b eta k}
\lim_{k \rightarrow\infty} \norm{\abs{\nabla b_{\eta,k}}^{1/(p_k-1)}}_{L^{2,1}(B_1)}=0.
\end{equation}
\end{lemma}

\begin{proof}
Using Lemma C.1 of \cite{DR23} one has
\begin{equation}
\begin{aligned}
\norm{\abs{\nabla b_{\eta,k}}^{1/(p_k-1)}}_{L^{2,1}(B_1)}
&\le C\ \pi^{\frac{p_k-2}{2p_k}} \left( \frac{2(p_k-1)}{p_k-2} \right)^{\frac{p_k-1}{p_k}} \norm{\abs{\nabla b_{\eta,k}}^{1/(p_k-1)}}_{L^{p_k}(B_1)} \\
&\le C (p_k-2)^{\frac{1-p_k}{p_k}} \norm{\nabla b_{\eta,k}}_{L^{p_k^\prime}(B_1)}^{1/(p_k-1)}
\end{aligned}
\end{equation}
Combining this estimate with Lemma \ref{LEMMA: First est of nabla b in necks} we have
\begin{equation}
\label{EQ: Comb est of nabla b in the pr}
\norm{\abs{\nabla b_{\eta,k}}^{1/(p_k-1)}}_{L^{2,1}(B_1)}
\le C\ (p_k-2)^{\frac{1-p_k}{p_k}+\frac{1}{p_k-1}}.
\end{equation}
Computing
\begin{equation}
\frac{1-p_k}{p_k}+\frac{1}{p_k-1}= \frac{3p_k-p_k^2-1}{(p_k-1)p_k} \longrightarrow \frac{1}{2}
\end{equation}
we conclude
\begin{equation}
\limsup_{k \rightarrow 0}
\norm{\abs{\nabla b_{\eta,k}}^{1/(p_k-1)}}_{L^{2,1}(B_1)}
\le C\lim_{k \rightarrow 0} (p_k-2)^{\frac{1-p_k}{p_k}+\frac{1}{p_k-1}}=0.
\end{equation}
\end{proof}

\begin{theorem}[$L^2$-energy quantization]
\label{THM: L2 en quant of na u}
There holds
\begin{equation}
\lim_{\eta\searrow0}\limsup_{k\rightarrow\infty}\norm{\nabla u_k}_{L^{2}(A(\eta,\delta_k))}=0.
\end{equation}
\end{theorem}

\begin{proof}
We have
\begin{equation}
\begin{aligned}
\norm{\abs{X_k}^{\frac{1}{p_k-1}}}_{L^{2,\infty}(A(\eta,\delta_k))}
&\le \norm{(1+\abs{\nabla u_k})}_{L^{2,\infty}(A(\eta,\delta_k))} \\
&\le C \left( \norm{1}_{L^{2,\infty}(A(\eta,\delta_k))} + \norm{\nabla u_k}_{L^{2,\infty}(A(\eta,\delta_k))} \right) \\
&\le C \left( \eta + \norm{\nabla u_k}_{L^{2,\infty}(A(\eta,\delta_k))} \right)
\end{aligned}
\end{equation}
and hence with Theorem \ref{THM: L2infty en quant of na uk}
\begin{equation}
\label{EQ: L21 of Xk ifn3N51nM}
\lim_{\eta\searrow0}\limsup_{k\rightarrow\infty}\norm{\abs{X_k}^{\frac{1}{p_k-1}}}_{L^{2,\infty}(A(\eta,\delta_k))}=0.
\end{equation}
For any function $f$ on a bounded domain $\Omega$ and any $p>2$ we can estimate
\begin{equation}
\label{EQ: ZBU3rZfi8U9djogqw}
\begin{aligned}
\norm{f^{\frac{1}{p-1}}}_{L^{2,1}(\Omega)}
&= 2 \int_0^\infty t\ \abs{\{\abs{f}\ge t^{p-1}\}}^{\frac{1}{2}} \frac{dt}{t}\\
&= 2 \int_0^1 t\ \abs{\{\abs{f}\ge t^{p-1}\}}^{\frac{1}{2}} \frac{dt}{t} + 2 \int_1^\infty t\ \big|\{\abs{f}\ge \underbrace{t^{p-1}}_{\eqqcolon s}\}\big|^{\frac{1}{2}} \frac{dt}{t} \\
&= 2 \int_0^1  \abs{\{\abs{f}\ge t^{p-1}\}}^{\frac{1}{2}} \ dt + \frac{2}{p-1} \int_1^\infty s^{\frac{1}{p-1}} \abs{\{\abs{f}\ge s\}}^{\frac{1}{2}} \frac{ds}{s} \\
& \le 2\ \abs{\Omega}^{\frac{1}{2}} + \frac{2}{p-1} \int_1^\infty s \ \abs{\{\abs{f}\ge s\}}^{\frac{1}{2}} \frac{ds}{s}
 \\ 
& \le 2\ \abs{\Omega}^{\frac{1}{2}} + \frac{2}{p-1} \int_0^\infty s \ \abs{\{\abs{f}\ge s\}}^{\frac{1}{2}} \frac{ds}{s}
 \\ 
& = 2\ \abs{\Omega}^{\frac{1}{2}} + \frac{2}{p-1} \norm{f}_{L^{2,1}(\Omega)}.
\end{aligned}
\end{equation}
Combining \eqref{EQ: ZBU3rZfi8U9djogqw} with Lemma \ref{LEMMA: L21 of na h} we obtain
\begin{equation}
\label{EQ: L21 skewed na h ZH82bn9Ns1}
\lim_{\eta\searrow0}\limsup_{k\rightarrow\infty} \norm{\abs{\nabla \mathfrak h_{\eta,k}^\pm}^{\frac{1}{p_k-1}}}_{L^{2,1}(A(\eta,\delta_k))}=0.
\end{equation}
Going back to the decomposition \eqref{EQ: Final decomp of na u} and using Lemma \ref{LEMMA: L21 of na b}, \eqref{EQ: L21 of Xk ifn3N51nM}, \eqref{EQ: L21 skewed na h ZH82bn9Ns1}, we find
\begin{equation}
\label{EQ: NUunf81b1dvwdJ12J}
\lim_{\eta\searrow0}\limsup_{k\rightarrow\infty} \norm{\abs{\nabla \varphi_{\eta,k}}^{\frac{1}{p_k-1}}}_{L^{2,\infty}(A(\eta,\delta_k))}=0.
\end{equation}
Using Wente's inequality with \eqref{EQ: Eq for De phi} and also \eqref{EQ: jnfemorjsfJMNi29d91}, \eqref{EQ: estimate of B 980jt4unJUN}, \eqref{EQ: Control of tilde B uhni893hre}, \eqref{EQ: Uniform bound on the energy 8912H2egBeZ} one finds
\begin{equation}
\label{EQ: uZG813g87UH}
\norm{\nabla \varphi_{\eta,k}}_{L^{2,1}(\C)} \le C \norm{\nabla\widetilde B_{\eta,k}}_{L^{2}(\C)} \norm{\nabla\widetilde u_{k}}_{L^{2}(\C)} \le C \norm{\nabla u_{k}}_{L^{2}(A(\eta,\delta_k))}^2\le C.
\end{equation}
Combining \eqref{EQ: uZG813g87UH} with \eqref{EQ: ZBU3rZfi8U9djogqw}
we find
\begin{equation}
\label{EQ: Ifw8193rH23nN}
\norm{\abs{\nabla \varphi_{\eta,k}}^{\frac{1}{p_k-1}}}_{L^{2,1}(A(\eta,\delta_k))} \le C.
\end{equation}
 By H\"older's inequality in Lorentz spaces
\begin{equation}
\norm{\abs{\nabla \varphi_{\eta,k}}^{\frac{1}{p_k-1}}}_{L^{2}(A(\eta,\delta_k))}
\le \norm{\abs{\nabla \varphi_{\eta,k}}^{\frac{1}{p_k-1}}}_{L^{2, \infty}(A(\eta,\delta_k))} \norm{\abs{\nabla \varphi_{\eta,k}}^{\frac{1}{p_k-1}}}_{L^{2, 1}(A(\eta,\delta_k))}
\end{equation}
Hence, using \eqref{EQ: Ifw8193rH23nN} and \eqref{EQ: NUunf81b1dvwdJ12J} we get
\begin{equation}
\label{EQ: UNuwf83hn1H12fd}
\lim_{\eta\searrow0}\limsup_{k\rightarrow\infty} \norm{\abs{\nabla \varphi_{\eta,k}}^{\frac{1}{p_k-1}}}_{L^{2}(A(\eta,\delta_k))}=0.
\end{equation}
Going back to the decomposition \eqref{EQ: Final decomp of na u} and using Lemma \ref{LEMMA: L21 of na b}, \eqref{EQ: L21 skewed na h ZH82bn9Ns1} and \eqref{EQ: UNuwf83hn1H12fd} we obtain
\begin{equation}
\lim_{\eta\searrow0}\limsup_{k\rightarrow\infty}\norm{\abs{X_{k}}^{\frac{1}{p_k-1}}}_{L^{2}(A(\eta,\delta_k))}=0.
\end{equation}
The bound $\abs{\nabla u_k}\le \abs{X_k}^{\frac{1}{p_k-1}}$ gives the claimed result.
\end{proof}

\subsection{$L^{2,1}$-energy quantization}

\begin{theorem}[$L^{2,1}$-energy quantization]
There holds
\begin{equation}
\lim_{\eta\searrow0}\limsup_{k\rightarrow\infty}\norm{\nabla u_k}_{L^{2,1}(A(\eta,\delta_k))}=0.
\end{equation}
\end{theorem}

\begin{proof}
By Wente's inequality, \eqref{EQ: jnfemorjsfJMNi29d91} and \eqref{EQ: estimate of B 980jt4unJUN} there holds
\begin{equation}
\begin{aligned}
\norm{\nabla \varphi_{\eta,k}}_{L^{2,1}(\C)}
&\le C \norm{\nabla\widetilde B_{\eta,k}}_{L^{2}(\C)} \norm{\nabla\widetilde u_k} _{L^{2}(\C)} \\
&\le C \norm{\nabla  B_{\eta,k}}_{L^{2}(A(\eta,\delta_k))} \norm{\nabla  u_{k}}_{L^{2}(A(\eta,\delta_k))} \\
&\le C \norm{\nabla  u_{k}}_{L^{2}(A(\eta,\delta_k))}^2
\end{aligned}
\end{equation}
Using Theorem \ref{THM: L2 en quant of na u} we find
\begin{equation}
\label{EQ: jNU8uj893NIJ1}
\lim_{\eta\searrow0}\limsup_{k\rightarrow\infty} \norm{\nabla \varphi_{\eta,k}}_{L^{2,1}(\C)}=0.
\end{equation}
Using \eqref{EQ: ZBU3rZfi8U9djogqw} we find
\begin{equation}
\label{EQ: un23J2jd10jN012}
\lim_{\eta\searrow0}\limsup_{k\rightarrow\infty} \norm{\abs{\nabla \varphi_{\eta,k}}^{\frac{1}{p_k-1}}}_{L^{2,1}(A(\eta,\delta_k))}=0.
\end{equation}
Going back to the decomposition \eqref{EQ: Final decomp of na u} and combining Lemma \ref{LEMMA: L21 of na b}, \eqref{EQ: L21 skewed na h ZH82bn9Ns1} and \eqref{EQ: un23J2jd10jN012}
\begin{equation}
\lim_{\eta\searrow0}\limsup_{k\rightarrow\infty} \norm{\abs{X_k}^{\frac{1}{p_k-1}}}_{L^{2,1}(A(\eta,\delta_k))}=0.
\end{equation}
The bound $\abs{\nabla u_k}\le \abs{X_k}^{\frac{1}{p_k-1}}$ gives the claimed result.
\end{proof}

\newpage
\ \vspace{15mm}
\section{Pointwise Control of the Gradient in the Neck Regions}

The aim of this section is to get an improved pointwise  estimate of  a sequence of critical points of $E_p$   in the neck regions, i.e. annuli connecting the macroscopic surface with the bubbles. Such an estimate is obtained with the aid of suitable conservation laws and a delicate dydic argument. The core of such estimates, in combination with  a weighted Poincar\`e's  inequality in annuli,   is to show later  that negative variations of
the second derivative of the energy $E_p$ cannot be located in the neck regions and therefore these variations do not asymptotically contribute to the Morse index.
    \begin{theorem}
\label{THEOREM: on Pointwise Estimate of the Gradient in the Necks p harm}
For any given $\be \in \left(0,\log_2 (3/2) \right)$ we find that for $k \in \N$ large and $\eta>0$ small
\begin{equation}
\label{EQ: 283hfunv0lpdnu13ef6zh1}
\begin{gathered}
\forall x \in A(\et,\de_k): \qquad
\abs{x}^2 \abs{\nabla u_k(x)}^2 \le \left[ \left( \frac{\abs{x}}{\et} \right)^\be + \left( \frac{\de_k}{\et \abs{x}} \right)^\be \right]\boldsymbol{\epsilon}_{\eta, \delta_k} +  \boldsymbol{\mathrm c}_{\eta, \delta_k},
\end{gathered}
\end{equation}
where
\begin{equation}
\lim_{\eta\searrow0}\limsup_{k\rightarrow\infty}\boldsymbol{\epsilon}_{\eta, \delta_k}=0,
\qquad \text{and} \qquad
\lim_{\eta\searrow0}\limsup_{k\rightarrow\infty}\boldsymbol{\mathrm c}_{\eta, \delta_k} \log^2\left(\frac{\eta^2}{\delta_k}\right) =0.
\end{equation}
\end{theorem}

\noindent
We  introduce the notation
\begin{equation}
A_j=B_{2^{-j}}\setminus B_{2^{-j-1}} , \qquad j \in \N.
\end{equation}
Moreover we will denote by $u_k$, a sequence $u_{p_k}$ of sequence of $p_k$-harmonic maps, with $p_k\to 2^+$ as $k\to +\infty.$

\begin{lemma} \label{LEMMA: Linfty bd on na u in the necks unif}
There is a constant $C>0$ such that for $k\in\N$ large there holds
\begin{equation}\label{linftyuk}
\norm{\left(1+\abs{\nabla u_k}^2 \right)^{\frac{p_k}{2}-1}}_{L^\infty(\Sigma)} \le C.
\end{equation}
\end{lemma}

\begin{proof}
This result was shown in \cite{LZ19}.
We reproduce here the  proof by clarifying all the details.\par
Compute with a change of variables
\begin{equation}
\label{EQ: uwif399813hAO2Pd01}
\int_{B_{\delta_k/\eta}} \abs{\nabla u_k}^{p_k} dx
=\delta_k^2 \int_{B_{1/\eta}} \abs{\nabla u_k(\delta_k y)}^{p_k} dy
=\delta_k^{2-{p_k}} \int_{B_{1/\eta}} \abs{\nabla v_k(y)}^{p_k} dy.
\end{equation}
For $\eta>0$ small we have  
\begin{equation} \label{EQ: 873rh7u943141g1}
\norm{\nabla v_\infty}_{L^{p_k}(B_{1/\eta})} 
\ge \frac{1}{2} \norm{\nabla v_\infty}_{L^{p_k}(\C)}>0,
\end{equation}
where the last inequality is true since $v_\infty$ is nontrivial.
Furthermore, for $k$ large we have
\begin{equation}\label{EQ: 23n89hunfinf23}
\norm{\nabla v_k}_{L^{p_k}(B_{1/\eta})} 
\ge \frac{1}{2} \norm{\nabla v_\infty}_{L^{p_k}(B_{1/\eta})}.
\end{equation}
Combining this with \eqref{EQ: uwif399813hAO2Pd01} and \eqref{EQ: Uniform bound on the energy 8912H2egBeZ} we get for $k$ large
\begin{equation}
\delta_k^{2-p_k} \le C.
\end{equation}
{\bf Claim 1.}  It holds
\begin{equation}
\norm{\nabla u_k}_{L^\infty(\Omega)} \le \frac{C}{\delta_k}.
\end{equation}
for $\Omega=A(\eta,\delta_k)$, $\Omega=B_{\frac{\delta_k}{\eta}}$ and $\Omega=\Sigma \setminus B_{\eta}$.

Observe that the estimate \eqref{linftyuk} directly follows from {\bf Claim 1}:
\begin{equation}
\label{EQ: Unifwun38913HUus12}
\begin{aligned}
\norm{\left(1+\abs{\nabla u_k}^2 \right)^{\frac{p_k}{2}-1}}_{L^\infty(\Sigma)}
&\le \norm{\left(1+\abs{\nabla u_k}^2 \right)}_{L^\infty(\Sigma)}^{\frac{p_k}{2}-1}
\le \left( 1+C \ \delta_k^{-2} \right)^{\frac{p_k}{2}-1} \\
&=(\delta_k^2+C)^{\frac{p_k}{2}-1}\ \delta_k^{2-p_k}
\le C\ \delta_k^{2-p_k}\le C.
\end{aligned}
\end{equation}
\par
{\bf Proof of Claim 1.}\par

\textbf{1.)} $\boldsymbol{\Omega=A(\eta,\delta_k)}:$
\\
From \eqref{EQ: UIN893eNJ12ddqw81Ha} and \eqref{EQ: Uniform bound on the energy 8912H2egBeZ} we have for $x \in A(\eta,\delta_k)$
\begin{equation}
\abs{\nabla u_k(x)} \le \frac{C}{\abs{x}} \le \frac{C}{\delta_k}.
\end{equation}
\textbf{2.)} $\boldsymbol{\Omega=B_{\frac{\delta_k}{\eta}}}:$
\\
We claim that there exists some $\lambda_0>0$ such that
\begin{equation}\label{EQ: nhb732hr7zt05d2rh1z2ueth4}
\sup_{z \in \C} \norm{\nabla v_\infty}_{L^2(B_{\lambda_0}(z))}<\frac{\varepsilon}{2},
\end{equation}
where $\varepsilon>0$ is as in Lemma \ref{LEMMA: epsilon reg sacks uhlenbeck for p harm}.
Otherwise, we can find a sequence of radii $\lambda_j \rightarrow 0$ and points $z_j \in \C$ such that $\norm{\nabla v_\infty}_{L^2(B_{\lambda_j}(z_j))} \ge \varepsilon/2$.
If $\abs{z_j} \rightarrow \infty$, then this is not possible as we can then find an infinite sequence of disjoint balls where the energy on each ball is bounded from bellow.
This would contradict the assumption $\nabla v_\infty \in L^{2}(\C)$.
If otherwise $\abs{z_j} \nrightarrow \infty$, then we can assume up to subsequences $z_j \rightarrow z_0$. But then we have with Lebesgues differentiation theorem the following contradiction
\begin{equation}
0= \lim_{j \rightarrow \infty} \int_{B_{\max(\lambda_j,2|z_j-z_0|)}(z_j)} \abs{\nabla v_\infty}^2 \ dz \ge \liminf_{j \rightarrow \infty} \int_{B_{\lambda_j}(z_j)} \abs{\nabla v_\infty}^2 \ dz \geq \frac{\varepsilon}{2} >0.
\end{equation}
Hence, we have shown \eqref{EQ: nhb732hr7zt05d2rh1z2ueth4}.
Now as $\nabla v_k \rightarrow \nabla v_\infty$ in $L^\infty_{loc}(\C)$ we can assume that for any given $\eta>0$ we have for $k\in \N$ large 
\begin{equation}
\norm{\nabla v_k- \nabla v_\infty}_{L^2\left(B_{\frac{1}{\eta}+\lambda_0 }\right)} < \frac{\varepsilon}{2} .
\end{equation}
(Here we again using the simplified notation $B_{\frac{1}{\eta} + \lambda_0}=B_{\frac{1}{\eta} + \lambda_0}(0)$.)
Hence, for any $x \in B_{\de_k/\eta}$ one has
\begin{equation}
\begin{aligned}
\norm{\nabla v_k}_{L^2(B_{\lambda_0}(x/\delta_k ))}
&\le \norm{\nabla v_\infty}_{L^2(B_{\lambda_0}(x/\delta_k ))} + \norm{\nabla v_k -\nabla v_\infty}_{L^2(B_{\lambda_0}(x/\delta_k ))} \\
&\le \frac{\varepsilon}{2} + \norm{\nabla v_k -\nabla v_\infty}_{L^2\left(B_{\frac{1}{\eta} + \lambda_0}\right)} \\
& <\varepsilon.
\end{aligned}
\end{equation}
where we also used \eqref{EQ: nhb732hr7zt05d2rh1z2ueth4}.
 
\par
We obtain with Lemma \ref{LEMMA: epsilon reg sacks uhlenbeck for p harm} that
\begin{equation}
\abs{\nabla u_k (x)}
= \frac{1}{\delta_k} \abs{\nabla v_k \left(\frac{x}{\delta_k}\right)}
\le \frac{C}{\delta_k \ \lambda_0} \norm{\nabla v_k}_{L^{2}(B_{\lambda_0}(x/\delta_k))}
\le \frac{C}{\delta_k}.
\end{equation}
\textbf{3.)} $\boldsymbol{\Omega=\Sigma\setminus B_{\eta}}:$
\\
By the same arguments as above we can show that there exists some $\lambda_0>0$ such that
\begin{equation}
\sup_{x \in \Sigma} \norm{\nabla u_\infty}_{L^{2}\left(B_{\lambda_0}(x)\right)} < \frac{\varepsilon}{2} .
\end{equation}
Now suppose $\eta \in (0,\lambda_0)$.
Then for large $k\in\N$ one has
\begin{equation}
\norm{\nabla u_k - \nabla u_\infty}_{L^2\left(\Sigma \setminus B_{\eta/2}\right)} < \frac{\varepsilon}{2}.
\end{equation}
Let now $x \in \Sigma \setminus B_{\eta}$.
Then using the fact that $B_{\eta/2}(x) \subset \Sigma \setminus B_{\eta/2}$ we obtain
\begin{equation}
\begin{aligned}
\norm{\nabla u_k}_{L^2(B_{\eta/2}(x))}
&\le \norm{\nabla u_\infty}_{L^2(B_{\eta/2}(x))} + \norm{\nabla u_k - \nabla u_\infty}_{L^2(B_{\eta/2}(x))} \\
&\le \norm{\nabla u_\infty}_{L^2(B_{\lambda_0}(x))} + \norm{\nabla u_k - \nabla u_\infty}_{L^2(\Sigma \setminus B_{\eta/2})} \\
&< \frac{\varepsilon}{2} + \frac{\varepsilon}{2} = \varepsilon.
\end{aligned}
\end{equation}
Hence, by Lemma \ref{LEMMA: epsilon reg sacks uhlenbeck for p harm} one has
\begin{equation}
\abs{\nabla u_{k}(x)}
\le \norm{\nabla u_k}_{L^\infty(B_{\eta/4}(x))}
\le \frac{C}{\eta} \norm{\nabla u_k}_{L^2(B_{\eta/2}(x))} \le \frac{C}{\eta} \le \frac{C}{\delta_k}
\end{equation}
\end{proof}
 
Now recall that we are working with the decomposition introduced in \eqref{EQ: Final decomp of na u}. \par

\begin{lemma}[Estimate of $\nabla \varphi_{\eta,k}$]\label{LEMMA: Dyadic est of na phi start}
For any $\gamma \in \left(0,\frac{2}{3}\right]$ there is a constant $C=C(\gamma)>0$ such that for $k\in\N$ large and $\eta>0$ small
\begin{equation}
\int_{A_j} \abs{\nabla \varphi_{\eta,k}}^2 \ dx 
\le C  \norm{\nabla u_k}_{L^2(A(\eta,\delta_k))}^2 \left(\gamma^j  + \sum_{l=0}^\infty \gamma^{\abs{l-j}} \int_{A_l} \abs{\nabla\widetilde u_k}^2 \ dx\right),
\end{equation}
where $j\in \N$.
\end{lemma}

\begin{proof}
Applying the weighted Wente inequality Lemma F.1 of \cite{DGR22} to \eqref{EQ: Eq for De phi} we have
\begin{equation}
\begin{aligned}
\int_{A_j} \abs{\nabla \varphi_{\eta,k}}^2 \ dx 
&\le \gamma^j \int_{\C} \abs{\nabla \varphi_{\eta,k}}^2 \ dx + C \int_{A(2\eta,\delta_k)} \abs{\nabla\widetilde B_{\eta,k}}^2 \ dx\ \sum_{l=0}^\infty \gamma^{\abs{l-j}} \int_{A_l} \abs{\nabla\widetilde u_k}^2 \ dx
\end{aligned}
\end{equation}
Using Lemma \ref{LEMMA: Whitney extension for ann in Lp}, \eqref{EQ: Equation for nabla B def} and Lemma \ref{LEMMA: Linfty bd on na u in the necks unif} we find that
\begin{equation}
\begin{aligned}
\int_{A(2\eta,\delta_k)} \abs{\nabla\widetilde B_{\eta,k}}^2 \ dx
&\le \int_{A(\eta,\delta_k)} \abs{\nabla B_{\eta,k}}^2 \ dx \\
&\le \norm{\left(1+\abs{\nabla u_k}^2 \right)^{\frac{p_k}{2}-1}}_{L^\infty(A(\eta,\delta_k))} \norm{\nabla u_k}_{L^2(A(\eta,\delta_k))}^2 \\
&\le C \norm{\nabla u_k}_{L^2(A(\eta,\delta_k))}^2
\end{aligned}
\end{equation}
By Wente's inequality applied to \eqref{EQ: Eq for De phi} and the above estimate we have
\begin{equation} \label{EQ: nuifwerifndqJNDW093e12e}
\int_{\C} \abs{\nabla \varphi_{\eta,k}}^2 \ dx \le C \norm{\nabla\widetilde B_{\eta,k}}_{L^2(\C)} \norm{\nabla\widetilde u_k}_{L^2(\C)}
\le C \norm{\nabla u_k}_{L^2(A(\eta,\delta_k))}^2.
\end{equation}
Hence,
\begin{equation}
\int_{A_j} \abs{\nabla \varphi_{\eta,k}}^2 \ dx 
\le C \gamma^j \norm{\nabla u_k}_{L^2(A(\eta,\delta_k))}^2 + C \norm{\nabla u_k}_{L^2(A(\eta,\delta_k))}^2 \sum_{l=0}^\infty \gamma^{\abs{l-j}} \int_{A_l} \abs{\nabla\widetilde u_k}^2 \ dx.
\end{equation}
\end{proof}

\begin{lemma}[Estimate of $\nabla\mathfrak h_{\eta,k}$]
\label{LEMMA: on the est of na h in necks}
For $k\in\N$ large and $\eta>0$ small there holds
\begin{equation}
\int_{A_j} \abs{\nabla \mathfrak h_{\eta,k}}^2 dz 
\le C \left[ \left( \frac{2^{-j}}{\et}\right)^2 + \left( \frac{\de_k}{2^{-j}\et}\right)^2 \right] \left(\norm{\nabla u_k}_{L^2(A(\eta,\delta_k))}^2 + \norm{\nabla b_{\eta,k}}_{L^{p_k^\prime}(B_1)} \right) ,
\end{equation}
where $j\in\N$ is such that $\frac{2\delta_k}{\eta}\le2^{-j-1}<2^{-j}\le \frac{\eta}{2} $.
\end{lemma}

\begin{proof}
For $x$ with $\frac{2\delta_k}{\eta}\le \abs{x} \le \frac{\eta}{2}$
by Lemma D.1 of \cite{DGR22} we have
\begin{equation}
\abs{\nabla \mathfrak h_{\eta,k}^+(x)}^2 \le \frac{C}{\eta^2} \norm{\nabla \mathfrak h_{\eta,k}^+}_{L^2(A(\eta,\delta_k))}^2
\end{equation}
and
\begin{equation}
\abs{\nabla \mathfrak h_{\eta,k}^-(x)}^2 \le \frac{C\delta_k^2}{\eta^2\abs{x}^4} \norm{\nabla \mathfrak h_{\eta,k}^-}_{L^2(A(\eta,\delta_k))}^2.
\end{equation}
Integrating over $x \in A_j$ we find that 
\begin{equation}
\int_{A_j}\abs{\nabla \mathfrak h_{\eta,k}^+(x)}^2 dz \le \frac{C\ 2^{-2j}}{\eta^2} \norm{\nabla \mathfrak h_{\eta,k}^+}_{L^2(A(\eta,\delta_k))}^2
\end{equation}
and
\begin{equation}
\int_{A_j}\abs{\nabla \mathfrak h_{\eta,k}^-(x)}^2 dz \le \frac{C\delta_k^2}{\eta^2\ 2^{-2j}} \norm{\nabla \mathfrak h_{\eta,k}^-}_{L^2(A(\eta,\delta_k))}^2.
\end{equation}
Hence,
\begin{equation}
\label{EQ: jiwnreINqwd0921e1dsw1274}
\int_{A_j} \abs{\nabla \mathfrak h_{\eta,k}}^2 dz 
\le C \left[ \left( \frac{2^{-j}}{\et}\right)^2 + \left( \frac{\de_k}{2^{-j}\et}\right)^2 \right] \left( \norm{\nabla \mathfrak h_{\eta,k}^+}_{L^2(A(\eta,\delta_k))}^2 + \norm{\nabla \mathfrak h_{\eta,k}^-}_{L^2(A(\eta,\delta_k))}^2 \right).
\end{equation}
As $\int_{\partial B_\rho} \nabla \mathfrak h_{\eta,k}^+ \cdot \nabla \mathfrak h_{\eta,k}^- \ d\sigma=0$ for any $\rho \in \left(\frac{\delta_k}{\eta}, \eta \right)$, we have with H\"older's inequality
\begin{equation}
\label{EQ: iweJIN2190je1NJ}
\begin{aligned}
&\hspace{-20mm}\norm{\nabla \mathfrak h_{\eta,k}^+}_{L^2(A(\eta,\delta_k))}^2 + \norm{\nabla \mathfrak h_{\eta,k}^-}_{L^2(A(\eta,\delta_k))}^2 \\
&= \norm{\nabla \mathfrak h_{\eta,k}}_{L^2(A(\eta,\delta_k))}^2 \\
&\le C\left(\norm{X_k}_{L^2(A(\eta,\delta_k))}^2 + \norm{\nabla \varphi_{\eta,k}}_{L^2(A(\eta,\delta_k))}^2 + \norm{\nabla b_{\eta,k}}_{L^2(A(\eta,\delta_k))}^2 \right) \\
&\le C\left(\norm{X_k}_{L^2(A(\eta,\delta_k))}^2 + \norm{\nabla u_{k}}_{L^2(A(\eta,\delta_k))}^2 + \norm{\nabla b_{\eta,k}}_{L^2(B_1)}^2 \right),
\end{aligned}
\end{equation}
where in the last line we used \eqref{EQ: nuifwerifndqJNDW093e12e} and \eqref{EQ: Final decomp of na u} .
Now we bound with Lemma \ref{LEMMA: Linfty bd on na u in the necks unif}
\begin{equation}
\label{EQ: uwibfeUNNWD0219hn10}
\norm{X_k}_{L^2(A(\eta,\delta_k))}^2 \le 
\norm{\left(1+\abs{\nabla u_k}^2 \right)^{\frac{p_k}{2}-1}}_{L^\infty(A(\eta,\delta_k))}^2 \norm{\nabla u_k}_{L^2(A(\eta,\delta_k))}^2
\le C \norm{\nabla u_k}_{L^2(A(\eta,\delta_k))}^2
\end{equation}
Furthermore, integrating by parts and by H\"older's inequality
\begin{equation}
\label{EQ: wifnUINDW0912eh811eSDW}
\norm{\nabla b_{\eta,k}}_{L^2(B_1)}^2
= \int_{B_1} \nabla^\perp  b_{\eta,k} \cdot X_k
\le \norm{\nabla b_{\eta,k}}_{L^{p_k^\prime}(B_1)} \norm{X_{k}}_{L^{p_k}(B_1)}.
\end{equation}
We bound with Lemma \ref{LEMMA: Linfty bd on na u in the necks unif} and with \eqref{EQ: Uniform bound on the energy 8912H2egBeZ}
\begin{equation}
\label{EQ: weuiofnBNIDW0913n13SW}
\begin{aligned}
\norm{X_{k}}_{L^{p_k}(B_1)}^{p_k}
&=\int_{B_1} (1+\abs{\nabla u_k}^2)^{(\frac{p_k}{2} - 1)p_k} \abs{\nabla u_k}^{p_k} dx \\
&\le \norm{\left(1+\abs{\nabla u_k}^2 \right)^{\frac{p_k}{2}-1}}_{L^\infty(\Sigma)}^{p_k} \norm{\nabla u_k}_{L^{p_k}(\Sigma)}^{p_k} \\
&\le C.
\end{aligned}
\end{equation}
Combining \eqref{EQ: jiwnreINqwd0921e1dsw1274}, \eqref{EQ: iweJIN2190je1NJ}, \eqref{EQ: uwibfeUNNWD0219hn10}, \eqref{EQ: wifnUINDW0912eh811eSDW} and \eqref{EQ: weuiofnBNIDW0913n13SW} one finds
\begin{equation}
\int_{A_j} \abs{\nabla \mathfrak h_{\eta,k}}^2 dz 
\le C \left[ \left( \frac{2^{-j}}{\et}\right)^2 + \left( \frac{\de_k}{2^{-j}\et}\right)^2 \right] \left(\norm{\nabla u_k}_{L^2(A(\eta,\delta_k))}^2 + \norm{\nabla b_{\eta,k}}_{L^{p_k^\prime}(B_1)} \right)
\end{equation}
\end{proof}

In the  following lemma we show a sort of entropy condition linking the parameter $p$ and the conformal class of the neck regions. In the sphere case it is actually a consequence of the $\varepsilon$-regularity as it was already shown in \cite{LZ19}.\par

\begin{lemma}
\label{LEMMA: Growth of  pk and ln etadelta}
For $\eta>0$ small there holds
\begin{equation}
\lim_{\eta\searrow0}\limsup_{k\rightarrow\infty}\ (p_k-2) \log\left(\frac{\eta^2}{\delta_k}\right)=0.
\end{equation}
\end{lemma}

\begin{proof}
We apply Lemma \ref{LEMMA: On the mono est for p harm} for the radii $r=\delta_k/\eta$ and $R=\eta$ to get
\begin{equation}
\begin{aligned}
&(p_k-2) \log\left(\frac{\eta^2}{\delta_k} \right) \norm{\nabla u_k}_{L^2(B_\frac{\delta_k}{\eta})}^2 \\
&\hspace{20mm} \le p_k \norm{(1+\abs{\nabla u_k }^2)^{\frac{p_k }{2}-1}}_{L^\infty(A(\eta,\delta_k))} \left(2\norm{\nabla u_k }_{L^2(A(\eta,\delta_k))}^2 + \eta^2-\frac{\delta_k^2}{\eta^2}  \right)
\end{aligned}
\end{equation}
For large $k$ and small $\eta>0$ there holds
\begin{equation}
\label{EQ: iwnfbiININ12081dn012s}
\norm{\nabla v_\infty}_{L^2(\C)} 
\le 2\norm{\nabla v_\infty}_{L^2(B_\frac{1}{\eta})} 
\le 4\norm{\nabla v_k}_{L^2(B_\frac{1}{\eta})} 
=4 \norm{\nabla u_k}_{L^2(B_\frac{\delta_k}{\eta})}.
\end{equation}
Note that as $v_\infty$ is nontrivial one has $\norm{\nabla v_\infty}_{L^2(\C)} >0$.
Using also Lemma \ref{LEMMA: Linfty bd on na u in the necks unif} we find that 
\begin{equation}
(p_k-2) \log\left(\frac{\eta^2}{\delta_k} \right) 
\le C  \left(\norm{\nabla u_k }_{L^2(A(\eta,\delta_k))}^2 + \eta^2\right)
\end{equation}
The claim follows by using Theorem \ref{THM: L2 en quant of na u}.
\end{proof}

\begin{corollary}
\label{COROLLARY: Limit of na uk to pk minus 2 eq 1}
\begin{equation}
\lim_{k \rightarrow \infty} \norm{\left(1+\abs{\nabla u_k}^2 \right)^{\frac{p_k}{2}-1}}_{L^\infty(\Sigma)}=1.
\end{equation}
\end{corollary}

\begin{proof}
By Lemma \ref{LEMMA: Growth of  pk and ln etadelta} we have
\begin{equation}
1= \lim_{\eta\searrow0}\limsup_{k\rightarrow\infty} \frac{\eta^{2(p_k-2)}}{\delta_k^{p_k-2}} = \lim_{k\rightarrow\infty} \delta_k^{2-p_k}.
\end{equation}
Following estimate \eqref{EQ: Unifwun38913HUus12} we have
\begin{equation}
\begin{aligned}
\norm{\left(1+\abs{\nabla u_k}^2 \right)^{\frac{p_k}{2}-1}}_{L^\infty(\Sigma)}
&\le \underbrace{(\delta_k^2+C)^{\frac{p_k}{2}-1}}_{\rightarrow 1} \ \underbrace{\delta_k^{2-p_k}}_{\rightarrow 1}.
\end{aligned}
\end{equation}
\end{proof}

 The new precise control on the energy of $b_{\eta,k}$ developed in Lemma \ref{LEMMA: First est of nabla b in necks} together with the entropy condition as in Lemma \ref{LEMMA: Growth of  pk and ln etadelta} (coming from \cite{LZ19}) allows to suitably control $\nabla b_{\eta,k}$ in the necks:
\begin{lemma}\label{LEMMA: Control of const of na b in necks}
For $k\in\N$ large and $\eta>0$ small there holds
\begin{equation}
\norm{\nabla b_{\eta,k}}_{L^{p_k^\prime}(A(\eta,\delta_k))}
\le C_{\eta,k},
\end{equation}
where
\begin{equation}
\lim_{\eta\searrow0}\limsup_{k\rightarrow\infty}\ \log\left(\frac{\eta^2}{\delta_k}\right) C_{\eta,k}=0.
\end{equation}
\end{lemma}

\begin{proof}
By Lemma \ref{LEMMA: First est of nabla b in necks}
\begin{equation}
\label{EQ: uundqu72ZB2jUQlswQ}
\norm{\nabla b_{\eta,k}}_{L^{p_k^\prime}(B_1)}
\le C (p_k-2).
\end{equation}
The claim follows by applying Lemma \ref{LEMMA: Growth of  pk and ln etadelta}.
\end{proof}

\begin{lemma} \label{LEMMA: Lower bound on the mass over diadic annulus with p-2}
There exists a constant $C>0$ such that for $k\in\N$ large and $\eta>0$ small the following holds:
For any $j\in\N$ with $\frac{\delta_k}{\eta}<2^{-j}<\eta$ we have
\begin{equation}
(p_k-2) \le C\left( \int_{A_j} \abs{\nabla u_k}^2 dx + 2^{-2j} \right).
\end{equation}
\end{lemma}

\begin{proof}
Combining Lemma \ref{LEMMA: On the mono est for p harm} and Lemma \ref{LEMMA: Linfty bd on na u in the necks unif} we find
\begin{equation}
(p_k-2) \norm{\nabla u_k}_{L^2(B_{2^{-j-1}})}^2
\le C \left( \int_{A_j} \abs{\nabla u_k}^2 dx + 2^{-j} \right)
\end{equation}
To conclude we use $\norm{\nabla u_k}_{L^2(B_{2^{-j-1}})} \ge \norm{\nabla u_k}_{L^2(B_\frac{\delta_k}{\eta})}$ and also \eqref{EQ: iwnfbiININ12081dn012s}.
\end{proof}

\begin{lemma}\label{LEMMA: Lemma on the start of dyadic est}
There exists a constant $C>0$ such that for $k\in\N$ large and $\eta>0$ small the following holds:
For any $j\in\N$ with $\frac{\delta_k}{\eta}<2^{-j}<\eta$ we have
\begin{equation}
\int_{A_j} \abs{\nabla u_k}^2 dx 
\le C \left( \norm{\nabla b_{\eta,k}}_{L^{p_k^\prime}(A_j)}^2 + \int_{A_j}\abs{\nabla \varphi_{\eta,k}}^2 dx + \int_{A_j}\abs{\nabla \mathfrak h_{\eta,k}}^2 dx+ 2^{-2j(p_k-1)} \right).
\end{equation}
\end{lemma}

\begin{proof}
Applying H\"older's inequality we find that
\begin{equation}
\int_{A_j} \abs{\nabla u_k}^2 dx 
\le \left(\int_{A_j} \abs{\nabla u_k}^{p_k} dx \right)^{\frac{2}{p_k}} \underbrace{\abs{A_j}^{\frac{p_k-2}{p_k}}}_{\le C}
\le C \left(\int_{A_j} \abs{\nabla u_k}^{p_k} dx \right)^{\frac{2}{p_k}}.
\end{equation}
By Minkowski's inequality and the above estimate
\begin{equation}
\begin{aligned}
\int_{A_j} \abs{\nabla u_k}^{p_k} dx + 2^{-jp_k}
&\ge 2^{1-\frac{p_k}{2}} \left[\left( \int_{A_j} \abs{\nabla u_k}^{p_k} dx\right)^{\frac{2}{p_k}} + 2^{-2j} \right]^{\frac{p_k}{2}} \\
&\ge C \left[ \int_{A_j} \abs{\nabla u_k}^2 dx + 2^{-2j} \right]^{\frac{p_k}{2}} \\
&= C \left[ \int_{A_j} \abs{\nabla u_k}^2 dx + 2^{-2j} \right]^{\frac{p_k^\prime}{2}} \left[ \int_{A_j} \abs{\nabla u_k}^2 dx + 2^{-2j} \right]^{\frac{p_k(p_k-2)}{2(p_k-1)}},
\end{aligned}
\end{equation}
where in the last line we used that $\frac{p_k}{2} = \frac{p_k^\prime}{2} + \frac{p_k(p_k-2)}{2(p_k-1)}$.
By Lemma \ref{LEMMA: Lower bound on the mass over diadic annulus with p-2} we get
\begin{equation}
\left[ \int_{A_j} \abs{\nabla u_k}^2 dx + 2^{-2j} \right]^{\frac{p_k(p_k-2)}{2(p_k-1)}}
\ge \Big[C (p_k-2) \Big]^{\frac{p_k(p_k-2)}{2(p_k-1)}}
= \underbrace{C^{\frac{p_k(p_k-2)}{2(p_k-1)}}}_{\rightarrow1}  \Big[ \underbrace{(p_k-2)^{(p_k-2)}}_{\rightarrow 1} \Big]^{\frac{p_k}{2(p_k-1)}},
\end{equation}
where we use the fact that $\lim_{x\searrow0}x^x=1.$
We conclude that for $k$ large there holds
\begin{equation}
\left[ \int_{A_j} \abs{\nabla u_k}^2 dx + 2^{-2j} \right]^{\frac{p_k^\prime}{2}}
\le C \int_{A_j} \abs{\nabla u_k}^{p_k} dx + 2^{-jp_k}
\end{equation}
Finally, taking the $\frac{2}{p_k^\prime}$ power, using Minkowski's inequality and H\"older's inequality one bounds
\begin{equation}
\begin{aligned}
\int_{A_j} \abs{\nabla u_k}^2 dx + 2^{-2j}
&\le C \left[\int_{A_j} \abs{\nabla u_k}^{p_k} dx + 2^{-jp_k} \right]^{\frac{2}{p_k^\prime}} \\
&\le C\left(\int_{A_j} \abs{\nabla u_k}^{p_k} dx  \right)^{\frac{2}{p_k^\prime}} + C\ 2^{-2j(p_k-1)} \\
&\le C\left(\int_{A_j} \abs{X_k}^{p_k^\prime} dx  \right)^{\frac{2}{p_k^\prime}} + C\ 2^{-2j(p_k-1)} \\
&\le C \left( \norm{\nabla b_{\eta,k}}_{L^{p_k^\prime}(A_j)}^2 + \norm{\nabla \varphi_{\eta,k}}_{L^{p_k^\prime}(A_j)}^2 + \norm{\nabla \mathfrak h_{\eta,k}}_{L^{p_k^\prime}(A_j)}^2 + 2^{-2j(p_k-1)} \right) \\
&\le C \left( \norm{\nabla b_{\eta,k}}_{L^{p_k^\prime}(A_j)}^2 + \norm{\nabla \varphi_{\eta,k}}_{L^{2}(A_j)}^2 + \norm{\nabla \mathfrak h_{\eta,k}}_{L^{2}(A_j)}^2 + 2^{-2j(p_k-1)} \right).
\end{aligned}
\end{equation}
\end{proof}

\begin{proof}[Proof (of Theorem \ref{THEOREM: on Pointwise Estimate of the Gradient in the Necks p harm})]
For $j \in \N$ with $\frac{4\delta_k}{\eta}\le 2^{-j} \le \frac{\eta}{2}$ combining Lemma \ref{LEMMA: Lemma on the start of dyadic est} and Lemma \ref{LEMMA: Dyadic est of na phi start} one has
\begin{equation}
\begin{aligned}
\int_{A_j} \abs{\nabla u_k}^2 dx 
&\le c_0 \Bigg[ 2^{-2j(p_k-1)}+ \gamma^j \norm{\nabla u_k}_{L^2(A(\eta,\delta_k))}^2 + \norm{\nabla \mathfrak h_{\eta,k}}_{L^{2}(A_j)}^2+ \norm{\nabla b_{\eta,k}}_{L^{p_k^\prime}(A_j)}^2 \\
&\hspace{20mm}+\norm{\nabla u_k}_{L^2(A(\eta,\delta_k))}^2 \sum_{l=0}^\infty \gamma^{\abs{l-j}} \int_{A_l} \abs{\nabla\widetilde u_k}^2  dx
 \Bigg].
\end{aligned}
\end{equation}
for some constant $c_0>0$ independent of $j$, $\eta$ or $k$.
Let us introduce
\begin{equation}
\begin{aligned}
a_j & \coloneqq \int_{A_j} \abs{\na\widetilde u_k}^2  dx, \\
b_j & \coloneqq c_0\left[ 2^{-2j(p_k-1)}+ \gamma^j \norm{\nabla u_k}_{L^2(A(\eta,\delta_k))}^2 + \norm{\nabla \mathfrak h_{\eta,k}}_{L^{2}(A_j)}^2+ \norm{\nabla b_{\eta,k}}_{L^{p_k^\prime}(A_j)}^2\right], \\
\varepsilon_0 &= \varepsilon_0(\et,\de_k) \coloneqq  c_0 \int_{A(\eta,\delta_k)} \abs{\nabla u_k}^2  dx.
\end{aligned}
\end{equation}
In this notation this is equivalent to
\begin{equation}
a_j \le b_j
+ \epv_0 \sum_{l=0}^\infty \gamma^{\abs{l-j}} a_l, \qquad \forall j \in [s_1,s_2]\coloneqq \Bigg[\ceil{-\log_2\left(\frac{\eta}{2}\right)}, \floor{-\log_2\left(\frac{4\delta_k}{\eta}\right)} \Bigg].
\end{equation}
Now we apply Lemma G.1 of \cite{DGR22} for some fixed $j \in \{s_1,\dots,s_2 \}$. 
Then for $\gamma<\mu<1$ there exists $C_{\mu, \gamma}>0$ such that
\begin{equation}
\begin{aligned}
\sum_{l=s_1}^{s_2} \mu^{|l-j|} a_{l} 
&= \sum_{l=s_1}^{s_2} \mu^{|l-j|} b_{l} + C_{\mu, \gamma}\ \varepsilon_0 \sum_{l=s_1}^{s_2} \mu^{|l-j|} a_{l} \\
&\hspace{8mm} + C_{\mu, \gamma}\ \varepsilon_0 \left( \mu^{|s_1-1-j|} a_{s_1-1} + \mu^{|s_1-2-j|} a_{s_1-2}
+ \mu^{|s_2+1-j|} a_{s_2+1} + \mu^{|s_2+2-j|} a_{s_2+2}\right),
\end{aligned}
\end{equation}
where we used the fact that $a_l=0$ for any $l \le s_1-3$ or $l \ge s_2+3$ coming from \eqref{EQ: wuinuifnvuUNIDWU0139jr1}.
By Theorem \ref{THM: L2 en quant of na u} 
\begin{equation}
\lim_{\eta\searrow0} \limsup_{k\rightarrow \infty}\ \varepsilon_0(\eta,\delta_k)=0.
\end{equation}
Hence, we can assume that $\et>0$ small enough and for $k \in \N$ large enough we have
\begin{equation}
C_{\mu,\gamma} \ \epv_0 < \frac{1}{2}.
\end{equation}
allowing to absorb the sum to the left-hand side
\begin{equation}
\begin{aligned}
&\sum_{l=s_1}^{s_2} \mu^{|l-j|} a_{l} \\
&\hspace{7mm} \le C \sum_{l=s_1}^{s_2} \mu^{|l-j|} b_{l} + C \ \varepsilon_0 \left( \mu^{j-s_1} a_{s_1-1} + \mu^{j-s_1} a_{s_1-2}
+ \mu^{s_2-j} a_{s_2+1} + \mu^{s_2-j} a_{s_2+2}\right), \\
&\hspace{7mm} \le C \sum_{l=s_1}^{s_2} \mu^{|l-j|} b_{l} + C \ \varepsilon_0 \left( \mu^{j-s_1} + \mu^{s_2-j}\right)\norm{\nabla u_k}_{L^2(A(\eta,\delta_k))}^2,
\end{aligned}
\end{equation}
where in the last line we used that for any $i$ one has $a_i \le \norm{\na\widetilde u_k}_{L^2(\C)}^2 \le C\norm{\nabla u_k}_{L^2(A(\eta,\delta_k))}^2$.
We introduce $\beta \coloneqq - \log_2 \mu\in (0,-\log_2 \gamma)\subset(0,1)$ such that
\begin{equation}
\mu = 2^{-\beta}.
\end{equation}
Now we focus on the bound of the expression
\begin{equation}
\begin{aligned}
&\sum_{l=s_1}^{s_2} \mu^{|l-j|} b_{l}\\
&\hspace{12mm} = c_0\sum_{l=s_1}^{s_2} \mu^{|l-j|} \left[ 2^{-2l(p_k-1)} + \gamma^l \norm{\nabla u_k}_{L^2(A(\eta,\delta_k))}^2 + \norm{\nabla \mathfrak h_{\eta,k}}_{L^{2}(A_l)}^2+ \norm{\nabla b_{\eta,k}}_{L^{p_k^\prime}(A_l)}^2\right].
\end{aligned}
\end{equation}
\textbf{1.)}
\begin{equation}
\begin{aligned}
\sum_{l=s_1}^{s_2} \mu^{|l-j|} 2^{-2l(p_k-1)}
&= \sum_{l=s_1}^{j} \mu^{j-l} 2^{-2l(p_k-1)}
+ \sum_{l=j+1}^{s_2} \mu^{l-j} 2^{-2l(p_k-1)} \\
&\le \mu^{j}\sum_{l=s_1}^{j} \mu^{-l} 2^{-2l}
+ \mu^{-j} \sum_{l=j+1}^{s_2} \mu^{l} 2^{-2l} \\
&\le \mu^{j}\sum_{l=s_1}^{j} \left(\frac{1}{4\mu}\right)^l
+ \mu^{-j} \sum_{l=j+1}^{s_2}  \left( \frac{\mu}{4}\right)^l \\
&\le \left( \frac{1}{4\mu}\right)^{s_1-1} \ \mu^{j}
+ \mu^{-j}  \left( \frac{\mu}{4}\right)^j \\
&= \left( \frac{1}{4\mu}\right)^{s_1-1} \ 2^{-\beta j}+ 2^{-2j} \\
&\le \left( \frac{1}{4\mu}\right)^{s_1-1} \ 2^{-\beta j} + 2^{-s_1} 2^{-j}.
\end{aligned}
\end{equation}
\textbf{2.)}
\begin{equation}
\begin{aligned}
\sum_{l=s_1}^{s_2} \mu^{\abs{l-j}} \gamma^l 
&= \mu^{j} \sum_{l=s_1}^{j} \left(\frac{\gamma}{\mu} \right)^l + \mu^{-j} \sum_{l=j+1}^{s_2} \left(\mu\ga\right)^l \\
&\le \mu^j \left(\frac{\gamma}{\mu} \right)^{s_1-1} + \mu^{-j} (\mu\ga)^{j} \\
&\le \mu^{j}  + \gamma^j \le 2\mu^j =2\ 2^{-\beta j}.
\end{aligned}
\end{equation}
\textbf{3.)}
With Lemma \ref{LEMMA: on the est of na h in necks} we get
\begin{equation}
\begin{aligned}
&\hspace{-10mm} \sum_{l=s_1}^{s_2} \mu^{|l-j|} \norm{\nabla \mathfrak h_{\eta,k}}_{L^{2}(A_l)}^2 \\
&\le C \left(\norm{\nabla u_k}_{L^2(A(\eta,\delta_k))}^2 + \norm{\nabla b_{\eta,k}}_{L^{p_k^\prime}(B_1)} \right) \sum_{l=s_1}^{s_2} \mu^{|l-j|} \left[ \left( \frac{2^{-l}}{\et}\right)^2 + \left( \frac{\de_k}{2^{-l}\et}\right)^2 \right] ,
\end{aligned}
\end{equation}
and compute
\begin{equation}
\begin{aligned}
&\hspace{-12mm}\sum_{l=s_1}^{s_2} \mu^{\abs{l-j}} \left[ \left( \frac{2^{-l}}{\et}\right)^2 + \left( \frac{\de_k}{2^{-l}\et}\right)^2 \right] \\
&= \frac{1}{\et^2} \left[ \sum_{l=s_1}^{j} \mu^{j}\left( \left( \frac{1}{4\mu} \right)^l + \left( \frac{4}{\mu} \right)^l \de_k^2 \right)
+  \sum_{l=j+1}^{s_2} \mu^{-j}\left( \left( \frac{\mu}{4} \right)^l + \left( 4\mu \right)^l \de_k^2 \right) \right] \\
& \le \frac{C}{\et^2} \left[ \mu^{j}\left( \left( \frac{1}{4\mu} \right)^{s_1-1} + \left( \frac{4}{\mu} \right)^{j+1} \de_k^2 \right)
+   \mu^{-j}\left( \left( \frac{\mu}{4} \right)^{j} + \left( 4\mu \right)^{s_2+1} \de_k^2 \right) \right] \\
& \le \frac{C}{\et^2} \Big( \mu^{j-s_1}\ 2^{-2s_1} + 2^{2j}\ \de_k^2 + 2^{-2j} + \mu^{s_2-j}\ 2^{2s_2}\ \de_k^2 \Big) \\
& = \frac{C}{\et^2} \Big(2^{(s_1-j)(-\log_2\mu)} \ 2^{-2s_1} + 2^{2j}\ \de_k^2 + 2^{-2j} + 2^{(j-s_2)(-\log_2 \mu)} \ 2^{2s_2}\ \de_k^2 \Big) \\
& = C \left[2^{(s_1-j)\be} \underbrace{\left( \frac{2^{-s_1}}{\et} \right)^2}_{\le C} + 2^{2j}\left( \frac{\de_k}{\et} \right)^2 + \left( \frac{2^{-j}}{\et} \right)^2 + 2^{(j-s_2)\be} \underbrace{\left( \frac{\de_k}{2^{-s_2}\et} \right)^2}_{\le C}  \right] \\
&\le C \left[ \left( \frac{2^{-j}}{\et} \right)^\be + \left( \frac{\de_k}{2^{-j}\et} \right)^2 + \left( \frac{2^{-j}}{\et} \right)^2 + \left( \frac{\de_k}{2^{-j}\et} \right)^\be \right] \\
&\le C \left[ \left( \frac{2^{-j}}{\et} \right)^\be + \left( \frac{\de_k}{2^{-j}\et} \right)^\be \right].
\end{aligned}
\end{equation}
where in the last line we used that $\be <2$, $\frac{2^{-j}}{2\et}\le 1$ and $\frac{\de_k}{2^{-j}\et}\le 1$.
\\ \textbf{4.)}
We bound using \eqref{EQ: HIninfwjei921ejir2}
\begin{equation}
\begin{aligned}
\label{EQ: 0921niwebfweijwefn8913}
\sum_{l=s_1}^{s_2} \mu^{|l-j|}  \norm{\nabla b_{\eta,k}}_{L^{p_k^\prime}(A_l)}^2
&= \sum_{l=s_1}^{s_2} \mu^{|l-j|}  \norm{\nabla b_{\eta,k}}_{L^{p_k^\prime}(A_l)}^{p_k^\prime} \norm{\nabla b_{\eta,k}}_{L^{p_k^\prime}(A_l)}^{2-p_k^\prime} \\
&\le \norm{\nabla b_{\eta,k}}_{L^{p_k^\prime}(A(\eta,\delta_k))}^{2-p_k^\prime} \sum_{l=s_1}^{s_2} \mu^{|l-j|}  \norm{\nabla b_{\eta,k}}_{L^{p_k^\prime}(A_l)}^{p_k^\prime} \\
&\le (C (p_k-2))^{2-p_k^\prime} \sum_{l=s_1}^{s_2}  \norm{\nabla b_{\eta,k}}_{L^{p_k^\prime}(A_l)}^{p_k^\prime} \\
&\le (C (p_k-2))^{2-p_k^\prime} \norm{\nabla b_{\eta,k}}_{L^{p_k^\prime}(A(\eta,\delta_k))}^{p_k^\prime},
\end{aligned}
\end{equation}
where in the last line we used additivity of the integral.
As $2-p_k^\prime=\frac{p_k-2}{p_k-1}$ we have
\begin{equation}
(C (p_k-2))^{2-p_k^\prime} 
\le C (p_k-2)^{2-p_k^\prime}
= C (\underbrace{(p_k-2)^{p_k-2}}_{\rightarrow 1})^\frac{1}{p_k-1}
\le C
\end{equation}
and also
\begin{equation}
\begin{aligned}
\norm{\nabla b_{\eta,k}}_{L^{p_k^\prime}(A(\eta,\delta_k))}^{p_k^\prime}
& \le \Big[\norm{\nabla b_{\eta,k}}_{L^{p_k^\prime}(A(\eta,\delta_k))} + (p_k-2)\Big]^{p_k^\prime} \\
&= \Big[\norm{\nabla b_{\eta,k}}_{L^{p_k^\prime}(A(\eta,\delta_k))} + (p_k-2)\Big]^{2} \ \Big[\norm{\nabla b_{\eta,k}}_{L^{p_k^\prime}(A(\eta,\delta_k))} + (p_k-2)\Big]^{\overbrace{p_k^\prime-2}^{<0}} \\
&\le C \Big[\norm{\nabla b_{\eta,k}}_{L^{p_k^\prime}(A(\eta,\delta_k))}^2 + (p_k-2)^2\Big] \ \underbrace{\Big[p_k-2\Big]^{p_k^\prime-2}}_{\rightarrow 1}.
\end{aligned}
\end{equation}
With \eqref{EQ: 0921niwebfweijwefn8913} we get
\begin{equation}
\begin{aligned}
\sum_{l=s_1}^{s_2} \mu^{|l-j|}  \norm{\nabla b_{\eta,k}}_{L^{p_k^\prime}(A_l)}^2
&\le C \Big[\norm{\nabla b_{\eta,k}}_{L^{p_k^\prime}(A(\eta,\delta_k))}^2 + (p_k-2)^2\Big] \\
&\eqqcolon \left(C_{\eta,k}\right)^2,
\end{aligned}
\end{equation}
where with Lemma \ref{LEMMA: Control of const of na b in necks} and Lemma \ref{LEMMA: Growth of  pk and ln etadelta} one has $\lim_{\eta\searrow0}\limsup_{k\rightarrow\infty}\ \log\left(\frac{\eta^2}{\delta_k}\right) C_{\eta,k}=0$. \\
Putting these bounds together we find
\begin{equation}
\label{EQ: INuhnfue91uqfwUNqw}
\begin{aligned}
\int_{A_j} \abs{\nabla u_k}^2 dx 
&=a_j \le \sum_{l=s_1}^{s_2} \mu^{|l-j|} a_{l} \\
&\le \left(\norm{\nabla u_k}_{L^2(A(\eta,\delta_k))}^2 + \norm{\nabla b_{\eta,k}}_{L^{p_k^\prime}(B_1)} \right) \Bigg[ \left( \frac{2^{-j}}{\et} \right)^\be + \left( \frac{\de_k}{2^{-j}\et} \right)^\be \Bigg] \\
&\hspace{35mm} + C \norm{\nabla u_k}_{L^2(A(\eta,\delta_k))}^2 \Bigg[ \mu^{j-s_1} + \mu^{s_2-j} + 2^{-\beta j} \Bigg] \\
&\hspace{45mm} + C \Bigg[ \left( \frac{1}{4\mu}\right)^{s_1} \ 2^{-\beta j} + 2^{-s_1} 2^{-j} + \left(C_{\eta,k}\right)^2 \Bigg].
\end{aligned}
\end{equation}
We have the following:
\\ \textbf{I:}
\begin{equation}
\mu^{j-s_1} =\left(2^{-\beta(j-s_1)} \right) = \left(2^{-j} 2^{s_1} \right)^\beta 
\le C \left( \frac{2^{-j}}{\eta} \right)^\beta
\end{equation}
\\ \textbf{II:}
\begin{equation}
\mu^{s_2-j} =\left(2^{-\beta(s_2-j)} \right) = \left(2^{j}\ 2^{-s_2} \right)^\beta 
\le C \left( \frac{\delta_k}{2^{-j}\eta} \right)^\beta 
\end{equation}
\\ \textbf{III:}
\begin{equation}
2^{-\beta j}
\le \left( \frac{2^{-j}}{\eta} \right)^\beta
\end{equation}
\\ \textbf{IV:}
\begin{equation}
\left( \frac{1}{4\mu}\right)^{s_1} \ 2^{-\beta j}
\le C\left( 2^{\beta-2}\right)^{-\log_2(\eta)} 2^{-\beta j}
\le C \eta^{2-\beta} 2^{-\beta j}
=  C \eta^{2} \left( \frac{2^{-j}}{\eta} \right)^\beta
\end{equation}
\\ \textbf{V:}
\begin{equation}
2^{-s_1} 2^{-j} \le C \ 2^{\log_2 \eta} 2^{-j}
\le C \ \eta \left( \frac{2^{-j}}{\eta} \right)
\le C \ \eta \left( \frac{2^{-j}}{\eta} \right)^\beta
\end{equation}
Going back to \eqref{EQ: INuhnfue91uqfwUNqw} we have found
\begin{equation}
\label{EQ: ijnvewiUN1298d1}
\begin{aligned}
&\int_{A_j} \abs{\nabla u_k}^2 dx \\
&\hspace{12mm}\le C \left( \norm{\nabla u_k}_{L^2(A(\eta,\delta_k))}^2+ \norm{\nabla b_{\eta,k}}_{L^{p_k^\prime}(B_1)} + \eta+\eta^2 \right)
\left( \left( \frac{2^{-j}}{\et} \right)^\be + \left( \frac{\de_k}{2^{-j}\et} \right)^\be \right) + \left(C_{\eta,k}\right)^2
\end{aligned}
\end{equation}
Let $x \in A_j$.
Put $r_x = \abs{x}/4$.
One has
\begin{equation}
B_{r_x}(x) \subset A_{j-1} \cup A_{j} \cup A_{j+1}.
\end{equation}
By $\varepsilon$-regularity Lemma \ref{LEMMA: epsilon reg sacks uhlenbeck for p harm} we can bound
\begin{equation}
\label{EQ: iwunfeiuNnquf39831dasdas}
\begin{aligned}
\abs{x}^2 \abs{\nabla u_k(x)}^2 
&= 2^6 \left(\frac{r_x}{2}\right)^2 \abs{\nabla u_k(x)}^2 \le 2^6 \left(\frac{r_x}{2}\right)^2 \norm{\nabla u_k}_{L^{\infty}(B_{r_x/2}(x))}^2 \le C \ \norm{\nabla u_k}_{L^2(B_{r_x}(x))}^2 \\
&\le C \left( \norm{\nabla u_k}_{L^2(A_{j-1})}^2 + \norm{\nabla u_k}_{L^2(A_{j})}^2 + \norm{\nabla u_k}_{L^2(A_{j+1})}^2 \right).
\end{aligned}
\end{equation}
Combining \eqref{EQ: ijnvewiUN1298d1} and \eqref{EQ: iwunfeiuNnquf39831dasdas} with the fact that $2^{-j-1}\le \abs{x}\le 2^{-j}$ we get.
\begin{equation}
\abs{x}^2 \abs{\nabla u_k(x)}^2 
\le C \left( \norm{\nabla u_k}_{L^2(A(\eta,\delta_k))}^2+ \norm{\nabla b_{\eta,k}}_{L^{p_k^\prime}(B_1)} + \eta+\eta^2 \right)
\left( \left( \frac{\abs{x}}{\et} \right)^\be + \left( \frac{\de_k}{\abs{x}\et} \right)^\be \right) + \left(C_{\eta,k}\right)^2
\end{equation}
With Theorem \ref{THM: L2 en quant of na u}, Lemma \ref{LEMMA: First est of nabla b in necks} and Lemma \ref{LEMMA: Control of const of na b in necks} for all $x \in A_j$ we have
\begin{equation}
\label{EQ: jwiefnuIN301nidqw}
\abs{x}^2 \abs{\nabla u_k(x)}^2 \le C\left[ \left( \frac{\abs{x}}{\et} \right)^\be + \left( \frac{\de_k}{\et \abs{x}} \right)^\be \right]\boldsymbol{\epsilon}_{\eta, \delta_k} +  \boldsymbol{\mathrm c}_{\eta, \delta_k},
\end{equation}
where
\begin{equation}
\lim_{\eta\searrow0}\limsup_{k\rightarrow\infty}\boldsymbol{\epsilon}_{\eta, \delta_k}=0,
\qquad \text{and} \qquad
\lim_{\eta\searrow0}\limsup_{k\rightarrow\infty}\boldsymbol{\mathrm c}_{\eta, \delta_k} \log^2\left(\frac{\eta^2}{\delta_k}\right) =0.
\end{equation}
Let now $x\in A\left(\frac{\eta}{4},\delta_k\right)$.
Then we can find some $j\in \N$ such that $2^{-j-1}\le \abs{x}\le 2^{-j}$.
But then $\frac{4\delta_k}{\eta} \le \abs{x} \le 2^{-j} \le 2 \abs{x} \le \frac{\eta}{2}$.
Therefore $j \in \{s_1, \dots,s_2 \}$ and estimate \eqref{EQ: jwiefnuIN301nidqw} is valid for $x\in A\left(\frac{\eta}{4},\delta_k\right)$.
This completes the proof.
\end{proof}

\newpage
\ \vspace{15mm}
\section{Stability of the Morse Index}

In this section we finally show the upper semicontinuity of the Morse index plus nullity for Sacks-Uhlenbeck sequences to spheres, more precisely
\begin{theorem}
\label{THM: usc morse index sacks-uhlenbeck with one bubble}
For $k\in \N $ large there holds
\begin{equation}
\operatorname{Ind}_{E_{p_k}}(u_k) + \operatorname{Null}_{E_{p_k}}(u_k)
\le \operatorname{Ind}_{E}(u_\infty) + \operatorname{Null}_{E}(u_\infty) + \operatorname{Ind}_{E}(v_\infty) + \operatorname{Null}_{E}(  v_\infty)
\end{equation}
\end{theorem}
We adapt the strategy introduced in \cite{DGR22}.
Let us briefly explain what this is. 
First, we show that the necks are not contributing to the negativity of the second variation. This we do by combining the pointwise control as in estimate \eqref{EQ: 283hfunv0lpdnu13ef6zh1} and a weighted Poincar\`e inequality (Lemma \ref{LEMMA: Lemma on neg in the necks Weighted Poincare}).
Second, we use Sylvester's law of inertia to change to a different measure incorporating the weights obtained in estimate \eqref{EQ: 283hfunv0lpdnu13ef6zh1}.
Finally, we apply spectral theory to the Jacobi operator of the second variation.
The result follows by combining these techniques.

\subsection{Positive contribution of the Necks}

In this section we prove that any variation supported in the neck region evaluates positively in the quadratic form.
More concrete:

\begin{theorem}
\label{THM: JDNW892fejeujfwef32}
For every $\be \in \left(0,\log_2 (3/2) \right)$ there exists some constant $\overline\kappa>0$ such that for $k\in \N$ large and $\eta >0$ small one has
\begin{equation}
\forall w \in V_{u_k} : \
(w=0 \text{ in } \Sigma \setminus A(\eta,\delta_k))
\Rightarrow
Q_{u_k}(w) \ge \overline\kappa \int_\Sigma \abs{w}^2 \omega_{\eta,k} \ dvol_h \ge 0,
\end{equation} 
where the weight function is given by
\begin{equation}
\omega_{\eta,k}=
\begin{cases}
\frac{1}{\abs{x}^2} \left[ \frac{\abs{x}^\beta}{\eta^\beta} + \frac{\delta_k^\beta}{\eta^\beta \abs{x}^\beta} + \frac{1}{\log^2\left(\frac{\eta^2}{\delta_k} \right)} \right] &\text{ if } x \in A(\eta,\delta_k), \\
\frac{1}{\eta^2} \left[1+  \frac{\delta_k^\beta}{\eta^{2\beta}} + \frac{1}{\log^2\left(\frac{\eta^2}{\delta_k} \right)} \right] &\text{ if } x \in \Sigma \setminus B_\eta, \\
\frac{\eta^2}{\delta_k^2} \left[\frac{(1+\eta^2)^2}{\eta^4\left(1+\delta_k^{-2} \abs{x}^2\right)^2} +  \frac{\delta_k^\beta}{\eta^{2\beta}} + \frac{1}{\log^2\left(\frac{\eta^2}{\delta_k} \right)} \right] &\text{ if } x \in B_{\delta_k/\eta}.
\end{cases}
\end{equation}
\end{theorem}

\begin{proof}
Let $w \in V_{u_k}$ with $w=0 \text{ in } \Sigma \setminus A(\eta,\delta_k))$.
Let $c_0>0$ be as in Lemma \ref{LEMMA: Lemma on neg in the necks Weighted Poincare}.
Let $\boldsymbol{\epsilon}_{\eta, \delta_k}$ and $\boldsymbol{\mathrm c}_{\eta, \delta_k}$ be as in Theorem \ref{THEOREM: on Pointwise Estimate of the Gradient in the Necks p harm} and we may assume with Lemma \ref{LEMMA: Linfty bd on na u in the necks unif}  that for $k\in\N$ large and $\eta>0$ small
\begin{equation}
\label{EQ: iwunfeUNIUDW8912ehn912}
\begin{aligned}
\frac{c_0}{3} - \boldsymbol{\epsilon}_{\eta, \delta_k} 3(1+\norm{\nabla u_k}_{L^\infty(\Sigma)}^2)^{\frac{p_k}{2}-1}  
&\ge \frac{c_0}{4} \eqqcolon \overline\kappa \\
\frac{c_0}{3} -\boldsymbol{\mathrm c}_{\eta, \delta_k} \log^2\left(\frac{\eta^2}{\delta_k} \right) 3(1+\norm{\nabla u_k}_{L^\infty(\Sigma)}^2)^{\frac{p_k}{2}-1}
&\ge \overline\kappa.
\end{aligned}
\end{equation}
Using Definition \ref{DEFINITION: Morse index and Nullityof p-harmonic maps}, Theorem \ref{THEOREM: on Pointwise Estimate of the Gradient in the Necks p harm}, Lemma \ref{LEMMA: Lemma on neg in the necks Weighted Poincare} and \eqref{EQ: iwunfeUNIUDW8912ehn912} we find for large $k$
\begin{equation}
\begin{aligned}
& Q_{u_k}(w) \\
&\hspace{2mm} = \underbrace{p_k\ (p_k-2) \int_\Sigma \left( 1+\abs{\nabla u_k}^2\right)^{p_k/2-2} \left(\nabla u_k \cdot \nabla w \right)^2  \ dvol_\Si}_{\ge 0} \\
&\hspace{20mm}+ \underbrace{p_k}_{1 \le\abs{\cdot} \le 3} \int_\Sigma \underbrace{\left( 1+\abs{\nabla u_k}^2\right)^{p_k/2-1}}_{1 \le \abs{\cdot}\le (1+\norm{\nabla u_k}_{L^\infty(\Sigma)}^2)^{\frac{p_k}{2}-1}} \left[\abs{\nabla w}^2 - \abs{\nabla u_k}^2 \abs{w}^2 \right] \ dvol_\Si \\
&\hspace{2mm} \ge \int_\Sigma \abs{\nabla w}^2 - \abs{\nabla u_k}^2 \abs{w}^2 3(1+\norm{\nabla u_k}_{L^\infty(\Sigma)}^2)^{\frac{p_k}{2}-1} \  dvol_\Si \\
&\hspace{2mm}= \int_{A(\eta,\delta_k))} \abs{\nabla w}^2 -  \abs{ w}^2 \abs{\nabla u_k}^2 3(1+\norm{\nabla u_k}_{L^\infty(\Sigma)}^2)^{\frac{p_k}{2}-1} \ dx \\
&\hspace{2mm}\ge \int_{A(\eta,\delta_k))} \abs{\nabla w}^2 - \frac{\abs{ w}^2}{\abs{x}^2} \Bigg[ \left( \left( \frac{\abs{x}}{\et} \right)^\beta + \left( \frac{\de_k}{\et \abs{x}} \right)^\beta \right)  \boldsymbol{\epsilon}_{\eta, \delta_k} + \boldsymbol{\mathrm c}_{\eta, \delta_k} \Bigg] 3(1+\norm{\nabla u_k}_{L^\infty(\Sigma)}^2)^{\frac{p_k}{2}-1} dx \\
&\hspace{2mm}\ge \int_{A(\eta,\delta_k))} \frac{\abs{ w}^2}{\abs{x}^2} \Bigg[ \left( \left( \frac{\abs{x}}{\et} \right)^\beta + \left( \frac{\de_k}{\et \abs{x}} \right)^\beta \right) \underbrace{\left(\frac{c_0}{3} - \boldsymbol{\epsilon}_{\eta, \delta_k}3(1+\norm{\nabla u_k}_{L^\infty(\Sigma)}^2)^{\frac{p_k}{2}-1} \right)}_{\ge \overline\kappa} \\
 &\hspace{35mm}+ \underbrace{\left( \frac{c_0}{3} -\boldsymbol{\mathrm c}_{\eta, \delta_k} \log^2\left(\frac{\eta^2}{\delta_k} \right) 3(1+\norm{\nabla u_k}_{L^\infty(\Sigma)}^2)^{\frac{p_k}{2}-1} \right)}_{\ge \overline\kappa} \frac{1}{\log^2\left(\frac{\eta^2}{\delta_k} \right)} \Bigg] dx \\
&\hspace{2mm}\ge \overline\kappa \int_{A(\eta,\delta_k))} \frac{\abs{ w}^2}{\abs{x}^2} \Bigg[ \left( \frac{\abs{x}}{\et} \right)^\beta + \left( \frac{\de_k}{\et \abs{x}} \right)^\beta + \frac{1}{\log^2\left(\frac{\eta^2}{\delta_k} \right)} \Bigg] dx.
\end{aligned}
\end{equation}
\end{proof}

\subsection{The Diagonalization of $Q_{u_k}$ with respect to the Weights $\omega_{\eta,k}$}

Consider the inner product
\begin{equation}
\langle w,v \rangle_{\omega_{\eta,k}} \coloneqq \int_{\Sigma} w \cdot v \ \omega_{\eta,k} \ dvol_{\Sigma}.
\end{equation}
\noindent
Then we look for the self-adjoint operator with respect to $\langle \cdot,\cdot \rangle_{\omega_{\eta,k}}$ of the quadratic form $Q_{u_k}$ and compute integrating by parts
\begin{equation}
\begin{aligned}
Q_{u_k}(w) 
&= p_k\ (p_k-2) \int_\Sigma \left( 1+\abs{\nabla u_k}^2\right)^{p_k/2-2} \left(\nabla u_k \cdot \nabla w \right)^2  \ dvol_\Sigma \\
&\hspace{20mm}+p_k \int_\Sigma \left( 1+\abs{\nabla u_k}^2\right)^{p_k/2-1} \left[\abs{\nabla w}^2 - \abs{\nabla u_k}^2 \abs{w}^2 \right] \ dvol_\Sigma \\
&= \int_\Sigma \Bigg[ -p_k\ (p_k-2) \divergence \left( \left( 1+\abs{\nabla u_k}^2\right)^{p_k/2-2} \left(\nabla u_k \cdot \nabla w \right)\ \nabla u_k \right)  \Bigg] \cdot w \ dvol_\Sigma \\
&\hspace{20mm}+ \int_\Sigma \Bigg[ - p_k \divergence\left( \left( 1+\abs{\nabla u_k}^2\right)^{p_k/2-1} \nabla w \right) \Bigg] \cdot w \ dvol_\Sigma \\
&\hspace{40mm}+ \int_\Sigma \Bigg[- p_k \left( 1+\abs{\nabla u_k}^2\right)^{p_k/2-1} \abs{\nabla u_k}^2 w\Bigg] \cdot w \ dvol_\Sigma
\end{aligned}
\end{equation}
Let us introduce the operator
\begin{equation}
\begin{aligned}
\mathcal L_{\eta,k}(w) 
&\coloneqq \omega_{\eta,k}^{-1}\ P_{u_k} \Bigg[ \left(-p_k\ (p_k-2) \divergence \left( \left( 1+\abs{\nabla u_k}^2\right)^{p_k/2-2} \left(\nabla u_k \cdot \nabla w \right)\ \nabla u_k \right) \right) \\
&\hspace{10mm} - p_k \divergence\left( \left( 1+\abs{\nabla u_k}^2\right)^{p_k/2-1} \nabla w \right) - p_k \left( 1+\abs{\nabla u_k}^2\right)^{p_k/2-1} \abs{\nabla u_k}^2 w \Bigg],
\end{aligned}
\end{equation}
where $P_{u_k}(x): \R^{n+1}\rightarrow T_{u_k(x)} S^n$ is the orthogonal projection.
Then we have found the formula
\begin{equation}
\label{EQ: iwubfnIUBPN239fn321}
Q_{u_k}(w)= \langle \mathcal L_{\eta,k} w,w \rangle_{\omega_{\eta,k}}. 
\end{equation}
Note that by construction $\mathcal L_{\eta,k}$ is self-adjoint with respect to the inner product $\langle \cdot ,\cdot \rangle_{\omega_{\eta,k}}$, i.e.
\begin{equation}\label{EQ: wijenfuNIUD138en9qd}
\langle \mathcal L_{\eta,k} w,v \rangle_{\omega_{\eta,k}} = \langle w, \mathcal L_{\eta,k} v \rangle_{\omega_{\eta,k}} .
\end{equation}
Recall the definition of $V_{u_k}$ in \eqref{EQ: wigunUBNI8193hf23f} and consider also the larger space
\begin{equation}
U_{u_k} = \left\{ w \in L^{2}_{\omega_{\eta,k}}(\Sigma; \R^{n+1}) \forwhich w(x) \in T_{u_k(x)} S^n, \quad \text{for a.e. } x\in\Sigma \right\}.
\end{equation}

\begin{lemma}[Spectral Decomposition]
\label{LEMMA: Spectral decomp of L}
There exists a Hilbert basis of the space $(U_{u_k}, \langle\cdot,\cdot\rangle_{\omega_{\eta,k}})$ of eigenfunctions of the operator $\mathcal L_{\eta,k}$ and the eigenvalues of $\mathcal L_{\eta,k}$ satisfy 
\begin{equation}
\lambda_1<\lambda_2<\lambda_3\ldots \rightarrow\infty.
\end{equation}
Furthermore, one has the orthogonal decomposition 
\begin{equation}
U_{u_k} = \bigoplus_{\lambda \in \Lambda_{\eta,k}}\mathcal E_{\eta,k}(\lambda),
\end{equation}
where 
\begin{equation}
\mathcal E_{\eta,k}(\lambda) \coloneqq \{ w \in V_{u_k} \ ; \ \mathcal L_{\eta,k}(w)=\lambda w \},
\qquad \Lambda_{\eta,k} \coloneqq\big\{ \lambda\in \R \ ; \ \mathcal E_{\eta,k}(\lambda) \setminus\{0\}  \neq \emptyset \big\}
\end{equation}
\end{lemma}

\begin{proof}
By the identity \eqref{EQ: wijenfuNIUD138en9qd} it is clear that $\mathcal L_{\eta,k}w$ has distributional meaning for every $w\in U_{u_k}$.
For $\lambda\in \R$ introduce
\begin{equation}
\mathcal J _\lambda \coloneqq \mathcal L_{\eta,k}+\lambda\ id
\end{equation}
\textbf{Claim 1:}
There exists some $\lambda_0>0$ (which may depend on $u_k$, $\eta$ or $\delta_k$) such that for any $\lambda\ge \lambda_0$ there holds
\begin{equation}
\begin{cases}
\displaystyle \forall f \in U_{u_k}:\ \exists! w \in V_{u_k}:\ \mathcal J _\lambda w=f, \\[2mm]
\displaystyle \int_{\Sigma} \left( \abs{\nabla w}^2 + \abs{w}^2 \omega_{\eta,k} \right) dvol_{\Sigma} \le 4 \norm{f}_{L^2_{\omega_{\eta,k}}(\Sigma)}^2. 
\end{cases}
\end{equation}
\textbf{Proof of Claim 1}
We start by establishing a-priori estimates.
Multiply the equation $\mathcal J _\lambda w=f$ by $w$ and integrating
\begin{equation}
\label{EQ: IBNNinwf29nUN13e}
\langle \mathcal J_{\lambda} w,w \rangle_{\omega_{\eta,k}}
= \langle\mathcal L_{\eta,k} w,w \rangle_{\omega_{\eta,k}} + \lambda \langle w,w\rangle_{\omega_{\eta,k}} 
= \langle f,w\rangle_{\omega_{\eta,k}}
\le \norm{f}_{L^2_{\omega_{\eta,k}}(\Sigma)} \norm{w}_{L^2_{\omega_{\eta,k}}(\Sigma)}.
\end{equation}
Now
\begin{equation}
\label{EQ: inwernIJNUEN203n233dfSW}
\begin{aligned}
\langle \mathcal J_{\lambda} w,w \rangle_{\omega_{\eta,k}}
&=Q_{u_k}(w) + \lambda \langle w,w\rangle_{\omega_{\eta,k}} \\
&= \underbrace{p_k\ (p_k-2) \int_\Sigma \left( 1+\abs{\nabla u_k}^2\right)^{p_k/2-2} \left(\nabla u_k \cdot \nabla w \right)^2  \ dvol_\Sigma}_{\ge 0} + \lambda \langle w,w\rangle_{\omega_{\eta,k}} \\
&\hspace{20mm}+\underbrace{p_k}_{1 \le\abs{\cdot} \le 3} \int_\Sigma \underbrace{\left( 1+\abs{\nabla u_k}^2\right)^{p_k/2-1}}_{1 \le \abs{\cdot}\le (1+\norm{\nabla u_k}_{L^\infty(\Sigma)}^2)^{\frac{p_k}{2}-1}} \left[\abs{\nabla w}^2 - \abs{\nabla u_k}^2 \abs{w}^2 \right] \ dvol_\Sigma  \\
& \ge \int_\Sigma \abs{\nabla w}^2 dvol_\Sigma + \int_\Sigma \abs{w}^2 \left(\lambda_0\ \omega_{\eta,k} - \norm{\nabla u_k}_{L^\infty(\Sigma)} 3(1+\norm{\nabla u_k}_{L^\infty(\Sigma)}^2)^{\frac{p_k}{2}-1} \right)  \ dvol_\Sigma
\end{aligned}
\end{equation}
Note that
\begin{equation}
\label{EQ: uifENNUIFENU2I389RH321}
\inf_{\Sigma} \omega_{\eta,k} >0,
\end{equation}
where the infimum may depend on $\eta$ and $\delta_k$.
So if we choose
\begin{equation}
\lambda_0 \coloneqq \frac{\norm{\nabla u_k}_{L^\infty(\Sigma)} 3(1+\norm{\nabla u_k}_{L^\infty(\Sigma)}^2)^{\frac{p_k}{2}-1}}{\inf_{\Sigma} \omega_{\eta,k}}+1>0
\end{equation}
we find with \eqref{EQ: IBNNinwf29nUN13e} and \eqref{EQ: inwernIJNUEN203n233dfSW} that
\begin{equation}
\label{EQ: weoufinINDW10rj3e1jWD}
\int_\Sigma \abs{\nabla w}^2 dvol_\Sigma + \int_\Sigma \abs{w}^2  \omega_{\eta,k}\ dvol_\Sigma
\le \langle \mathcal J_{\lambda} w,w \rangle_{\omega_{\eta,k}}
\le \norm{f}_{L^2_{\omega_{\eta,k}}(\Sigma)} \norm{w}_{L^2_{\omega_{\eta,k}}(\Sigma)}
\end{equation}
Using the  inequality $\norm{f}_{L^2_{\omega_{\eta,k}}(\Sigma)} \norm{w}_{L^2_{\omega_{\eta,k}}(\Sigma)} \le 2 \norm{f}_{L^2_{\omega_{\eta,k}}(\Sigma)}^2 + \frac{1}{2} \norm{w}_{L^2_{\omega_{\eta,k}}(\Sigma)}^2$ there follows
\begin{equation}
\int_{\Sigma} \left( \abs{\nabla w}^2 + \abs{w}^2 \omega_{\eta,k} \right) dvol_{\Sigma} \le 4 \norm{f}_{L^2_{\omega_{\eta,k}}(\Sigma)}^2.
\end{equation}
For the existence consider the functional
\begin{equation}
J_\lambda(w) \coloneqq Q_{u_k}(w)+ \int_{\Sigma}w \cdot (\lambda\ w -2f)\ \omega_{\eta,k} \ dvol_{\Sigma}
\end{equation}
and bound with \eqref{EQ: weoufinINDW10rj3e1jWD}
\begin{equation}
\begin{aligned}
J_\lambda(w) 
&= \langle \mathcal J_{\lambda} w,w \rangle_{\omega_{\eta,k}} -2 \int_\Sigma w \cdot f \ \omega_{\eta,k} \ dvol_{\Sigma} \\
&\ge \int_\Sigma \abs{\nabla w}^2 dvol_\Sigma + \int_\Sigma \abs{w}^2  \omega_{\eta,k}\ dvol_\Sigma -2 \int_\Sigma w \cdot f \ \omega_{\eta,k} \ dvol_{\Sigma} \\
&\ge \int_\Sigma \abs{\nabla w}^2 dvol_\Sigma + \frac{1}{2} \int_\Sigma \abs{w}^2  \omega_{\eta,k}\ dvol_\Sigma -2 \int_\Sigma \abs{f}^2  \omega_{\eta,k} \ dvol_{\Sigma},
\end{aligned}
\end{equation}
where in the last line we used that
\begin{equation}
2 \int_\Sigma w \cdot f \ \omega_{\eta,k} \ dvol_{\Sigma}
\le \frac{1}{2}  \int_\Sigma \abs{w}^2  \omega_{\eta,k} \ dvol_{\Sigma}
+ 2 \int_\Sigma \abs{f}^2  \omega_{\eta,k} \ dvol_{\Sigma}.
\end{equation}
Existence and uniqueness of minimizers $w$ to the functional $J_\lambda$ follows with classical Euler-Lagrange variational theory as in \cite{DGR22} appendix section B.
This completes the proof of Claim 1.
Rellich-Kondrachov together with Claim 1 gives that $\mathcal J_{\lambda}: U_{u_k} \rightarrow U_{u_k}$ is a compact linear operator on the separable Hilbert space $(U_{u_k}, \langle\cdot,\cdot\rangle_{\omega_{\eta,k}})$.
The spectral decomposition follows from theorem 6.11 in \cite{B11}.
\end{proof}

\begin{lemma}[Sylvester Law of Inertia]
\label{LEMMA: Ind plus Null eq dim of eigenspaces}
\begin{equation}
\operatorname{Ind}(u_k) + \operatorname{Null}(u_k) = \operatorname{dim}\left( \bigoplus_{\lambda\le 0} \mathcal E_{\eta,k}(\lambda) \right)
\end{equation}
\end{lemma}

\begin{proof}
By \eqref{EQ: iwubfnIUBPN239fn321} it is clear that $\operatorname{Null}(u_k) = \operatorname{dim} \left(\mathcal E_{\eta,k}(0) \right)$. 
In the following we show the equality $\operatorname{Ind}(u_k) = \operatorname{dim}\left( \oplus_{\lambda< 0} \mathcal E_{\eta,k}(\lambda) \right)$
\\
"$\ge$": Let $W\coloneqq \bigoplus_{\lambda<0} \mathcal E_{\eta,k}(\lambda)$.
Then $W \subset V_{u_k}$ and clearly $Q_{u_k}|_{W}<0$. Hence, $\operatorname{Ind}(u_k)\ge \operatorname{dim}(W)$.
"$\le$": 
Let $W \subset V_{u_k}$ be such that $Q_{u_k}|_{W}<0$. Then by Lemma \ref{LEMMA: Spectral decomp of L}
\begin{equation}
 V_{u_k}= \bigoplus_{\lambda \in \Lambda_{\eta,k}} \mathcal E_{\eta,k}(\lambda).
\end{equation}
Clearly,
\begin{equation}
W\cap\bigoplus_{\lambda\ge 0} \mathcal E_{\eta,k}(\lambda)= \{0\}.
\end{equation}
From the two equations above there follows $W \subset \bigoplus_{\lambda< 0} \mathcal E_{\eta,k}(\lambda)$. This shows the desired inequality.
\end{proof}

Set
\begin{equation}
\mu_{\eta,k}\coloneqq \norm{\frac{\abs{\nabla u_k}^2}{\omega_{\eta,k}}}_{L^\infty(\Sigma)}.
\end{equation}
Then one has

\begin{lemma}
\label{LEMMA: prop on the seq mu eta k}
\begin{equation}
\begin{aligned}
&\hspace{-6mm} \exists \eta_0>0: \ \exists k_0>0: \ \exists C>0: \\
&i)\ \mu_0\coloneqq \sup_{\eta\in(0,\eta_0)} \sup_{k\ge k_0} \mu_{\eta,k}< \infty,  \\
&ii)\ \lim_{\eta \searrow0} \limsup_{k \rightarrow \infty} \mu_{\eta,k}=0, \\
&iii)\ \forall\eta \in(0,\eta_0): \forall k \ge k_0: \inf \Lambda_{\eta,k} \ge -C\ \mu_{\eta,k}\ge-C\ \mu_0.
\end{aligned}
\end{equation}
\end{lemma}

\begin{proof}
\textit{i) \& ii):}
We decompose
\begin{equation}
\label{EQ: 2iufnujfnuiwseUJN}
\mu_{\eta,k} \le \norm{\frac{\abs{\nabla u_k}^2}{\omega_{\eta,k}}}_{L^\infty(\Sigma\setminus B_{\eta} )} + \norm{\frac{\abs{\nabla u_k}^2}{\omega_{\eta,k}}}_{L^\infty(B_\eta\setminus B_{\delta_k/\eta} )} + \norm{\frac{\abs{\nabla u_k}^2}{\omega_{\eta,k}}}_{L^\infty(B_{\delta_k/\eta} ))}
\end{equation}
 Note that by Theorem \ref{THEOREM: on Pointwise Estimate of the Gradient in the Necks p harm} we have
\begin{equation}
\label{EQ: i29jtgjijsgrisdfNJD}
\norm{\frac{\abs{\nabla u_k}^2}{\omega_{\eta,k}}}_{L^\infty(B_\eta\setminus B_{\delta_k/\eta} )} \leq \boldsymbol{\epsilon}_{\eta, \delta_k} + \boldsymbol{\mathrm c}_{\eta, \delta_k} \log^2\left(\frac{\eta^2}{\delta_k}\right)\longrightarrow0,
\end{equation}
as $k\rightarrow\infty,\eta\searrow0$.
Furthermore
\begin{equation}
\label{EQ: 93j4ugj3uNJUSDJU}
\norm{\frac{\abs{\nabla u_k}^2}{\omega_{\eta,k}}}_{L^\infty(\Sigma\setminus B_{\eta} )}
=\norm{\frac{\abs{\nabla u_k}^2}{\eta^{-2}\left(1+\frac{\delta_k^\beta}{\eta^{2\beta}}+ \frac{1}{\log^2\left(\frac{\eta^2}{\delta^k} \right)}\right)}}_{L^\infty(\Sigma\setminus B_{\eta} )}
\le \eta^2 \norm{\nabla u_k}_{L^\infty(\Sigma\setminus B_{\eta} )}
\end{equation}
Note that
\begin{equation}
\label{EQ: oi2fj2gfonJNJDW}
\lim_{\eta \searrow0} \limsup_{k \rightarrow \infty}\
\eta^2 \norm{\nabla u_k}_{L^\infty(\Sigma\setminus B_{\eta} )}
= \lim_{\eta \searrow0}\
\eta^2 \norm{\nabla u_\infty}_{L^\infty(\Sigma\setminus B_{\eta} )}=0.
\end{equation}
Recall that due to the point removability theorem and the stereographic projection one has that
\begin{equation}
\label{EQ: uiwpgjnpi324gnu2f3njJUNDW}
\abs{\nabla v_\infty (y)}^2 \le C\frac{1}{\left(1+\abs{y}^2\right)^2}.
\end{equation}
For $x \in B_{\delta_k/\eta}$ we estimate
\begin{equation}
\label{EQ: NJUWij2f2ifjm2f12}
\begin{aligned}
\frac{\abs{\nabla u_k(x)}^2}{\omega_{\eta,k}(x)}
&= \frac{\abs{\nabla u_k(x)}^2}{\frac{\eta^2}{\delta_k^2} \left[\frac{(1+\eta^2)^2}{\eta^4\left(1+\delta_k^{-2} \abs{x}^2\right)^2} +  \frac{\delta_k^\beta}{\eta^{2\beta}}+ \frac{1}{\log^2\left(\frac{\eta^2}{\delta^k} \right)} \right] } \\
&\le \frac{\abs{\nabla u_k(x)}^2}{\frac{\eta^2}{\delta_k^2} \left[\frac{(1+\eta^2)^2}{\eta^4\left(1+\delta_k^{-2} \abs{x}^2\right)^2} \right] } \\
&\le \frac{\eta^2}{(1+\eta^2)^2} \delta_k^2 \abs{\nabla u_k(x)}^2 \left(1+\delta_k^{-2} \abs{x}^2\right)^2 \\
&= \frac{\eta^2}{(1+\eta^2)^2} \norm{\abs{\nabla v_k(y)}^2 \left(1+ \abs{y}^2\right)^2}_{L^\infty(B_{1/\eta})}.
\end{aligned}
\end{equation}
Note that due to uniform convergence 
\begin{equation}
\limsup_{k \rightarrow \infty} \norm{\abs{\nabla v_k(y)}^2 \left(1+ \abs{y}^2\right)^2}_{L^\infty(B_{1/\eta})}
= \norm{\abs{\nabla v_\infty(y)}^2 \left(1+ \abs{y}^2\right)^2}_{L^\infty(B_{1/\eta})} \le C,
\end{equation}
where we used \eqref{EQ: uiwpgjnpi324gnu2f3njJUNDW} in the last inequality. 
Going back to \eqref{EQ: NJUWij2f2ifjm2f12} this allows to finally get
\begin{equation}
\label{EQ: 2j3mi2g4njnJNJDW}
\lim_{\eta\searrow0}\limsup_{k \rightarrow \infty}\norm{\frac{\abs{\nabla u_k}^2}{\omega_{\eta,k}}}_{L^\infty(B_{\delta_k/\eta} ))}
\le C\lim_{\eta\searrow0}\frac{\eta^2}{(1+\eta^2)^2}=0.
\end{equation}
Going back to \eqref{EQ: 2iufnujfnuiwseUJN} and combining \eqref{EQ: i29jtgjijsgrisdfNJD}, \eqref{EQ: 93j4ugj3uNJUSDJU}, \eqref{EQ: oi2fj2gfonJNJDW}, \eqref{EQ: 2j3mi2g4njnJNJDW} we conclude \textit{i)} and \textit{ii)}.

\noindent
\textit{iii):} 
Let $\lambda \in \Lambda_{\eta,k}$.
Then there exists an eigenvector $0\ne w \in V_{u_k}$ of $\mathcal L_{\eta,k}$ corresponding to the eigenvalue $\lambda$, i.e. $\mathcal L_{\eta,k}(w)=\lambda w$.
We get
\begin{equation}
\begin{aligned}
\lambda \langle w,w \rangle_{\omega_{\eta,k}}
&=  \langle \mathcal L_{\eta,k}(w),w \rangle_{\omega_{\eta,k}} =Q_{u_k}(w) \\
&= \underbrace{p_k\ (p_k-2) \int_\Sigma \left( 1+\abs{\nabla u_k}^2\right)^{p_k/2-2} \left(\nabla u_k \cdot \nabla w \right)^2  \ dvol_\Sigma}_{\ge 0}  \\
&\hspace{20mm}+p_k \int_\Sigma \left( 1+\abs{\nabla u_k}^2\right)^{p_k/2-1} \left[\abs{\nabla w}^2 - \abs{\nabla u_k}^2 \abs{w}^2 \right] \ dvol_\Sigma \\
&\ge  -\int_\Sigma p_k\left( 1+\abs{\nabla u_k}^2\right)^{p_k/2-1} \abs{\nabla u_k}^2 \abs{w}^2 \ dvol_\Sigma
\end{aligned}
\end{equation}
By Lemma \ref{LEMMA: Linfty bd on na u in the necks unif} one has that for large $k$ one has
\begin{equation}
p_k\left( 1+\abs{\nabla u_k}^2\right)^{p_k/2-1}  \le  C
\end{equation}
and hence
\begin{equation}
\begin{aligned}
\lambda \langle w,w \rangle_{\omega_{\eta,k}} 
&\ge -C \int_\Sigma   \abs{\nabla u_k}^2 \abs{w}^2 \ dvol_\Sigma \\
&\ge -C\ \mu_{\eta,k} \int_{\Sigma} \abs{w}^2 \omega_{\eta,k} \ dvol_{\Sigma} \\
&= - C\ \mu_{\eta,k} \langle w,w \rangle_{\omega_{\eta,k}}.
\end{aligned}
\end{equation}
This completes the proof of the lemma.
\end{proof}

In the following we focus on the limiting maps $u_\infty:\Sigma\rightarrow S^n$ and $v_\infty: \C \rightarrow S^n$ as appearing in Definition \ref{Definition: Bubble tree convergence of p harm map with one bubble}.
We proceed analogous to \cite{DGR22}.
We compute for $w\in V_{u_\infty}$ integrating by parts
\begin{equation}
\begin{aligned}
Q_{u_\infty}(w) 
&= 2 \int_{\Sigma} \abs{\nabla w}^2 - \abs{\nabla u_\infty}^2 \abs{w}^2 dvol_\Sigma \\
&=2 \int_{\Sigma} \left( -\Delta w - w \abs{\nabla u_\infty}^2 \right) \cdot w\ dvol_\Sigma.
\end{aligned}
\end{equation}
Note that for any fixed $\eta>0$ we have the pointwise limit
\begin{equation}
\omega_{\eta,k}(x) \rightarrow \omega_{\eta,\infty}(x) \coloneqq \begin{cases}
\frac{1}{\eta^2}, &\text{ if } x\in \Sigma \setminus B_{\eta} \\
\frac{1}{\eta^\beta \abs{x}^{2-\beta}}, &\text{ if } x\in B_{\eta}
\end{cases}, \qquad \text{ as } k \rightarrow \infty.
\end{equation}
We introduce 
\begin{equation}
\mathcal L_{\eta,\infty}: V_{u_{\infty}} \rightarrow V_{u_{\infty}};
\quad\mathcal L_{\eta,\infty}(w) \coloneqq 2\ P_{u_\infty} \left(\omega_{\eta,\infty}^{-1}(- \Delta w  -w \abs{\nabla u_\infty}^2) \right), 
\end{equation}
such that
\begin{equation}
Q_{u_\infty}(w)= \langle \mathcal L_{\eta,\infty} w, w \rangle_{\omega_{\eta,\infty}},
\end{equation}
where we used 
\begin{equation}
\langle w,v \rangle_{\omega_{\eta,\infty}} \coloneqq \int_{\Sigma} w \cdot v \ \omega_{\eta,\infty} \ dvol_{\Sigma}.
\end{equation}
As above a simple integration by parts shows that
\begin{equation}
\begin{aligned}
Q_{v_\infty}(w)
&= 2\int_\C \left(- \Delta w  -w \abs{\nabla v_\infty}^2 \right) \cdot w  \ dz,
\end{aligned}
\end{equation}
Let $v_k(z) \coloneqq u_k(\delta_k z)$ as in Definition \ref{Definition: Bubble tree convergence of p harm map with one bubble}.
With a change of variables
\begin{equation}
\int_{B_{\eta}} \abs{\nabla u_k(x)}^2 \omega_{\eta,k}(x) \ dx
= \int_{B_{\frac{\eta}{\delta_k}}} \abs{\nabla v_k(z)}^2 \delta_k^2 \ \omega_{\eta,k}(\delta_k z) \ dz
\end{equation}
motivating the definition of
\begin{equation}
\widehat\omega_{\eta,k}(z) \coloneqq \delta_k^2 \ \omega_{\eta,k}(\delta_k z), \qquad z \in B_{\frac{\eta}{\delta_k}}.
\end{equation}
One has the pointwise limit
\begin{equation}
\widehat \omega_{\eta,k}(z) = \delta_k^2 \ \omega_{\eta,k}(\delta_k z) \rightarrow \widehat \omega_{\eta,\infty} (z) \coloneqq \begin{cases}
\frac{1}{\eta^\beta} \frac{1}{\abs{z}^{2+\beta}} , &\text{ if } z\in \C \setminus B_{1/\eta} \\
\frac{1}{\eta^2} \frac{(1+\eta^2)^2}{(1+\abs{z}^2)^2}, &\text{ if } z\in B_{1/\eta}
\end{cases},
\qquad \text{ as } k \rightarrow \infty.
\end{equation}
We introduce 
\begin{equation}
\widehat{\mathcal L}_{\eta,\infty}: V_{v_{\infty}} \rightarrow V_{v_{\infty}}; 
\qquad\widehat{\mathcal L}_{\eta,\infty}(w) \coloneqq 2\ P_{v_\infty} \left(\widehat\omega_{\eta,\infty}^{-1}(- \Delta w  -w \abs{\nabla v_\infty}^2) \right),
\end{equation}
such that 
\begin{equation}
Q_{v_\infty}(w)= \langle \widehat{\mathcal L}_{\eta,\infty} w, w \rangle_{\widehat\omega_{\eta,\infty}},
\end{equation}
where we used 
\begin{equation}
\langle w,v \rangle_{\widehat\omega_{\eta,\infty}} \coloneqq \int_{\C} w \cdot v \ \widehat \omega_{\eta,\infty} \ dz.
\end{equation}
In the following let 
\begin{equation}
St: S^2 \rightarrow \C
\end{equation}
denote the stereographic projection.
We introduce the notation
\begin{equation}
\widetilde v_\infty \coloneqq v_\infty \circ St, \quad
\widetilde w\coloneqq w \circ St, \quad
\widetilde\omega_{\eta,\infty} \coloneqq [\widehat\omega_{\eta,\infty}(y)(1+\abs{y}^2)^2] \circ St.
\end{equation}
With a change of variables
\begin{equation}
\begin{aligned}
Q_{v_\infty}(w)
&= 2\int_\C \left(\widehat\omega_{\eta,\infty}^{-1}(- \Delta w  -w \abs{\nabla v_\infty}^2) \right) \cdot w \ \widehat\omega_{\eta,\infty} \ dvol_{\Sigma} \\
&= 2\int_{S^2} \left(\widetilde\omega_{\eta,\infty}^{-1}(- \Delta\widetilde w  -\widetilde w \abs{\nabla \widetilde v_\infty}^2) \right) \cdot\widetilde w \ \widetilde\omega_{\eta,\infty} \ dvol_{\Sigma} = Q_{\widetilde v_\infty}(\widetilde w),
\end{aligned}
\end{equation}
We introduce
\begin{equation}
\widetilde{\mathcal L}_{\eta,\infty}: V_{\widetilde v_{\infty}} \rightarrow V_{\widetilde v_{\infty}}; 
\qquad\widetilde{\mathcal L}_{\eta,\infty}(\widetilde w) \coloneqq2\ P_{\widetilde v_\infty} \left(\widetilde\omega_{\eta,\infty}^{-1}(- \Delta \widetilde w  -\widetilde w \abs{\nabla \widetilde v_\infty}^2) \right)
\end{equation}
Let
\begin{equation}
U_{u_\infty} = \left\{ w \in L^{2}_{\omega_{\eta,\infty}}(\Sigma; \R^{n+1}) \forwhich w(x) \in T_{u_\infty(x)} S^n, \quad \text{for a.e. } x\in\Sigma \right\},
\end{equation}
and 
\begin{equation}
U_{\widetilde v_\infty} = \left\{ w \in L^{2}_{\widetilde\omega_{\eta,\infty}}(S^2; \R^{n+1}) \forwhich w(x) \in T_{ \widetilde v_\infty(x)} S^n, \quad \text{for a.e. } x\in S^2 \right\}.
\end{equation}
In Lemma IV.5 of \cite{DGR22} the following result was shown:

\begin{lemma}
\begin{enumerate}
\item The separable Hilbert space $(U_{u_\infty}, \langle\cdot,\cdot\rangle_{\omega_{\eta,\infty}})$ has a Hilbert basis consisting of eigenfunctions of $\mathcal L_{\eta,\infty}$.
\item The separable Hilbert space $(U_{\widetilde v_\infty}, \langle\cdot,\cdot\rangle_{\widetilde \omega_{\eta,\infty}})$ has a Hilbert basis consisting of eigenfunctions of $\widetilde{\mathcal L}_{\eta,\infty}$.
\end{enumerate}
\end{lemma}

\noindent
We continue by introducing the limiting eigenspaces 
\begin{equation}
\mathcal E_{\eta,\infty}(\lambda) \coloneqq \{ w \in V_{u_\infty} \ ; \ \mathcal L_{\eta,\infty}(w)=\lambda w \},
\end{equation}
and
\begin{equation}
\widehat{\mathcal E}_{\eta,\infty}(\lambda) \coloneqq \{ w \in V_{v_\infty} \ ; \ \widehat{\mathcal L}_{\eta,\infty}(w)=\lambda w \}.
\end{equation}
And their nonpositive contribution
\begin{equation}
\mathcal E_{\eta,\infty}^0\coloneqq \bigoplus_{\lambda \le 0} \mathcal E_{\eta,\infty}(\lambda), \qquad
\widehat{\mathcal E}_{\eta,\infty}^0\coloneqq \bigoplus_{\lambda \le 0} \widehat{\mathcal E}_{\eta,\infty}(\lambda).
\end{equation}
In \cite{DGR22} in (IV.38) and (IV.45) the following result was shown:
\begin{lemma}
\label{LEMMA: Ind plus Null eq dim of eig for infty lim func}
\begin{equation}
\begin{aligned}
&i)\ \dim\left(\mathcal E_{\eta,\infty}^0\right) 
\le \operatorname{Ind}(u_\infty) + \operatorname{Null}(u_\infty), \\
&ii)\ \dim\left(\widehat{\mathcal E}_{\eta,\infty}^0\right) 
\le \operatorname{Ind}(\widetilde v_\infty) + \operatorname{Null}(\widetilde v_\infty),
\end{aligned}
\end{equation}
\end{lemma}

We consider the unit sphere (finite dimensional as the ambient space is finite dimensional) given by
\begin{equation}
\mathcal S_{\eta,k}^0 \coloneqq \left\{ w \in \bigoplus_{\lambda \le 0} \mathcal E_{\eta,k}(\lambda) \ ; \ \langle w,w \rangle_{\omega_{\eta,k}}=1 \right\}.
\end{equation}

\begin{lemma}
\label{LEMMA: On the non zero of the lim functs}
For any $k \in \N$ let $w_k \in \mathcal S_{\eta,k}^0$.
Then there exist a subsequence such that
\begin{equation}
w_k \rightharpoondown w_\infty, \text{ weakly in } W^{1,2}(\Sigma) \cap W^{2,2}_{loc}(\Sigma\setminus\{q\}),
\end{equation}
\begin{equation}
w_{k}(\delta_k y) \rightharpoondown \sigma_\infty(y), \text{ weakly in } W^{2,2}_{loc}(\C)
\end{equation}
and 
\begin{equation}
\text{ either }w_\infty \ne 0, \quad \text{ or } \sigma_\infty \ne 0.
\end{equation}
\end{lemma}

\begin{proof}
We have $Q_{u_k}(w_k)\le 0$.
With Lemma \ref{LEMMA: Linfty bd on na u in the necks unif} and Lemma \ref{LEMMA: prop on the seq mu eta k} we can estimate 
\begin{equation}
\begin{aligned}
&\hspace{-15mm}\int_\Sigma \left( 1+\abs{\nabla u_k}^2\right)^{p_k/2-1}  \abs{\nabla u_k}^2 \abs{w_k}^2 \ dvol_\Sigma \\
&\le \norm{\left(1+\abs{\nabla u_k}^2 \right)^{\frac{p_k}{2}-1}}_{L^\infty(\Sigma)}  \norm{\frac{\abs{\nabla u_k}^2}{\omega_{\eta,k}}}_{L^{\infty}(\Sigma)} \int_{\Sigma} \abs{w_k}^2 \omega_{\eta,k} \ dvol_{\Sigma}\\ &\le C.
\end{aligned}
\end{equation}
This implies
\begin{equation}
\begin{aligned}
\int_{\Sigma} \abs{\nabla w_k}^2 \ dvol_{\Sigma}
&\le p_k \int_{\Sigma} \left( 1+\abs{\nabla u_k}^2\right)^{p_k/2-1} \abs{\nabla w_k}^2 \ dvol_{\Sigma} \\
& = \underbrace{Q_{u_k}(w_k) }_{\le 0}  \underbrace{-p_{k}(p_k-2) \int_{\Sigma } \left( 1+\abs{\nabla u_k}^2\right)^{p_k/2-2} \left(\nabla u_k \cdot \nabla w \right)^2  \ dvol_\Sigma}_{\le0} \\
&\hspace{20mm} + \underbrace{p_k \int_\Sigma \left( 1+\abs{\nabla u_k}^2\right)^{p_k/2-1}  \abs{\nabla u_k}^2 \abs{w_k}^2 \ dvol_\Sigma}_{\le C} \\
&\le C.
\end{aligned}
\end{equation}
Therefore we may assume up to passing to subsequences that
\begin{equation}
\label{EQ: HINUJNFIUWEinwoienui9812}
w_k \rightharpoondown w_\infty \quad \text{ in } W^{1,2}(\Sigma) \qquad \text{ and } \qquad
w_{k}(\delta_k y+x_k) \rightharpoondown \sigma_\infty(y) \quad \text{ in } W^{1,2}(\C).
\end{equation}
In the following we show
\\
\textbf{Claim 1:}
$\forall\eta>0:\exists C, k_0: \forall k\ge k_0:
\norm{\nabla^2 w_k}_{L^2(\Sigma\setminus B_{\eta})} \le C(\eta)$.
\\
\textbf{Proof of Claim 1:}
For $w\in V_{u_k}$ we consider the operator
\begin{equation}
\mathfrak E_{\eta,k}(w)^i \coloneqq- \partial_\alpha \left( A_{i,j}^{\alpha,\beta} \ \partial_\beta w^j \right),
\end{equation}
where
\begin{equation}
A_{i,j}^{\alpha,\beta} 
\coloneqq p_k (p_k-2)(1+\abs{\nabla u_k}^2)^{\frac{p_k}{2}-2}\ \partial_\alpha u_k^i\ \partial_\beta u_k^j
+ p_k (1+\abs{\nabla u_k}^2)^{\frac{p_k}{2}-1}\ \delta_{\alpha\beta}\ \delta_{ij}.
\end{equation}
There holds
\begin{equation}
\label{EQ: IJNIUqnufien12983JIN12dnunqdi918g}
\mathcal L_{\eta,k} (w)
= \omega_{\eta,k}^{-1}\ P_{u_k} \left[ \mathfrak E_{\eta,k}(w)-p_k(1+\abs{\nabla u_k}^2)^{\frac{p_k}{2}-1} \abs{\nabla u_k}^2 w \right].
\end{equation}
Next, we show that the operator $\mathfrak E_{\eta,k}$ is elliptic in the sense that the coefficients satisfy for large $k$ the Legendre-Hadamard condition
\begin{equation}
\label{EQ: wbwef1715JNDWqoif9n3r13}
A_{i,j}^{\alpha,\beta} a_\alpha a_\beta b^i b^j
\ge c \abs{a}^2 \abs{b}^2, \qquad \forall a \in \R^2, \forall b \in \R^{n+1},
\end{equation}
as in section 3.4.1 in \cite{GM13}.
We can bound
\begin{equation}
\begin{aligned}
&\hspace{-20mm}\abs{p_k (p_k-2)(1+\abs{\nabla u_k}^2)^{\frac{p_k}{2}-2}\ \partial_\alpha u_k^i\ \partial_\beta u_k^j\ a_\alpha a_\beta b^i b^j} \\
&\le 2(n+1) \ p_k \underbrace{(p_k-2)}_{\rightarrow 0} \underbrace{\frac{\abs{\nabla u_k}^2}{1+\abs{\nabla u_k}^2}}_{\le 1} \underbrace{\norm{\left(1+\abs{\nabla u_k}^2 \right)^{\frac{p_k}{2}-1}}_{L^\infty(\Sigma)}}_{\le C} \abs{a}^2 \abs{b}^2 \\
&\le C (p_k-2) \abs{a}^2 \abs{b}^2,
\end{aligned}
\end{equation}
where we used also Lemma \ref{LEMMA: Linfty bd on na u in the necks unif}.
Hence, for large  $k$ we may assume that 
\begin{equation}
\abs{p_k (p_k-2)(1+\abs{\nabla u_k}^2)^{\frac{p_k}{2}-2}\ \partial_\alpha u_k^i\ \partial_\beta u_k^j\ a_\alpha a_\beta b^i b^j} \le \abs{a}^2 \abs{b}^2.
\end{equation}
This allows to bound
\begin{equation}
\begin{aligned}
A_{i,j}^{\alpha,\beta} a_\alpha a_\beta b^i b^j
&\ge - \abs{a}^2 \abs{b}^2 + p_k (1+\abs{\nabla u_k}^2)^{\frac{p_k}{2}-1}\ \abs{a}^2 \abs{b}^2 \\
&\ge - \abs{a}^2 \abs{b}^2 + 2 \abs{a}^2 \abs{b}^2 \\
&= \abs{a}^2 \abs{b}^2.
\end{aligned}
\end{equation}
We have showed \eqref{EQ: wbwef1715JNDWqoif9n3r13} with constant $c=1$.
This proves that $\mathfrak E_{\eta,k}$ is an elliptic operator and the theory of elliptic systems as in section 4.3.1 of \cite{GM13} applies, i.e.
there exists some constant $C=C(\eta)>0$ which may depend on $\eta$ but not $k$ such that
\begin{equation}
\label{EQ: INUINiuwneg29e983rfn9}
\norm{\nabla^2 w_k}_{L^2(\Sigma \setminus B_\eta)}
\le C \left( \norm{w_k}_{W^{1,2}(\Sigma)} + \norm{\mathfrak E_{\eta,k} (w_k)}_{L^2(\Sigma \setminus B_{\frac{\eta}{2}})} \right).
\end{equation}
It remains to show that
\begin{equation}
\norm{\mathfrak E_{\eta,k} (w_k)}_{L^2(\Sigma \setminus B_{\frac{\eta}{2}})} \le C.
\end{equation}
To that end, we write with \eqref{EQ: IJNIUqnufien12983JIN12dnunqdi918g}
\begin{equation}
\begin{aligned}
\label{EQ: upINUFIPN9813rn1ior3n1io212e}
\mathfrak E_{\eta,k} (w_k) 
&= P_{u_k} \mathfrak E_{\eta,k} (w_k) + (id - P_{u_k}) \mathfrak E_{\eta,k} (w_k) \\
&= P_{u_k} \mathfrak E_{\eta,k} (w_k) + (id - P_{u_k}) \mathfrak E_{\eta,k} (w_k) \\
&= P_{u_k} \left[ \mathfrak E_{\eta,k}(w_k)-p_k(1+\abs{\nabla u_k}^2)^{\frac{p_k}{2}-1} \abs{\nabla u_k}^2 w_k \right] + P_{u_k} \left[p_k(1+\abs{\nabla u_k}^2)^{\frac{p_k}{2}-1} \abs{\nabla u_k}^2 w_k \right] \\
&\hspace{20mm} + \left\langle \mathfrak E_{\eta,k} (w_k), u_k \right\rangle u_k \\
&= \omega_{\eta,k}\ \mathcal L_{\eta,k} (w_k) + P_{u_k} \left[p_k(1+\abs{\nabla u_k}^2)^{\frac{p_k}{2}-1} \abs{\nabla u_k}^2 w_k \right] + \left\langle \mathfrak E_{\eta,k} (w_k), u_k \right\rangle u_k.
\end{aligned}
\end{equation}
As by assumption $w_k \in \mathcal S_{\eta,k}^0$ we can write 
\begin{equation}
w_k = \sum_{j=1}^{N_k} c_k^j\ \phi_k^j,
\qquad \text{ where } \sum_{j=1}^{N_k} (c_k^j)^2=1
\end{equation}
and $\phi_k^1,\ldots,\phi_k^{N_k}$ is an orthonormal basis of $\oplus_{\lambda \le 0} \mathcal E_{\eta,k}(\lambda)$.
Then
\begin{equation}
\mathcal L_{\eta,k} (w_k) 
= \sum_{j=1}^{N_k} c_k^j \ \mathcal L_{\eta,k} (\phi_k^j)
= \sum_{j=1}^{N_k} c_k^j \ \lambda_{k}^j\ \phi_k^j.
\end{equation}
Hence,
\begin{equation}
\begin{aligned}
\label{EQ: IBNIUNFIPUNUIFWIN21}
\norm{\omega_{\eta,k}\ \mathcal L_{\eta,k} (w_k)}_{L^2(\Sigma \setminus B_{\frac{\eta}{2}})}
&\le \norm{\omega_{\eta,k}}_{L^\infty(\Sigma \setminus B_{\frac{\eta}{2}})}^{1/2} \norm{\mathcal L_{\eta,k} (w_k)}_{L^2_{\omega_{\eta,k}}(\Sigma\setminus B_{\frac{\eta}{2}})} \\
&\le C \norm{\mathcal L_{\eta,k} (w_k)}_{L^2(\Sigma)} \\
&\le C \sum_{j=1}^{N_k} (c_k^j \ \lambda_{k}^j)^2 \\
&\le C (\inf \Lambda_{\eta,k})^2 \\
&\le C\  \mu_0^2,
\end{aligned}
\end{equation}
where we used Lemma \ref{LEMMA: prop on the seq mu eta k} and its notations as well as the fact that $\norm{\omega_{\eta,k}}_{L^\infty(\Sigma \setminus B_{\frac{\eta}{2}})}  \le C=C(\eta)$.
We also have
\begin{equation}
\begin{aligned}
\label{EQ: bfe238913bnip1nrBNIUdq}
&\hspace{-15mm}\norm{P_{u_k} \left[p_k(1+\abs{\nabla u_k}^2)^{\frac{p_k}{2}-1} \abs{\nabla u_k}^2 w_k \right]}_{L^2(\Sigma \setminus B_{\frac{\eta}{2}})} \\
&= p_k\norm{(1+\abs{\nabla u_k}^2)^{\frac{p_k}{2}-1} \abs{\nabla u_k}^2 w_k }_{L^2(\Sigma \setminus B_{\frac{\eta}{2}})} \\
&\le p_k \norm{(1+\abs{\nabla u_k}^2)^{\frac{p_k}{2}-1} \abs{\nabla u_k}}_{L^\infty(\Sigma \setminus B_{\frac{\eta}{2}})} \mu_{\eta,k}^{\frac{1}{2}} \ \underbrace{\norm{w_k}_{L^2_{\omega_{\eta,k}}(\Sigma \setminus B_{\frac{\eta}{2}})}}_{=1} \\
&\le C(\eta),
\end{aligned}
\end{equation}
where we used the strong convergence in \eqref{EQ: Cond of bubb conv ji0wr931jJGA} and also Lemma \ref{LEMMA: prop on the seq mu eta k} with its notations.
Now
\begin{equation}
\begin{aligned}
\label{EQ: inIFUPf293ngfgo4on2g}
\left\langle \mathfrak E_{\eta,k} (w_k), u_k \right\rangle
&= p_k\ (p_k-2) \divergence \left( \left( 1+\abs{\nabla u_k}^2\right)^{p_k/2-2} \left(\nabla u_k \cdot \nabla w_k \right)\ \nabla u_k \right) \cdot u_k \\
&\hspace{20mm} + p_k \divergence\left( \left( 1+\abs{\nabla u_k}^2\right)^{p_k/2-1} \nabla w_k \right) \cdot u_k \\
&= p_k\ (p_k-2) \Bigg[ \underbrace{\nabla\left( \frac{\left(\nabla u_k \cdot \nabla w_k \right)}{1+\abs{\nabla u_k}^2} \right) \left( 1+\abs{\nabla u_k}^2\right)^{p_k/2-1} \cdot \nabla u_k \cdot u_k}_{=0 \text{ (since }\partial_\alpha u \cdot u=0)} \\
&\hspace{20mm} + \frac{\left(\nabla u_k \cdot \nabla w_k \right)}{1+\abs{\nabla u_k}^2}\divergence \left( \left( 1+\abs{\nabla u_k}^2\right)^{p_k/2-1} \nabla u_k \right) \cdot u_k \Bigg] \\
&\hspace{40mm} + p_k \divergence\left( \left( 1+\abs{\nabla u_k}^2\right)^{p_k/2-1} \nabla w_k \right) \cdot u_k \\
&= -p_k\ (p_k-2) \frac{\left(\nabla u_k \cdot \nabla w_k \right)}{1+\abs{\nabla u_k}^2} \left( 1+\abs{\nabla u_k}^2\right)^{p_k/2-1} \abs{\nabla u_k}^2 \\
&\hspace{20mm} + p_k \divergence \left( \left( 1+\abs{\nabla u_k}^2\right)^{p_k/2-1} \nabla w_k \right) \cdot u_k,
\end{aligned}
\end{equation}
where we also used \eqref{EQ: Euler-Lagrange equations in div form for p harm}.
Furthermore, using the facts that
\begin{equation}
w_k \cdot u_k=0, \qquad
\nabla w_k \cdot u_k= - w_k \cdot \nabla u_k,
\end{equation}
we get
\begin{equation}
\begin{aligned}
\label{EQ: njIFIN0913rj1f3okJNIqdw}
&\hspace{-15mm}\divergence \left( \left( 1+\abs{\nabla u_k}^2\right)^{p_k/2-1} \nabla w_k \right) \cdot u_k \\
&= \divergence \left( \left( 1+\abs{\nabla u_k}^2\right)^{p_k/2-1} \nabla w_k \cdot u_k \right) - \left( 1+\abs{\nabla u_k}^2\right)^{p_k/2-1} \nabla w_k \cdot \nabla u_k \\
&= -\divergence \left( \left( 1+\abs{\nabla u_k}^2\right)^{p_k/2-1} \nabla u_k \cdot w_k \right) - \left( 1+\abs{\nabla u_k}^2\right)^{p_k/2-1} \nabla w_k \cdot \nabla u_k \\
&= -\divergence \left( \left( 1+\abs{\nabla u_k}^2\right)^{p_k/2-1} \nabla u_k \right) \cdot w_k - 2\left( 1+\abs{\nabla u_k}^2\right)^{p_k/2-1} \nabla w_k \cdot \nabla u_k \\
&= - \left( 1+\abs{\nabla u_k}^2\right)^{p_k/2-1} \abs{\nabla u_k}^2 u_k \cdot w_k - 2\left( 1+\abs{\nabla u_k}^2\right)^{p_k/2-1} \nabla w_k \cdot \nabla u_k,
\end{aligned}
\end{equation}
where we also used \eqref{EQ: Euler-Lagrange equations in div form for p harm}.
We can with \eqref{EQ: inIFUPf293ngfgo4on2g} and \eqref{EQ: njIFIN0913rj1f3okJNIqdw} now bound
\begin{equation}
\begin{aligned}
\label{EQ: ZUBFIUEZBNUI1893rhn1r12e}
&\hspace{-5mm}\norm{\left\langle \mathfrak E_{\eta,k} (w_k), u_k \right\rangle u_k}_{L^2(\Sigma\setminus B_{\frac{\eta}{2}})} \\
&= \norm{\left\langle \mathfrak E_{\eta,k} (w_k), u_k \right\rangle }_{L^2(\Sigma\setminus B_{\frac{\eta}{2}})} \\
&\le C \norm{(1+\abs{\nabla u_k}^2)^{\frac{p_k}{2}-1} \abs{\nabla u_k} }_{L^\infty(\Sigma \setminus B_{\frac{\eta}{2}})} \norm{\nabla w_k}_{L^2(\Sigma\setminus B_{\frac{\eta}{2}})} \\
&\hspace{20mm} + C \norm{(1+\abs{\nabla u_k}^2)^{\frac{p_k}{2}-1} \abs{\nabla u_k} }_{L^\infty(\Sigma \setminus B_{\frac{\eta}{2}})} \mu_{\eta,k}^{\frac{1}{2}} \norm{w_k}_{L^2_{\omega_{\eta,k}}(\Sigma\setminus B_{\frac{\eta}{2}})} \\
&\hspace{40mm} +C \norm{(1+\abs{\nabla u_k}^2)^{\frac{p_k}{2}-1} \abs{\nabla u_k} }_{L^\infty(\Sigma \setminus B_{\frac{\eta}{2}})} \norm{\nabla w_k}_{L^2(\Sigma\setminus B_{\frac{\eta}{2}})} \\
&\le C(\eta),
\end{aligned}
\end{equation}
where in the last line we used \eqref{EQ: HINUJNFIUWEinwoienui9812}, the strong convergence coming from \eqref{EQ: Cond of bubb conv ji0wr931jJGA} and Lemma \ref{LEMMA: prop on the seq mu eta k}.
Combining \eqref{EQ: INUINiuwneg29e983rfn9}, \eqref{EQ: upINUFIPN9813rn1ior3n1io212e}, \eqref{EQ: IBNIUNFIPUNUIFWIN21}, \eqref{EQ: bfe238913bnip1nrBNIUdq} and \eqref{EQ: ZUBFIUEZBNUI1893rhn1r12e} Claim 1 follows.
\\
Let now $\sigma_k(y) \coloneqq w_{k}(\delta_k y)$.
Proceeding similar as in the proof of Claim 1 we can also show
\\
\textbf{Claim 2:}
$\forall\eta>0:\exists C, k_0: \forall k\ge k_0:
\norm{\nabla^2 \sigma_k}_{L^2(B_{\frac{1}{\eta}})} \le C(\eta)$.
\\
With Claim 1 and Claim 2 we find that
\begin{equation}
\label{EQ: JINFU10930fi13fjm}
w_k \rightharpoondown w_\infty, \text{ weakly in } W^{2,2}_{loc}(\Sigma\setminus\{q\}),
\end{equation}
and
\begin{equation}
\label{EQ: UIN1902j1i0niqw1221mnq12D}
w_{k}(\delta_k y) \rightharpoondown \sigma_\infty(y), \text{ weakly in } W^{2,2}_{loc}(\C).
\end{equation}
It remains to show that either $w_\infty \ne 0$ or $\sigma_\infty \ne 0$.
For a contradiction assume that $w_\infty = 0$ and $\sigma_\infty = 0$.
Let $\chi \in C^\infty([0,\infty);[0,1])$ with
\begin{equation}
\chi = \begin{cases}
1, &\text{ on } [0,1] \\
0, &\text{ on } [2,\infty)
\end{cases}
\end{equation}
Introduce the notation
\begin{equation}
\check w_k \coloneqq
w_k\ \chi\left(2 \frac{\left|x\right|}{\eta}\right)\left(1-\chi\left(\eta \frac{\left|x\right|}{\delta_k}\right)\right)  \in W^{1,2}_0(A(\eta,\delta_k);\R^{n+1})\cap V_{u_k}.
\end{equation}
Because of \eqref{EQ: JINFU10930fi13fjm}, \eqref{EQ: UIN1902j1i0niqw1221mnq12D} and because $w_\infty = 0$ and $\sigma_\infty = 0$ we find that
\begin{equation}
\label{EQ: BPNBFEIU193rnu13nijqUN129812}
\lim_{k \rightarrow \infty} \norm{\nabla(w_k- \check w_k)}_{L^2(\Sigma)}=0.
\end{equation}
We have
\begin{equation}
\begin{aligned}
\label{EQ: iwebfnwuUIN8913rn1d1312}
\abs{Q_{u_k}(w_k)-Q_{u_k}(\check w_k)}
&\le p_k (p_k-2) \int_\Sigma \left( 1+\abs{\nabla u_k}^2\right)^{\frac{p_k}{2} -2} \abs{\left(\nabla u_k \cdot \nabla w_k \right)^2 - \left(\nabla u_k \cdot \nabla \check w_k \right)^2 } \ dvol_\Sigma \\
&+ p_k \int_\Sigma \left( 1+\abs{\nabla u_k}^2\right)^{\frac{p_k}{2} - 1} \left( \abs{\abs{\nabla w_k}^2-\abs{\nabla\check w_k}^2} + \abs{\nabla u_k}^2 \abs{\abs{w}^2 - \abs{\check w}^2} \right) \ dvol_\Sigma \\
&\eqqcolon I + II
\end{aligned}
\end{equation}
First, with Lemma \ref{LEMMA: Linfty bd on na u in the necks unif} and \eqref{EQ: BPNBFEIU193rnu13nijqUN129812}
\begin{equation}
\begin{aligned}
I &\le p_k (p_k-2) \underbrace{\norm{\left(1+\abs{\nabla u_k}^2 \right)^{\frac{p_k}{2}-1}}_{L^\infty(\Sigma)}}_{\le C} \int_\Sigma \underbrace{\frac{\abs{\nabla u_k}^2}{1+\abs{\nabla u_k}^2}}_{\le 1} \left(\abs{\nabla w_k}^2 + \abs{\nabla\check w_k}^2 \right) \ dvol_\Sigma \\
&\le C (p_k-2) \int_\Sigma \left(\abs{\nabla w_k}^2 + \abs{\nabla\check w_k}^2 \right) \ dvol_\Sigma \\
&\le C (p_k-2) \rightarrow 0, \qquad \text{ as } k \rightarrow \infty.
\end{aligned}
\end{equation}
Second, with Lemma \ref{LEMMA: Linfty bd on na u in the necks unif}, \eqref{EQ: JINFU10930fi13fjm}, \eqref{EQ: UIN1902j1i0niqw1221mnq12D}, \eqref{EQ: BPNBFEIU193rnu13nijqUN129812} and \eqref{EQ: Cond of bubb conv ji0wr931jJGA}
\begin{equation}
\begin{aligned}
II
&\le C \norm{\left(1+\abs{\nabla u_k}^2 \right)^{\frac{p_k}{2}-1}}_{L^\infty(\Sigma)}\int_\Sigma \left( \abs{\abs{\nabla w_k}^2-\abs{\nabla\check w_k}^2} + \abs{\nabla u_k}^2 \abs{\abs{w}^2 - \abs{\check w}^2} \right) \ dvol_\Sigma \\
&\le C \underbrace{\int_{\Sigma\setminus B_{\frac{\eta}{2}}} \left( \abs{\abs{\nabla w_k}^2-\abs{\nabla\check w_k}^2} + \mu_{\eta,k}\ \overbrace{\abs{\nabla u_k}^2}^{\le C(\eta)} \abs{\abs{w}^2 - \abs{\check w}^2} \right) \ dvol_\Sigma}_{\rightarrow 0, \text{ as }k \rightarrow \infty } \\
&\hspace{20mm} +C\underbrace{\int_{B_{\frac{2\delta_k}{\eta}}} \left( \abs{\abs{\nabla w_k}^2-\abs{\nabla\check w_k}^2} + \mu_{\eta,k}\ \overbrace{\abs{\nabla u_k}^2}^{\le C(\eta)} \abs{\abs{w}^2 - \abs{\check w}^2} \right) \ dvol_\Sigma}_{\rightarrow 0, \text{ as }k \rightarrow \infty }.
\end{aligned}
\end{equation}
Going back to \eqref{EQ: iwebfnwuUIN8913rn1d1312} we have shown
\begin{equation}
\lim_{k \rightarrow\infty} \abs{Q_{u_k}(w_k)-Q_{u_k}(\check w_k)}=0.
\end{equation}
The fact that $Q_{u_k}(w_k) \le 0$ implies 
\begin{equation}
\label{EQ: vuicojnthbonurihy}
\limsup_{k\rightarrow\infty} Q_{u_k}(\check w_k) \le 0.
\end{equation}
But now also with  \eqref{EQ: JINFU10930fi13fjm} and \eqref{EQ: UIN1902j1i0niqw1221mnq12D}
\begin{equation}
\begin{aligned}
\abs{1- \int_{\Sigma} \abs{\check w_k}^2 \omega_{\eta,k}\ dvol_{\Sigma}}
&= \abs{\int_{\Sigma} \abs{ w_k}^2 \omega_{\eta,k}\ dvol_{\Sigma} - \int_{\Sigma} \abs{\check w_k}^2 \omega_{\eta,k}\ dvol_{\Sigma} } \\
& \le \underbrace{\int_{\Sigma\setminus B_{\frac{\eta}{2}}} \abs{\abs{ w_k}^2 - \abs{\check w_k}^2} \overbrace{\omega_{\eta,k}}^{\le C(\eta)} \ dvol_{\Sigma}}_{\rightarrow0, \text{ as }k \rightarrow\infty} + \underbrace{\int_{B_{\frac{2\delta_k}{\eta}}} \abs{\abs{ w_k}^2 - \abs{\check w_k}^2} \overbrace{\omega_{\eta,k}}^{\le C(\eta)} \ dvol_{\Sigma}}_{\rightarrow0, \text{ as }k \rightarrow\infty},
\end{aligned}
\end{equation}
which implies
\begin{equation}
\lim_{k\rightarrow\infty}
\int_{\Sigma} \abs{\check w_k}^2 \omega_{\eta,k}\ dvol_{\Sigma}=1.
\end{equation}
Since $\check w_k \in W^{1,2}_0(A(\eta,\delta_k);\R^{n+1})\cap V_{u_k}$ we have thanks to Theorem \ref{THM: JDNW892fejeujfwef32} for some constant $\overline\kappa >0$ 
\begin{equation}
\liminf_{k\rightarrow\infty} Q_{u_k} (\check w_k)
\ge \overline\kappa \ \lim_{k\rightarrow\infty} \int_{\Sigma} \abs{\check w_k}^2 \omega_{\eta,k}\ dvol_{\Sigma}
= \overline\kappa >0.
\end{equation}
This is a contradiction to \eqref{EQ: vuicojnthbonurihy} and we have shown that either $w_\infty \ne 0$ or $\sigma_\infty \ne 0$.
\end{proof}

We can finally show

\begin{proof}[Proof (of Theorem \ref{THM: usc morse index sacks-uhlenbeck with one bubble})]
By Lemma \ref{LEMMA: Ind plus Null eq dim of eigenspaces} and Lemma \ref{LEMMA: Ind plus Null eq dim of eig for infty lim func} it suffices to show that for $k \in \N$ large and $\eta>0$ small
\begin{equation}
\operatorname{dim}\left( \bigoplus_{\lambda\le 0} \mathcal E_{\eta,k}(\lambda) \right) 
\le \dim\left(\mathcal E_{\eta,\infty}^0\right) + \dim\left(\widehat{\mathcal E}_{\eta,\infty}^0\right).
\end{equation}
Let $N\in\N$ be fixed. 
For $k \in\N$ let $\phi_k^1,\ldots,\phi_k^N$ be a free orthonormal family of $U_{u_k}$ of eigenfunctions of the operator $\mathcal L_{\eta,k}$ with according negative eigenvalues $\lambda_{k}^1, \ldots, \lambda_{k}^N \le 0$.
For a contradiction we assume that 
\begin{equation}
\label{EQ: jnioecnHI1829he1n1ed231}
N > \dim\left(\mathcal E_{\eta,\infty}^0\right) + \dim\left(\widehat{\mathcal E}_{\eta,\infty}^0\right).
\end{equation}
By Lemma \ref{LEMMA: On the non zero of the lim functs} we find that up to subsequences
\begin{equation}
\label{EQ: iuwenfiUNIUND9183hrn9231}
\phi_k^j \rightharpoondown \phi_\infty^j, \text{ weakly in } W^{2,2}_{loc}(\Sigma\setminus\{q\})
\end{equation}
and 
\begin{equation}
\sigma_k^j(z) \coloneqq \phi_k^j(\delta_k^j z) \rightharpoondown \sigma_\infty^j(z), \text{ weakly in } W^{2,2}_{loc}(\C).
\end{equation}
Let $r>0$ and $w\in W^{1,2}(\Sigma;\R^m)$ with $\supp (w)\subset \Sigma \setminus B_r(q)$. Consider
\begin{equation}
\label{EQ: inwejnJININF130fn}
\begin{aligned}
\langle \mathcal L_{\eta,k} \phi_k^j , w \rangle_{\omega_{\eta,k}} 
&= \underbrace{p_k\ (p_k-2) \int_{\Sigma \setminus B_r(q)} \left( 1+\abs{\nabla u_k}^2\right)^{p_k/2-2} \left(\nabla u_k \cdot \nabla \phi_k^j \right) \left(\nabla u_k \cdot P_{u_k} \nabla w \right)  \ dvol_\Sigma}_{\eqqcolon I_{\eta,k}} \\
&\hspace{10mm}+\underbrace{p_k \int_{\Sigma \setminus B_r(q)} \left( 1+\abs{\nabla u_k}^2\right)^{p_k/2-1} \nabla  \phi_k^j \cdot P_{u_k} \nabla w \ dvol_\Sigma}_{\eqqcolon II_{\eta,k}} \\
&\hspace{20mm}-\underbrace{p_k \int_{\Sigma \setminus B_r(q)} \left( 1+\abs{\nabla u_k}^2\right)^{p_k/2-1}  \abs{\nabla u_k}^2 \phi_k^j \cdot P_{u_k}w \ dvol_\Sigma}_{\eqqcolon III_{\eta,k}} \\
\end{aligned}
\end{equation}
First, with Lemma \ref{LEMMA: Linfty bd on na u in the necks unif}
\begin{equation}
\begin{aligned}
\abs{I_{\eta,k}} 
&\le p_k (p_k-2) \underbrace{\norm{\left(1+\abs{\nabla u_k}^2 \right)^{\frac{p_k}{2}-1}}_{L^\infty(\Sigma)}}_{\le C} \int_{\Sigma \setminus B_r(q)} \underbrace{\frac{\abs{\nabla u_k}^2}{1+\abs{\nabla u_k}^2}}_{\le 1} \abs{\nabla \phi_k^j} \ \abs{\nabla w}  \ dvol_\Sigma \\
&\le C (p_k-2) \underbrace{\norm{\nabla \phi_k^j}_{L^2(\Sigma \setminus B_r(q))}}_{\le C} \norm{\nabla w}_{L^2(\Sigma)} \\
& \le C (p_k-2) \rightarrow0, \qquad \text{ as } k \rightarrow \infty.
\end{aligned}
\end{equation}
Second, using \eqref{EQ: iuwenfiUNIUND9183hrn9231} and Corollary \ref{COROLLARY: Limit of na uk to pk minus 2 eq 1} we know that
\begin{equation}
\left( 1+\abs{\nabla u_k}^2\right)^{p_k/2-1} \nabla  \phi_k^j \rightharpoondown \nabla  \phi_\infty^j, \text{ weakl in } W^{1,2}_{loc}(\Sigma\setminus\{q\})
\end{equation} 
and hence also with \eqref{EQ: Cond of bubb conv ji0wr931jJGA}
\begin{equation}
II_{\eta,k} \rightarrow 2 \int_\Sigma \nabla  \phi_\infty^j \cdot P_{u_\infty} \nabla w \ dvol_\Sigma,  \qquad \text{ as } k \rightarrow \infty.
\end{equation}
Third, using \eqref{EQ: iuwenfiUNIUND9183hrn9231}, \eqref{EQ: Cond of bubb conv ji0wr931jJGA} and Corollary \ref{COROLLARY: Limit of na uk to pk minus 2 eq 1} we know that
\begin{equation}
\left( 1+\abs{\nabla u_k}^2\right)^{p_k/2-1} \abs{\nabla u_k}^2 \phi_k^j \rightharpoondown \abs{\nabla u_\infty}^2 \phi_\infty^j, \text{ weakl in } W^{2,2}_{loc}(\Sigma\setminus\{q\})
\end{equation} 
and hence 
\begin{equation}
III_{\eta,k} \rightarrow 2 \int_\Sigma \abs{\nabla u_\infty}^2 \phi_\infty^j \cdot P_{u_\infty} w \ dvol_\Sigma,  \qquad \text{ as } k \rightarrow \infty.
\end{equation}
Going back to \eqref{EQ: inwejnJININF130fn} we have shown that 
\begin{equation}
\langle \mathcal L_{\eta,k} \phi_k^j , w \rangle_{\omega_{\eta,k}} \rightarrow \langle \mathcal L_{\eta,\infty} \phi_\infty^j , w \rangle_{\omega_{\eta,\infty}} , \qquad \text{ as } k \rightarrow \infty.
\end{equation}
This means that
\begin{equation}
\mathcal L_{\eta,k} \phi_k^j \rightharpoondown  \mathcal L_{\eta,\infty} \phi_\infty^j, \text{ weakly in } W^{1,2}_{loc}(\Sigma\setminus\{q\}).
\end{equation}
This together with
\begin{equation}
\mathcal L_{\eta,k} \phi_k^j = \lambda_k^j \phi_k^j \rightharpoondown \lambda_\infty^j \phi_\infty^j, \text{ weakly in } W^{1,2}_{loc}(\Sigma\setminus\{q\})
\end{equation}
 gives
\begin{equation}
\mathcal L_{\eta,\infty} \phi_\infty^j = \lambda_\infty^j \phi_\infty^j \text{ in } \mathcal D^\prime(\Sigma\setminus\{q\}).
\end{equation}
Since $\phi_\infty^j\in W^{1,2}(\Sigma)$ we can deduce using Lemma \ref{LEMMA: On the point removability of Sob functions} that indeed
\begin{equation}
\mathcal L_{\eta,\infty} \phi_\infty^j = \lambda_\infty^j \phi_\infty^j \text{ in } \Sigma .
\end{equation}
Similar one shows that
\begin{equation}
\widehat {\mathcal L}_{\eta,\infty} \sigma_\infty^j = \lambda_\infty^j \sigma_\infty^j \text{ in } \Sigma.
\end{equation}
Now since by \eqref{EQ: jnioecnHI1829he1n1ed231} $N > \dim({\mathcal E}_{\eta,\infty}^0 \times \widehat{\mathcal E}_{\eta,\infty}^0)$ we have that the family $(\phi_\infty^j,\sigma_\infty^j)_{j=1\ldots N}$ is linearly dependent and we can find some $(c_\infty^1, \ldots, c_\infty^N)\neq 0$ such that
\begin{equation}
\label{EQ: wifnuJNIOFN103mnf223}
\sum_{j=1}^N c_\infty^j \phi_\infty^j=0 \qquad \text{ and } \qquad \sum_{j=1}^N c_\infty^j \sigma_\infty^j=0.
\end{equation}
Let 
\begin{equation}
w_k \coloneqq \frac{1}{\sum_{j=1}^N c_\infty^j} \sum_{j=1}^N c_\infty^j \phi_k^j.
\end{equation}
Then $w_k \in \mathcal S_{\eta,k}^0$ and by Lemma \ref{LEMMA: On the non zero of the lim functs} up to subsequences
\begin{equation}
w_k \rightharpoondown w_\infty, \text{ in } \dot W^{1,2}(\Sigma) \qquad \text{ and } \qquad
w_{k}(\delta_k y+x_k) \rightharpoondown \sigma_\infty(y), \text{ in } \dot W^{1,2}(\C)
\end{equation}
and either $w_\infty \ne 0$ or $\sigma_\infty \ne 0$.
But by \eqref{EQ: wifnuJNIOFN103mnf223} one has $(w_\infty,\sigma_\infty)=(0,0)$.
This is a contradiction.

\end{proof}

\newpage
\newpage
\ \vspace{15mm}
\appendix
\section{Appendix}
\subsection{Whitney Type Extension}

In the following we always use the notation $B_r=B_r(0) \subset \R^2$.
\begin{lemma}
\label{LEMMA: Whitney extension for ann in Lp}
Let $0< 2r<R$ and $\frac{3}{2}\le p \le 4$. Suppose $a\in W^{1,p}(B_R\setminus B_r)$. 
Then there exists an extension $\widetilde{a}\in W^{1,p}({\C})$ of $a$ with
\begin{equation}
\supp(\nabla\widetilde{a})\subset B_{2R}\setminus B_{r/2} ,
\end{equation}
such that
\begin{equation}
 \int_{\C}|\nabla \widetilde{a}|^p\ dx \le C \int_{B_R\setminus B_r}|\nabla a|^p\ dx ,
\end{equation}
where $C>0$ is independent of $p$, $r$ and $R$. We can also impose moreover
\begin{equation}
\int_{B_{2R}\setminus B_R}|\nabla \widetilde{a}|^p\ dx\le C \int_{B_R\setminus B_{R/2}}|\nabla a|^p\ dx ,
\end{equation}
and
\begin{equation}
\int_{B_{r}\setminus B_{r/2}}|\nabla \widetilde{a}|^p\ dx\le C \int_{B_{2r}\setminus B_{r}}|\nabla a|^p\ dx.
\end{equation}
Furthermore, we remark that the extension is independent of $p$, i.e. if $a \in W^{1,p_1}\cap W^{1,p_2}$ then the according extensions $\widetilde a_{p_1}$ and $\widetilde a_{p_2}$ are equal $\widetilde a=\widetilde a_{p_1}=\widetilde a_{p_2}$.
\end{lemma}

\begin{proof}
Let $\chi \in C^\infty({\R}_{\ge0})$ with $0\le \chi\le1$ such that $\chi\equiv 1$ on $[0,1]$ and $\chi\equiv 0$ on $[2,\infty)$. We denote for any $t>0$
\begin{equation}
\chi_t(x)\coloneqq\chi\left(\frac{|x|}{t}\right), \qquad x \in \C.
\end{equation}
For $|x|\ge r$ we define
\begin{equation}
\widehat{a}(x):=\chi_r(x)\ a(x)+(1-\chi_r(x))\ \overline{a}_r, \quad\text{ where }\quad\overline{a}_r:=\frac{1}{|B_{2r}\setminus B_r|}\int_{B_{2r}\setminus B_r} a(x)\ dx.
\end{equation}
Because of Poincar\`e's inequality there holds
\begin{equation}
\label{EQ: uinwfi1en89Nj1Q}
\int_{{\C}\setminus B_r}|\nabla \widehat{a}|^p\ dx
\le C  \int_{B_{2r}\setminus B_{r}}|\nabla a|^p\ dx.
\end{equation}
Here we can find a constant that does not depend on $p$ by complex interpolation, see Theorem 4.1.2. in \cite{BL12}.
Now, we extend $a$ inside $B_r$ by taking
\begin{equation}
\widetilde{a}(x)\coloneqq\widehat{a}\left(r^2\,\frac{x}{|x|^2}\right).
\end{equation}
With the change of variables $y=x\ r^2/\abs{x}^2$ we bound
\begin{equation}
\begin{aligned}
\int_{B_r} \abs{\nabla\widetilde a(x)}^p \ dx
&= \int_{B_r} \abs{\nabla \widehat a \left(r^2\,\frac{x}{|x|^2}\right)}^p \frac{r^{2p}}{\abs{x}^{2p}} \ dx \\
&= \int_{B_{2r}\setminus B_r} \abs{\nabla \widehat a \left(y\right)}^p \frac{\abs{y}^{2p-4}}{r^{2p-4}} \ dy \\
&\le C \int_{B_{2r}\setminus B_r} \abs{\nabla \widehat a \left(y\right)}^p \ dy \\
&\le C \int_{B_{2r}\setminus B_{r}}|\nabla a|^p\ dx,
\end{aligned}
\end{equation}
where in the last line we used \eqref{EQ: uinwfi1en89Nj1Q}.
We do something similar in order to extend $a$ outside $B_R$. For $|x|<R$ we define
\begin{equation}
\breve{a}(x)\coloneqq(1-\chi_R(x))\ a(x)+\chi_R(x)\ \overline{a}_R,\quad\text{ where }\quad\overline{a}_R:=\frac{1}{|B_{R}\setminus B_{R/2}|}\int_{B_{R}\setminus B_{R/2}} a(x)\ dx.
\end{equation}
Because of Poincar\`e's inequality there holds
\begin{equation}
\label{EQ: o2ifIHB183ndH1nQ}
\int_{B_R}|\nabla \breve{a}|^p\ dx
\le C  \int_{B_{R}\setminus B_{R/2}}|\nabla a|^p\ dx.
\end{equation}
Here we can find a constant that does not depend on $p$ by complex interpolation, see Theorem 4.1.2. in \cite{BL12}.
Now, we extend $a$ outside $B_R$ by taking
\begin{equation}
\widetilde{a}(x):=\breve{a}\left(R^2\,\frac{x}{|x|^2}\right).
\end{equation}
With the change of variables $y=x\ R^2/\abs{x}^2$ we bound
\begin{equation}
\begin{aligned}
\int_{\C\setminus B_R} \abs{\nabla\widetilde a(x)}^p \ dx
&= \int_{\C\setminus B_R} \abs{\nabla \breve a \left(R^2\,\frac{x}{|x|^2}\right)}^p \frac{R^{2p}}{\abs{x}^{2p}} \ dx \\
&= \int_{B_{R}\setminus B_{R/2}} \abs{\nabla \breve a \left(y\right)}^p \frac{\abs{y}^{2p-4}}{R^{2p-4}} \ dy \\
&\le C \int_{B_{R}\setminus B_{R/2}} \abs{\nabla \breve a \left(y\right)}^p \ dy \\
&\le C \int_{B_{R}\setminus B_{R/2}}|\nabla a|^p\ dx,
\end{aligned}
\end{equation}
where in the last line we used \eqref{EQ: o2ifIHB183ndH1nQ}.
\end{proof}

\subsection{Monotonicity Formula}
\begin{lemma}\label{LEMMA: Stationary p harm map equation}
Let $p>2$ and $u\in W^{1,p}(B_1;S^n)$.
Then $u$ is a stationary $p$-harmonic map if and only if
\begin{equation}\label{EQ: nuinf293uhh3fnu321fHBUOIS}
0= \frac{\partial}{\partial x^i}\left( (1+\abs{\nabla u}^2)^{\frac{p}{2}} \right) -p\ \divergence\left( (1+\abs{\nabla u}^2)^{\frac{p}{2} -1} \left\langle\frac{\partial u}{\partial x^i} , \nabla u \right\rangle \right),
\qquad \forall i =1,2.
\end{equation}
\end{lemma}

\begin{proof}
Let $\psi \in C^\infty_c(B_1, \R^2)$ and let $\Psi_s(x)=x +s\psi(x)$, for $s$ in a small interval around $0\in\R$ we will have that $\Psi_s\in B_1$ and that $\Psi_s$ is a diffeomorphism of $B_1$ to itself.
Introduce
\begin{equation}
u_s=u \circ \Psi_s .
\end{equation}
We compute the expression
\begin{equation}
\frac{d}{ds}\bigg|_{s=0}E_p(u_s).
\end{equation}
By the chain rule we compute
\begin{equation}\label{EQ: uh29hjfu2h45gnHZBNHUBD}
\nabla u_s(x) = \nabla u(\Psi_s(x)) + s \sum_{i=1}^2 \frac{\partial u}{\partial x^i}(\Psi_s(x)) \ \nabla \psi^i(x).
\end{equation}
Also by the chain rule one has
\begin{equation}\label{EQ: 29ufn3unfHBU23r12r}
\frac{d}{ds} E_p(u_s)
= \frac{p}{2} \int_{B_1} (1+\abs{\nabla u_s}^2)^{\frac{p}{2} -1} \left( \frac{d}{ds} \abs{\nabla u_s}^2 \right) \ dx.
\end{equation}
Now squaring the expression in \eqref{EQ: uh29hjfu2h45gnHZBNHUBD} we get
\begin{equation}\label{EQ: 2u9fu9023fjUHNUI2uf20}
\begin{aligned}
\abs{\nabla u_s}^2
&= \left|\nabla u \circ \Psi_s\right|^2 +2s \sum_{i,j=1}^m \left\langle\frac{\partial u}{\partial x^i} \circ \Psi_s, \frac{\partial u}{\partial x^j} \circ \Psi_s\right\rangle \frac{\partial \psi^i}{\partial x^j} \\
 &\hspace{10mm} +s^2 \sum_{i,j=1}^m \left\langle\frac{\partial u}{\partial x^i} \circ \Psi_s, \frac{\partial u}{\partial x^j} \circ \Psi_s\right\rangle \nabla \psi^i \cdot \nabla \psi^j.
\end{aligned}
\end{equation}
With a change of variables
\begin{equation}\label{EQ: u39hjgHBU80u2jr89}
\begin{aligned}
&\hspace{-30mm} \left( \frac{p}{2} \int_{B_1} (1+\abs{\nabla u_s}^2)^{\frac{p}{2}-1} \left( \frac{d}{ds} \abs{\nabla u \circ \Psi_s}^2 \right) \ dx \right)\bigg|_{s=0} \\
&= \frac{p}{2} \int_{B_1} (1+\abs{\nabla u}^2)^{\frac{p}{2} -1} \left( \frac{d}{ds}\bigg|_{s=0} \abs{\nabla u \circ \Psi_s}^2 \right) \ dx \\
&=  \int_{B_1}  \left( \frac{d}{ds}\bigg|_{s=0} (1+\abs{\nabla u\circ \Psi_s}^2)^{\frac{p}{2}} \right) \ dx \\
&=   \frac{d}{ds}\bigg|_{s=0} \int_{B_1} (1+\abs{\nabla u\circ \Psi_s}^2)^{\frac{p}{2}} \ dx \\
&= \frac{d}{ds}\bigg|_{s=0} \int_{B_1} (1+\abs{\nabla u}^2)^{\frac{p}{2}} \abs{\det \nabla \Psi_s^{-1}} dx \\
&= - \int_{B_1} (1+\abs{\nabla u}^2)^{\frac{p}{2}} \divergence\psi \ dx ,
\end{aligned}
\end{equation}
where in the last line we used the fact that $\frac{d}{ds}\big|_{s=0}\abs{\det \nabla \Psi_s^{-1}} = -\divergence\psi$.
Putting \eqref{EQ: 29ufn3unfHBU23r12r}, \eqref{EQ: 2u9fu9023fjUHNUI2uf20} and \eqref{EQ: u39hjgHBU80u2jr89} together we obtain
\begin{equation}\label{EQ: wuifen832zr7hbg82BUDIJA}
\begin{aligned}
\frac{d}{ds}\bigg|_{s=0} E_p(u_s) = p \int_{B_1}(1+\abs{\nabla u}^2)^{\frac{p}{2} -1} \sum_{i,j=1}^m \left\langle\frac{\partial u}{\partial x^i} , \frac{\partial u}{\partial x^j}\right\rangle \frac{\partial \psi^i}{\partial x^j} \ dx -  \int_{B_1} (1+\abs{\nabla u}^2)^{\frac{p}{2}} \divergence\psi \ dx
\end{aligned}
\end{equation}
For some fixed $i=1,2$ let $\psi= e_i \phi$ for some scalar $\phi \in C^\infty_c(B_1)$.
Integrating by parts
\begin{equation}
\begin{aligned}
&\hspace{-20mm} p \int_{B_1}(1+\abs{\nabla u}^2)^{\frac{p}{2} -1} \sum_{j=1}^m \left\langle\frac{\partial u}{\partial x^i} , \frac{\partial u}{\partial x^j}\right\rangle \frac{\partial \phi}{\partial x^j} \ dx \\
&=- p \int_{B_1} \divergence \left( (1+\abs{\nabla u}^2)^{\frac{p}{2} -1} \left\langle\frac{\partial u}{\partial x^i} , \nabla u \right\rangle \right)\phi \ dx
\end{aligned}
\end{equation}
Going back to \eqref{EQ: wuifen832zr7hbg82BUDIJA} one finds for $i=1,2$ that
\begin{equation}\label{EQ: bhuu89HNJqf82fqwolu}
\begin{aligned}
&\hspace{-12mm} \frac{d}{ds}\bigg|_{s=0} E_p\Big(u\circ (id+s \phi e_i)\Big) \\
&= \int_{B_1} \left[ \frac{\partial}{\partial x^i}\left( (1+\abs{\nabla u}^2)^{\frac{p}{2}} \right) -p\ \divergence\left( (1+\abs{\nabla u}^2)^{\frac{p}{2} -1} \left\langle\frac{\partial u}{\partial x^i} , \nabla u \right\rangle \right) \right] \phi\ dx
\end{aligned}
\end{equation}
From \eqref{EQ: bhuu89HNJqf82fqwolu} the claim follows.
\end{proof}

\begin{lemma}
A $p$-harmonic map $u$ is also a stationary $p$-harmonic map.
\end{lemma}

\begin{proof}
By Lemma \ref{LEMMA: Stationary p harm map equation} we only need to check \eqref{EQ: nuinf293uhh3fnu321fHBUOIS}.
By \cite{SU81} we know that $u$ is regular and hence we can compute formally for $i=1,2$
\begin{equation}
\begin{aligned}
&\hspace{-20mm}\frac{\partial}{\partial x^i}\left( (1+\abs{\nabla u}^2)^{\frac{p}{2}} \right) -p\ \divergence\left( (1+\abs{\nabla u}^2)^{\frac{p}{2} -1} \left\langle\frac{\partial u}{\partial x^i} , \nabla u \right\rangle \right) \\
&= \underbrace{\frac{p}{2}(1+\abs{\nabla u}^2)^{\frac{p}{2}-1} 2 \left\langle\nabla u, \frac{\partial}{\partial x^i}\nabla u \right\rangle - p \left\langle\nabla\frac{\partial u}{\partial x^i}, (1+\abs{\nabla u}^2)^{\frac{p}{2}-1}\nabla u \right\rangle}_{=0} \\
&- \left\langle\frac{\partial u}{\partial x^i}, \divergence\left( (1+\abs{\nabla u}^2)^{\frac{p}{2}-1}\nabla u \right) \right\rangle \\
&=- \left\langle\frac{\partial u}{\partial x^i},  (1+\abs{\nabla u}^2)^{\frac{p}{2}-1}\abs{\nabla u}^2 u \right\rangle \\
&=- (1+\abs{\nabla u}^2)^{\frac{p}{2}-1}\abs{\nabla u}^2 \underbrace{\left\langle\frac{\partial u}{\partial x^i}, u \right\rangle}_{=\frac{\partial}{\partial x^i} \abs{u}^2=0},
\end{aligned}
\end{equation}
where we used the $p$-harmonic map equation \eqref{EQ: Euler-Lagrange equations in div form for p harm}.
\end{proof}

\begin{lemma} 
[Monotonicity Formula]
\label{LEMMA: Monotonicity Formula for p harm maps}
Let $p>2$ and let $u \in W^{1,p}(B_1,S^n)$ be a $p$-harmonic map.
Then for any $\rho\in(0,1)$ there holds
\begin{equation}\label{EQ: ijwnfeJN320j2njafaq1}
\begin{aligned}
&\hspace{-15mm}\frac{p-2}{p\ \rho} \int_{B_\rho} (1+\abs{\nabla u}^2)^{\frac{p}{2}-1} \abs{\nabla u_k}^2 \ dx \\
&=\int_{\partial B_\rho} (1+\abs{\nabla u}^2)^{\frac{p}{2}-1} \left(\abs{\frac{\partial u}{\partial r}}^2 -\frac{1}{p} \abs{\nabla u}^2 \right) d\sigma \\
&\hspace{10mm}+ \frac{2}{\rho\ p} \int_{B_\rho} (1+\abs{\nabla u}^2)^{\frac{p}{2}-1} \ dx - \frac{1}{p} \int_{\partial B_\rho} (1+\abs{\nabla u}^2)^{\frac{p}{2}-1} \ d\sigma,
\end{aligned}
\end{equation}
where $d\sigma=r\ d\vartheta$ is the surface measure on $\partial B_\rho$.
\end{lemma}

\begin{proof}
Let $\eta \in C^\infty_c(B_1)$.
Testing \eqref{EQ: nuinf293uhh3fnu321fHBUOIS} with $\eta(x) x^i$ we obtain
\begin{equation}
\begin{aligned}
\int_{B_1} x^i \frac{\partial \eta}{\partial x^i} (1+\abs{\nabla u}^2)^{\frac{p}{2}}
-x^i p (1+\abs{\nabla u}^2)^{\frac{p}{2} -1} \left\langle\frac{\partial u}{\partial x^i} , \nabla u \cdot \nabla \eta \right\rangle \ dx \\
= -\int_{B_1} \eta \left( (1+\abs{\nabla u}^2)^{\frac{p}{2}} - p (1+\abs{\nabla u}^2)^{\frac{p}{2} -1} \abs{\frac{\partial u}{\partial x^i}}^2 \right) \ dx
\end{aligned}
\end{equation}
Summing in $i=1,2$ one has
\begin{equation}\label{EQ: ui2hfn3iuhw8JNJ}
\begin{aligned}
\int_{B_1} x \cdot \nabla \eta \ (1+\abs{\nabla u}^2)^{\frac{p}{2}}
-p (1+\abs{\nabla u}^2)^{\frac{p}{2} -1} \left\langle x \cdot \nabla u , \nabla u \cdot \nabla \eta \right\rangle \ dx \\
= -\int_{B_1} \eta \left( 2(1+\abs{\nabla u}^2)^{\frac{p}{2}} - p (1+\abs{\nabla u}^2)^{\frac{p}{2} -1} \abs{\nabla u}^2 \right) \ dx
\end{aligned}
\end{equation}
Let $\rho\in(0,1)$.
Let $\chi \in C^\infty(\R;[0,1])$ with $\chi \equiv 0$ on $(-\infty,0]$ and $\chi \equiv 1$ on $[1,\infty)$.
Consider the test function $\eta_k(x)= \chi(k(\rho-\abs{x}))$.
Passing to the limit $k\rightarrow\infty$ in \eqref{EQ: ui2hfn3iuhw8JNJ} we have 
\begin{equation}
\begin{aligned}
&\hspace{-20mm} -\rho \int_{\partial B_\rho} (1+\abs{\nabla u}^2)^{\frac{p}{2}} \ d\sigma
+ \frac{p}{\rho} \int_{\partial B_\rho} (1+\abs{\nabla u}^2)^{\frac{p}{2} -1} \abs{x \cdot \nabla u}^2 \ d\sigma \\
&= -\int_{B_\rho} 2(1+\abs{\nabla u}^2)^{\frac{p}{2}} - p (1+\abs{\nabla u}^2)^{\frac{p}{2} -1} \abs{\nabla u}^2 \ dx \\
&= -\int_{B_\rho} (1+\abs{\nabla u}^2)^{\frac{p}{2} -1} \left( 2+ \left(2- p\right) \abs{\nabla u}^2 \right) \ dx
\end{aligned}
\end{equation}
Finally, multiplying with $(\rho\ p)^{-1}$ one has
\begin{equation}
\begin{aligned}
&\hspace{-10mm} \frac{p	-2}{\rho\ p} \int_{B_\rho} (1+\abs{\nabla u}^2)^{\frac{p}{2} -1} \abs{\nabla u}^2 \ dx - \frac{2}{\rho\ p} \int_{B_\rho} (1+\abs{\nabla u}^2)^{\frac{p}{2} -1} \ dx \\
&=-\frac{1}{p} \int_{\partial B_\rho} (1+\abs{\nabla u}^2)^{\frac{p}{2}} \ d\sigma + \int_{\partial B_\rho} (1+\abs{\nabla u}^2)^{\frac{p}{2} -1} \abs{\frac{\partial u}{\partial r}}^2 \ d\sigma \\
&=-\frac{1}{p} \int_{\partial B_\rho} (1+\abs{\nabla u}^2)^{\frac{p}{2} -1} \ d\sigma + \int_{\partial B_\rho} (1+\abs{\nabla u}^2)^{\frac{p}{2} -1} \left(\abs{\frac{\partial u}{\partial r}}^2 -\frac{1}{p} \abs{\nabla u}^2 \right) d\sigma.
\end{aligned}
\end{equation}
\end{proof}

\begin{lemma}
\label{LEMMA: On the mono est for p harm}
Let $p>2$ and let $u \in W^{1,p}(B_1,S^n)$ be a $p$-harmonic map.
Then for any $0<r<R<1$ one has
\begin{equation}
(p-2) \log\left(\frac{R}{r} \right) \norm{\nabla u}_{L^2(B_r)}^2 
\le p\norm{(1+\abs{\nabla u }^2)^{\frac{p }{2}-1}}_{L^\infty(B_R \setminus B_r)} \left(2\norm{\nabla u }_{L^2(B_R \setminus B_r)}^2 + (R^2-r^2) \right).
\end{equation}
\end{lemma}

\begin{proof}
Integrating \eqref{EQ: ijwnfeJN320j2njafaq1} from $\rho=r$ to $\rho=R$ the monotonicity formula given in Lemma \ref{LEMMA: Monotonicity Formula for p harm maps} we have
\begin{equation}\label{EQ: iwbfniIZBI1389hn2}
\begin{aligned}
&\hspace{-2mm}\int_{r}^R \frac{p-2}{\rho\ p} \int_{B_\rho} (1+\abs{\nabla u}^2)^{\frac{p}{2}-1} \abs{\nabla u}^2 \ dx \ d\rho \\
&\hspace{6mm}= \int_{ B_R \setminus B_r} (1+\abs{\nabla u }^2)^{\frac{p }{2}-1} \left(\abs{\frac{\partial u }{\partial r}}^2 -\frac{1}{p } \abs{\nabla u }^2 \right) dx\\
&\hspace{12mm}+ \int_r^R \left[ \frac{2}{\rho\ p } \int_{B_\rho} (1+\abs{\nabla u }^2)^{\frac{p }{2}-1} \ dx - \frac{1}{p } \int_{\partial B_\rho} (1+\abs{\nabla u }^2)^{\frac{p }{2}-1} \ d\sigma \right] d\rho.
\end{aligned}
\end{equation}
On the one hand, 
\begin{equation}\label{EQ: iwvnUBN28hdHZ12ka}
\begin{aligned}
\int_r^R \frac{p -2}{\rho\ p} \int_{B_\rho} (1+\abs{\nabla u }^2)^{\frac{p }{2}-1} \abs{\nabla u }^2 \ dx \ d\rho
&\ge \int_r^R \frac{p -2}{\rho\ p} \int_{B_r} \abs{\nabla u }^2 \ dx \ d\rho \\
&= \frac{p -2}{p } \norm{\nabla u }_{L^2(B_r)}^2 \log\left(\frac{R}{r} \right). \\
\end{aligned}
\end{equation}
where the last inequality is valid for $\eta>0$ small and $k$ large by strong convergence. 
On the other hand,
\begin{equation}\label{EQ: iwfnIiN129enjdqwqdwS}
\int_{ B_R \setminus B_r} (1+\abs{\nabla u }^2)^{\frac{p }{2}-1} \left(\abs{\frac{\partial u }{\partial r}}^2 -\frac{1}{p } \abs{\nabla u }^2 \right) dx
\le 2 \norm{(1+\abs{\nabla u }^2)^{\frac{p }{2}-1}}_{L^\infty(B_R \setminus B_r)}  \norm{\nabla u }_{L^2(B_R \setminus B_r)}^2.
\end{equation}
and also 
\begin{equation}\label{EQ: iuwnfUINDW1029eej1jhw}
\begin{aligned}
&\hspace{-20mm}\int_r^R \left[ \frac{2}{\rho\ p } \int_{B_\rho} (1+\abs{\nabla u }^2)^{\frac{p }{2}-1} \ dx - \frac{1}{p } \int_{\partial B_\rho} (1+\abs{\nabla u }^2)^{\frac{p }{2}-1} \ d\sigma \right] d\rho\\
&\le \norm{(1+\abs{\nabla u }^2)^{\frac{p }{2}-1}}_{L^\infty(B_R \setminus B_r)}
\int_r^R \frac{1}{\rho} \abs{B_\rho}  + \frac{1}{2} \abs{\partial B_\rho} d\rho \\
&= \norm{(1+\abs{\nabla u }^2)^{\frac{p }{2}-1}}_{L^\infty(B_R \setminus B_r)} \int_r^R \rho \ d\rho \\
&\le \norm{(1+\abs{\nabla u }^2)^{\frac{p }{2}-1}}_{L^\infty(B_R \setminus B_r)} (R^2-r^2)
\end{aligned}
\end{equation}
Combining \eqref{EQ: iwbfniIZBI1389hn2}, \eqref{EQ: iwvnUBN28hdHZ12ka}, \eqref{EQ: iwfnIiN129enjdqwqdwS} and \eqref{EQ: iuwnfUINDW1029eej1jhw} the claim follows.
\end{proof}

\subsection{Hodge/Helmholtz-Weyl Decomposition}
\begin{lemma} [Hodge/Helmholtz-Weyl Decomposition]
\label{LEMMA: Hodge/Helmholtz-Weyl Decomposition}
Let $\Omega\subset \R^2$ be a simply connected bounded open set with $C^2$ boundary and let $1<p<\infty$.
Let $f\in (L^p(\Omega))^2$ be a vector field in $\Omega$. 
Then
\begin{equation}\label{EQ: onfmuw8103JN83nIdA}
f= \nabla A + \nabla^\perp B \qquad \text{ in }\Omega,
\end{equation}
where $A,B \in W^{1,p}(\Omega)$ with $ \nabla^\perp B \cdot \nu=0$ on $\partial\Omega$.
Furthermore, this decomposition is unique up to additive constants. 
\end{lemma}

\begin{proof}
By Theorem III.1.2 of \cite{G11} there exists some $A \in W^{1,p}(\Omega)$ and some $g \in (L^p(\Omega))^2$ with $\divergence(g)=0$ in $\Omega$ and $g\cdot \nu=0$ on $\partial\Omega$ such that
\begin{equation}
f=\nabla A + g \qquad \text{ in }\Omega.
\end{equation}
The vector field
\begin{equation}
\widetilde g = \begin{pmatrix}
g_2\\
-g_1
\end{pmatrix}
\end{equation}
is curl-free and hence by the Poincar\`e lemma there exists some potential $B \in W^{1,p}(\Omega)$ such that $\nabla B=\widetilde g$.
That is $\nabla^\perp B=g$ with $ \nabla^\perp B \cdot \nu=0$ on $\partial\Omega$.
Finally, assume that $\nabla A + \nabla^\perp B = \nabla A^\prime + \nabla^\perp B^\prime$ with with $ \nabla^\perp B \cdot \nu=\nabla^\perp B^\prime \cdot \nu=0$ on $\partial\Omega$.
Then integrating by parts
\begin{equation}
\int_\Omega \abs{\nabla (A-A^\prime)}^2 dx 
= \int_\Omega \nabla (A-A^\prime) \cdot \nabla^\perp (B-B^\prime) dx
= \int_\Omega \divergence \left[(A-A^\prime) \cdot \nabla^\perp (B-B^\prime)\right] dx=0,
\end{equation}
implying $\nabla A= \nabla A^\prime$ and $\nabla^\perp B= \nabla^\perp B^\prime$.
\end{proof}

\begin{lemma}
\label{LEMMA: Poincare Lemma stated in our setting}
Let $\Omega\subset \R^2$ be a simply connected bounded open set with $C^2$ boundary and let $1<p<\infty$.
Let $f,g \in (L^p(\Omega))^2$ with 
\begin{equation}
\curl(f)=0 \text{ in } \Omega\qquad \text{ and }\qquad \divergence(g)=0 \text{ in } \Omega
\end{equation}
Then there exists some potentials $\Phi, \Psi \in W^{1,p}(\Omega)$ such that
\begin{equation}
f=\nabla \Phi \qquad \text{ and } \qquad g=\nabla^\perp\Psi.
\end{equation}
\end{lemma}

\begin{proof}
The existence of $\Phi$ follows from applying the Poincar\`e lemma to $f$ and the existence of $\Psi$ follows from applying the Poincar\`e lemma to $\widetilde g =(g_2,-g_1)$.
\end{proof}

\begin{lemma}
\label{LEMMA: T assigns div free comp is bd unif}
Let $\Omega\subset \R^2$ be a simply connected bounded open set with $C^2$ boundary.
To $1<p<\infty$ consider the operator
\begin{equation}
{T_p}: (L^p(\Omega))^2 \rightarrow (L^p(\Omega))^2;
\qquad {T_p} (f)= \nabla^\perp B,
\end{equation}
which to a vector field $f\in (L^p(\Omega))^2$ assigns $\nabla^\perp B$, where 
$\nabla^\perp B$ is given by \eqref{EQ: onfmuw8103JN83nIdA}.
Then ${T_p}$ is a bounded linear operator and furthermore for any $1<p_1<p_2<\infty$ one has
\begin{equation}
\sup_{p \in [p_1,p_2]} \norm{{T_p}}<\infty.
\end{equation}
Henceforward we omit writing any dependence on $p$ and denote ${T}={T_p}$.
Finally, we also note that for any $u \in W^{1,p}(\Omega)$ one has $T(\nabla u)=0$.
\end{lemma}

\begin{proof}
The operator ${T_p}$ is well-defined by the result in Lemma \ref{LEMMA: Hodge/Helmholtz-Weyl Decomposition}. Clearly, ${T_p}$ is a projection on a Banach space and hence bounded. The uniform bound on the norm is a consequence of the Riesz-Thorin theorem. 
\end{proof}

\subsection{Weighted Poincar\`e Inequality}
\begin{lemma}[Weighted Poincar\`e Inequality]
\label{LEMMA: Lemma on neg in the necks Weighted Poincare}
\begin{equation}
\label{EQ: uwenfwjeufnwefbHBH}
\begin{gathered}
\forall \beta \in (0,2): \ \exists c_0>0: \ \exists \eta_0>0: \ \exists k_0>0: \forall \eta<\eta_0: \ \forall k \ge k_0: \ \forall f \in W^{1,2}_0(A(\eta,\delta_k)): \\
\int_{A(\eta,\delta_k)} c_0 \left( \frac{\de_k}{\et \abs{x}} \right)^\be \frac{\abs{f}^2}{\abs{x}^2} \ dx \le \int_{A(\eta,\delta_k)} \abs{\nabla f}^2 dx, \\
\int_{A(\eta,\delta_k)} c_0 \left( \frac{\abs{x}}{\et} \right)^\be \frac{\abs{f}^2}{\abs{x}^2} \ dx \le \int_{A(\eta,\delta_k)} \abs{\nabla f}^2  dx,  \\
\int_{A(\eta,\delta_k)} c_0 \frac{1}{\log^2\left(\frac{\eta^2}{\delta_k} \right)} \frac{\abs{f}^2}{\abs{x}^2} \ dx \le \int_{A(\eta,\delta_k)} \abs{\nabla f}^2  dx.
\end{gathered}
\end{equation}
\end{lemma}
This is shown in Lemma IV.1 of \cite{DR23}.

 \subsection{A Lemma on Capacity and the Lower Semicontinuity of the Morse Index}
\label{SUBSECTION: A Lemma on Capacity and the Lower Semicontinuity of the Morse Index}

\begin{lemma}
\label{LEMMA: On the point removability of Sob functions}
The set
\[
{\mathcal{V}} \;=\; \{\,f \in W^{1,2}(B_1)\;;\;\mathrm{supp}(f)\,\subset\subset\,B_1 \setminus \{0\}\}
\]
is dense in \(\,W^{1,2}(B_1 \setminus \{0\})\).  Consequently, \(W^{1,2}(B_1 \setminus \{0\}) = W^{1,2}(B_1) \).
\end{lemma}

\begin{proof}
Let $f\in W^{1,2}(B_1\setminus\{0\} )$.
To \(\epsilon>0\) define the function \(\psi_\epsilon\colon \mathbb{R}^n \to \mathbb{R}\) by
\begin{equation}
\psi_\epsilon(x) \;=\;
\begin{cases}
1, & \text{if } |x| \le \epsilon^2,\\[6pt]
\dfrac{\log\left(\frac{|x|}{\epsilon}\right)}{\log(\epsilon)}, & \text{if } \epsilon^2 < |x| < \epsilon,\\[6pt]
0, & \text{if } |x| \ge \epsilon.
\end{cases}
\end{equation}
To $M>0$ let 
\begin{equation}
T_Mf=\begin{cases}
M ,& \text{ if } f\ge M, \\
f ,& \text{ if } -M\le f\le M, \\
-M ,& \text{ if } f\le -M.
\end{cases}
\end{equation}
Introduce
\[
f_{M,\epsilon} \;=\; T_M f \,\bigl(1 - \psi_\epsilon\bigr).
\]
Notice that \(\mathrm{supp}\bigl(f_{M,\epsilon}\bigr)\subset\subset B_1 \setminus \{0\}\) and hence \(f_{M,\epsilon}\in {\mathcal{V}}\).
Let $\rho > 0$ be any small number.
Now we want to show that we can approximate $f$ by $f_{M,\epsilon}$ up to an error $\rho$ in the $W^{1,2}$-topology by choosing $M>0$ and $\epsilon>0$ suitably. 
Choose $M_0>0$ so large that 
\begin{equation}\label{TM0}
\|T_{M_0} f - f\|_{W^{1,2}(B_1)} < \frac{\rho}{2}.
\end{equation}
Now we note that $\| \psi_\epsilon \|_{L^\infty(B_1)}\rightarrow0$ as $\epsilon\searrow0$ and also
\begin{equation}
\|\nabla \psi_\epsilon \|_{L^2(B_1)}^2
= \frac{1}{|\log(\epsilon)|^2} \int_{B_{\epsilon}\setminus B_{\epsilon^2}} \frac{1}{|x|^2}\ dx
=\frac{2\pi}{|\log(\epsilon)|^2} \int_{\epsilon^2}^\epsilon \frac{1}{r}\ dr = \frac{2\pi}{|\log(\epsilon)|} \rightarrow 0,
\end{equation}
as $\epsilon\searrow0$.
Therefore we can choose $\epsilon_0>0$ so small that 
\begin{eqnarray}\label{fMep}
\|f_{M,\epsilon}\|_{W^{1,2}(B_1)}&\le& M_0 \| \psi_{\epsilon_0} \|_{L^2(B_1)} +
M_0 \|\nabla \psi_{\epsilon_0}\|_{L^{2}(B_1)} \nonumber\\&+& \|\nabla f \|_{L^2(B_1)} \|\psi_{\epsilon_0} \|_{L^\infty(B_1)} < \frac{\rho}{2}.
\end{eqnarray}
By combining \eqref{fMep} and \eqref{TM0} we get
\begin{equation}
\begin{aligned}
\|f_{{M_0},\epsilon_0} - f\|_{W^{1,2}(B_1)}
\;&\le\;
\|T_{M_0} f - f\|_{W^{1,2}(B_1)}
\;+\;
\|T_{M_0} f\,\psi_{\epsilon_0}\|_{W^{1,2}(B_1)} \\
&\le \frac{\rho}{2} + M_0 \| \psi_{\epsilon_0} \|_{L^2(B_1)} + M_0 \|\nabla \psi_{\epsilon_0} \|_{L^2(B_1)} + \|\nabla f \|_{L^2(B_1)} \|\psi_{\epsilon_0} \|_{L^\infty(B_1)}
< \rho.
\end{aligned}
\end{equation}
Since $\rho>0$ was arbitrary we have shown the claim.
\end{proof}

In the following we are always working in the setting and with the notations  introduced in Section \ref{SECTION: Preliminary definition and results}.
(Recall for instance $u_k, u_\infty, v_k, v_\infty, V_u, Q_u(\cdot), \Sigma, S^n$.)

\begin{proposition}[Lower Semicontinuity of Morse Index]\label{LSC}
For large $k$ there holds
\begin{equation}
\operatorname{Ind}_{E}(u_\infty) + \operatorname{Ind}_{E}(v_\infty)
\le \operatorname{Ind}{E_p}(u_k)\,.
\end{equation}
\end{proposition}

\begin{proof}
We set  $N_1:= \operatorname{Ind}(u_\infty)$ and $N_2 := \operatorname{Ind}(v_\infty)$.
Let $w^1,\dots,w^{N_1}$ be a basis of 
\begin{equation}
\{ w \in V_{u_\infty} ; Q_{u_\infty}(w)<0 \}.
\end{equation}
and let $\sigma^1,\dots,\sigma^{N_2}$
\begin{equation}
\{ \sigma \in V_{v_\infty} ; Q_{v_\infty}(\sigma)<0 \}.
\end{equation}
 There holds
\begin{equation}
\label{EQ: oiuwnfer9NIOQn23dq}
w^i \cdot u_\infty=0
\qquad \text{ and } \qquad
\sigma^i \cdot v_\infty=0,~~~\mbox{for all}~ i.
\end{equation}
 
\textbf{1.} By Lemma \ref{LEMMA: On the point removability of Sob functions} there exists a sequence $(f_l^i)_l\subset W^{1,2}(\Sigma)$ and radii $r_l^i>0$ such that
\begin{equation}
\label{EQ: uinwfe09jij0i2d3nJNdq}
\lim_{l\rightarrow\infty}\norm{f_l^i- w^i}_{W^{1,2}(\Sigma)}=0, \qquad \forall i=1,\dots,N_1
\end{equation}
and with $\supp(f^i_l)\subset \Sigma \setminus B_{r_l^i}$.
For $l \in \N$, $k \in \N$ and $i=1,\dots,N_1$, let us introduce
\begin{equation}
w_{l,k}^i \coloneqq 
f^i_l-\langle f^i_l , u_k \rangle u_k \qquad \text{in } \Sigma.
\end{equation}
One has $w_{l,k}^i \in V_{u_k}$.

{\bf Claim 1.} It holds:
\begin{equation}
\label{EQ: iuwnfenUINdiu23nf234f223ya}
\lim_{l\rightarrow\infty} \limsup_{k \rightarrow \infty} \norm{w^i_{l,k}-w^i}_{W^{1,2}(\Sigma)}=0.
\end{equation}
{\bf Proof of Claim 1.}\par
To that end, bound
\begin{equation}
\begin{aligned}
\norm{w^i_{l,k}-w^i}_{W^{1,2}(\Sigma)}
&\le \norm{f^i_{l}-w^i}_{W^{1,2}(\Sigma)}
+ \norm{\langle f^i_l , u_k \rangle u_k}_{W^{1,2}(\Sigma \setminus B_{r^i_l})} \\
&\le \norm{f^i_{l}-w^i}_{W^{1,2}(\Sigma)}
+ \norm{\langle f^i_l-w_i , u_k \rangle u_k}_{W^{1,2}(\Sigma)}
+ \norm{\langle w^i , u_k \rangle u_k}_{W^{1,2}(\Sigma \setminus B_{r^i_l})} \\
&\le 2\norm{f^i_{l}-w^i}_{W^{1,2}(\Sigma)}
+ \norm{\langle w^i , u_k \rangle u_k}_{W^{1,2}(\Sigma \setminus B_{r^i_l})}
\end{aligned}
\end{equation}
Let $\rho>0$.
Then for large $l$, we have by \eqref{EQ: uinwfe09jij0i2d3nJNdq} that
\begin{equation}
\norm{f^i_{l}-w^i}_{W^{1,2}(\Sigma)} < \frac{\rho}{2}.
\end{equation}
For such a given fixed $l$, we have by strong convergence of the sequence $(u_k)_k$  in $\Sigma \setminus B_{r^i_l})$. Hence from   \eqref{EQ: oiuwnfer9NIOQn23dq} it follows  that
\begin{equation}
\lim_{k\rightarrow\infty} \norm{\langle w^i , u_k \rangle u_k}_{W^{1,2}(\Sigma \setminus B_{r^i_l})}=0.
\end{equation}
Therefore for or large $l$ we have
\begin{equation}
\limsup_{k\rightarrow\infty}\norm{w^i_{l,k}-w^i}_{W^{1,2}(\Sigma)} < \rho.
\end{equation}
We conclude the {\bf proof of Claim 1.}\par 

We can now use \eqref{EQ: iuwnfenUINdiu23nf234f223ya} to get
\begin{equation}
\lim_{l\rightarrow\infty} \limsup_{k \rightarrow \infty}
\abs{Q_{u_k}(w^i_{l,k})-Q_{u_\infty}(w^i)}=0
\end{equation}
This implies that, for large $l$ and large $k$, we have 
\begin{equation}
\label{EQ: iunwef9u2u3n2193un12ejjNqd}
Q_{u_k}(w^i_{l,k}) <0.
\end{equation} \\
\textbf{2.} Let $(g_l^i)_{l}\subset W^{1,2}(\C)$ be a sequence and $R_l^i\nearrow\infty$ as $l\to +\infty$  be  such that   $\supp(g^i_l)\subset B_{R_l^i}$, and 
\begin{equation}
\label{EQ: uq9nfenUINIUDNui9qew}
\lim_{l\rightarrow\infty}\norm{g_l^i- \sigma^i}_{W^{1,2}(\C)}=0, \qquad \forall i=1,\dots,N_2.
\end{equation}
 For $l \in \N$, $k \in \N$ (with $\delta_k \le \frac{1}{R_l^i}$) and $i=1,\dots,N_2$, let us introduce
\begin{equation}
\sigma_{l,k}^i \coloneqq
\begin{cases}
g^i_l(\frac{\cdot}{\delta_k}) -\langle g^i_l(\frac{\cdot}{\delta_k}) , u_k \rangle u_k,
 & \abs{x} \le \delta_k R^i_l \le 1 \\
0, &\text{else}.
\end{cases}
\end{equation}
One has $\sigma_{l,k}^i \in V_{u_k}$.\par

{\bf Claim 2.}  We have:
\begin{equation}
\label{EQ: ONJJD0i13diamndJN12q156}
\lim_{l\rightarrow\infty} \limsup_{k \rightarrow \infty} \norm{\sigma^i_{l,k}(\delta_k \cdot) -\sigma^i}_{W^{1,2}(\C)}=0.
\end{equation}
{\bf Proof of Claim 2.}\par
It holds
\begin{equation}
\begin{aligned}
\norm{\sigma^i_{l,k}(\delta_k \cdot) -\sigma^i}_{W^{1,2}(\C)}
&\le \norm{g^i_{l}-\sigma^i}_{W^{1,2}(\C)}
+ \norm{\langle g^i_l , v_k \rangle v_k}_{W^{1,2}(\C)} \\
&\le \norm{g^i_{l}-	\sigma^i}_{W^{1,2}(\C)}
+ \norm{\langle g^i_l-\sigma_i , v_k \rangle v_k}_{W^{1,2}(\C)}
+ \norm{\langle \sigma^i , v_k \rangle v_k}_{W^{1,2}(B_{R^i_l})} \\
&\le 2\norm{g^i_{l}-	\sigma^i}_{W^{1,2}(\C)}
+ \norm{\langle \sigma^i , v_k \rangle v_k}_{W^{1,2}(B_{R^i_l})} 
\end{aligned}
\end{equation}
 
Let $\rho>0$,
then from  \eqref{EQ: uq9nfenUINIUDNui9qew}, for large $l$, it follows   that
\begin{equation}
\norm{g^i_{l}-\sigma^i}_{W^{1,2}(\C)} < \frac{\rho}{2}.
\end{equation}
For such a given fixed $l$ we have by strong convergence of the sequence $(v_k)_k$ and the fact \eqref{EQ: oiuwnfer9NIOQn23dq} that
\begin{equation}
\lim_{k\rightarrow\infty} \norm{\langle \sigma^i , v_k \rangle v_k}_{W^{1,2}(B_{R^i_l})}=0.
\end{equation}
Hence for large $l$ we have shown that
\begin{equation}
\limsup_{k\rightarrow\infty}\norm{\sigma^i_{l,k}(\delta_k\cdot)-\sigma^i}_{W^{1,2}(\C)} < \rho.
\end{equation}
This proves \eqref{EQ: ONJJD0i13diamndJN12q156} and we conclude the proof of {\bf Claim2.}\par

We can now use \eqref{EQ: ONJJD0i13diamndJN12q156} to get
\begin{equation}
\lim_{l\rightarrow\infty} \limsup_{k \rightarrow \infty}
\abs{Q_{v_k}(\sigma^i_{l,k})-Q_{v_\infty}(\sigma^i)}=0.
\end{equation}
This implies that for large $l$ and large $k$ we have
\begin{equation}
\label{EQ: jiNIJDWIQIIOEM893n1e9n}
Q_{v_k}(\sigma^i_{l,k}) <0.
\end{equation} \\
\textbf{3.}
Now we claim that for large $l$ and large $k$ the family
\begin{equation}
\mathcal B_{l,k} \coloneqq \{ w_{l,k}^1, \dots, w_{l,k}^{N_1}, \sigma_{l,k}^1, \dots, \sigma_{l,k}^{N_2} \} \subset V_{u_k}
\end{equation}
is linearly independent.
(In the following $G(B)$ denotes the determinant of the Gram matrix of a given basis $B$.)
As $(w^i)_{i=1,\dots, N_1}$ is a linear independent family we know that the determinant of the Gram matrix is non-zero, i.e. there is some $\kappa_1>0$ such that
\begin{equation}
\label{EQ: iu2n3fpfdwi0ej2013frn381209HZu1}
G(\{w_1,\dots,w_{N_1}\})= \det \Big[ \left( \langle w^i , w^j \rangle_{L^2(\Sigma)} \right)_{i,j} \Big]\ge \kappa_1 >0.
\end{equation}
Similar as $(\sigma^i)_{i=1,\dots, N_1}$ is a linear independent family we know that the determinant of the Gram matrix is non-zero, i.e. there is some $\kappa_2>0$ such that
\begin{equation}
\label{EQ: inHUIniuqefni23f9013r}
G(\{\sigma_1,\dots,\sigma_{N_2}\})=
\det \Big[ \left( \langle \sigma^i , \sigma^j \rangle_{L^2(\C)} \right)_{i,j} \Big]\ge \kappa_2 >0.
\end{equation}
Now note that as $\supp(w^i_{l,k})\subset\Sigma\setminus B_{r^i_l}$ and $\supp(\sigma^i_{l,k})\subset B_{R^i_l\delta_k}$ for large $l$ and large $k$
we will find that
\begin{equation}
\langle w^i_{l,k} , \sigma^j_{l,k} \rangle_{L^2(\Sigma)} = 0, \qquad \forall i=1,\dots,N_1,\forall j=1,\dots,N_2.
\end{equation}
Hence, if we compute the Gram matrix of $\mathcal B_{l,k}$ we have
\begin{equation}
\label{EQ: inuweunfiNIJ3209r1nirqfuz1ws0jw}
G(\mathcal B_{l,k})=
\det \Bigg[ \left( \langle w^i_{l,k} , w^j_{l,k} \rangle_{L^2(\Sigma)} \right)_{i,j} \Bigg]\
\det \Bigg[ \left( \langle \sigma^i_{l,k} , \sigma^j_{l,k} \rangle_{L^2(\Sigma)}  \right)_{i,j} \Bigg]
\end{equation}
By \eqref{EQ: iuwnfenUINdiu23nf234f223ya} and \eqref{EQ: ONJJD0i13diamndJN12q156} we know that 
\begin{equation}
\label{EQ: ijnwfeuinIUui123d1}
 \langle w^i_{l,k} , w^j_{l,k} \rangle_{L^2(\Sigma)} \rightarrow \langle w^i , w^j \rangle_{L^2(\Sigma)},
\qquad
\langle \sigma^i_{l,k} , \sigma^j_{l,k} \rangle_{L^2(\Sigma)} \rightarrow \langle \sigma^i , \sigma^j\rangle_{L^2(\C)},
\end{equation}
as $k \rightarrow \infty$ and $l \rightarrow \infty$.
Combining \eqref{EQ: ijnwfeuinIUui123d1} with \eqref{EQ: inuweunfiNIJ3209r1nirqfuz1ws0jw} and 	\eqref{EQ: iu2n3fpfdwi0ej2013frn381209HZu1}, \eqref{EQ: inHUIniuqefni23f9013r} we find  for large $l$ and large $k$ that
\begin{equation}
G(\mathcal B_{l,k}) \ge \frac{\kappa_1 \kappa_2}{2}>0.
\end{equation}
As the determinant of the Gram matrix of $\mathcal B_{l,k}$ is non-zero we deduce that the family $\mathcal B_{l,k}$ is linearly independent.
This with \eqref{EQ: iunwef9u2u3n2193un12ejjNqd} and \eqref{EQ: jiNIJDWIQIIOEM893n1e9n} gives
\begin{equation}
N_1+N_2= \dim(span(\mathcal B_{l,k})) \le \dim(\{ w \in V_{u_k} ; Q_{u_k}(w)<0 \}) = \operatorname{Ind}(u_k).
\end{equation}
This concludes the proof of Proposition \ref{LSC}.
\end{proof}

\end{document}